\documentclass[final]{ws-ccm}
\usepackage{amssymb,amsfonts,amsmath}
\usepackage{mathtools,mathrsfs}
\usepackage[usenames,dvipsnames]{color}
\usepackage{appendix}

\def\ds{\displaystyle}

\newcommand{\Om}{\Omega}
\newcommand{\e}{\varepsilon}

\newcommand{\wto}{\rightharpoonup}

\newcommand{\RR}{{\mathbb R}}

\newcommand{\NN}{{\mathbb N}}
\newcommand{\HH}{{\mathscr H}}
\newcommand{\LL}{{\mathscr L}}
\newcommand{\E}{{\mathcal E}}
\newcommand{\M}{{\mathscr M}}
\newcommand{\C}{{\mathscr C}}
\newcommand{\D}{{\mathscr D}}
\newcommand{\bdel}{{\boldsymbol \delta}}

\newcommand{\res}{\mathop{\hbox{\vrule height 7pt width .5pt depth 0pt
\vrule height .5pt width 6pt depth 0pt}}\nolimits}

\begin{document}

\markboth{J.-F. Babadjian \& V. Millot} {Unilateral gradient flow of the Ambrosio-Tortorelli functional by minimizing movements}

\title{UNILATERAL GRADIENT FLOW OF THE AMBROSIO-TORTORELLI FUNCTIONAL\\ BY MINIMIZING MOVEMENTS}

\author{JEAN-FRAN\c{C}OIS BABADJIAN}
\address{Universit\'e Pierre et Marie Curie - Paris 6\\
CNRS, UMR 7598 Laboratoire J.-L. Lions\\
Paris, France\\
\emph{\tt{jean-francois.babadjian@upmc.fr}}}

\author{VINCENT MILLOT}
\address{Universit\'e Paris Diderot - Paris 7\\
CNRS, UMR 7598 Laboratoire J.-L. Lions\\
Paris, France\\
\emph{\tt{millot@ljll.univ-paris-diderot.fr}}}

\maketitle

\begin{abstract}
{\bf Abstract.}  Motivated by models of fracture mechanics, this paper is devoted to the analysis of a unilateral gradient flow of the Ambrosio-Tortorelli functional, where unilaterality comes from an irreversibility constraint on the fracture density.  Solutions of such evolution are constructed by means of an implicit Euler scheme. An asymptotic analysis in the Mumford-Shah regime is then carried out. It shows the convergence towards a generalized heat equation outside a time increasing crack set.
 In the spirit of gradient flows in metric spaces, a notion of curve of maximal unilateral slope is also investigated, and analogies with the unilateral slope of the Mumford-Shah functional are also discussed. 
\end{abstract}


%%%%%%%%%%%%%%%%%%%%%%%%%%%%%%%%%%%%%%%%%%%%%%%%%%%%%%%%%%%%%%%%%%%%%%%%
%%%%%%%%%%%%%%%%%%%%%%%%%%%%%%%%%%%%%%%%%%%%%%%%%%%%%%%%%%%%%%%%%%%%%%%%
                                                              			  %%%%%%%%%%%%%%%%%%%%%%%%%%%%%%%%%%%%%%%%%%%%%%%%%    
\section{Introduction}                       			  %%%%%%%%%%%%%%%%%%%%%%%%%%%%%%%%%%%%%%%%%%%%%%%%%
                                                              			  %%%%%%%%%%%%%%%%%%%%%%%%%%%%%%%%%%%%%%%%%%%%%%%%%    
%%%%%%%%%%%%%%%%%%%%%%%%%%%%%%%%%%%%%%%%%%%%%%%%%%%%%%%%%%%%%%%%%%%%%%%%
%%%%%%%%%%%%%%%%%%%%%%%%%%%%%%%%%%%%%%%%%%%%%%%%%%%%%%%%%%%%%%%%%%%%%%%%

Many free discontinuity problems are variational in nature and involve two unknowns, a function $u$ and a discontinuity set $\Gamma$ across which $u$ may jump. The most famousÊ example is certainly  the minimization of the {\sc Mumford-Shah} (MS) functional introduced in \cite{MS} to approach image segmentation. It is defined by
$$\frac12\int_{\Om \setminus \Gamma} |\nabla u|^2\, dx + \HH^{N-1}(\Gamma)+\frac{\beta}{2} \int_\Om (u-g)^2\, dx\,,$$
where $\Om \subset \RR^N$ is a bounded Lipschitz open set, $\HH^{N-1}$ is the $(N-1)$-dimensional Hausdorff measure, $\beta>0$ is a fidelity (constant) factor, and $g \in L^\infty(\Om)$ stands for the grey level of the original image. In the resulting minimization process, we end up with a segmented image $u:\OmÊ\setminus \Gamma \to \RR$ and a set of contours $\Gamma \subset \Om$. 
To efficiently tackle this problem, a weak formulation in the space of {\it Special functions of Bounded Variation} has been suggested and solved in \cite{DGCL}, where the set $\Gamma$ is replaced by the jump set $J_u$ of $u$. The new energy is defined for $u \in SBV^2(\Om)$ by
\begin{equation}\label{MSSBV}
\frac12\int_\Om |\nabla u|^2\, dx + \HH^{N-1}(J_u)+\frac{\beta}{2} \int_\Om (u-g)^2\, dx\,,
\end{equation}
where $\nabla u$ is now intended to be the measure theoretic gradient of $u$.

\vskip3pt

A related model based on Mumford-Shah type functionals has been introduced by {\sc Francfort \&  Marigo} in \cite{FM} (see also \cite{BFM}) to describe quasi-static crack propagation inside elastic bodies. It is a variational model relying on three fundamental principles: {\it (i)} the fractured body must stay in elastic equilibrium at each time ({\it quasi-static hypothesis}); {\it (ii)} the crack can only grow ({\it irreversibility constraint}); {\it (iii)} an energy balance holds. In the anti-plane setting, the equilibrium and irreversibility principles lead us to look for constrained critical points (or local minimizers) at each time of the Mumford-Shah functional, where $u$ stands now for the scalar displacement while $\Gamma$ is the crack. Unfortunately, there is no canonical notion of local minimality since the family of all admissible cracks is not endowed with a natural topology. Seeking local minimizers of such energies has consequently become a great challenge, and a lot of works in that direction have considered  global minimizers instead, see \cite{DMT1,DMFT,FL}. In the discrete setting,  one looks at each time step for a pair $(u_i,\Gamma_i)$ minimizing 
$$(u,\Gamma) \mapsto  \E_*(u,\Gamma):=\frac12\int_{\Om \setminus \Gamma} |\nabla u|^2\, dx + \HH^{N-1}(\Gamma)\,,$$ 
among all cracks $\Gamma \supset \Gamma_{i-1}$ and all displacements $u:\Om \setminus \Gamma \to \RR$ satisfying an updated 
boundary condition, where $\Gamma_{i-1}$ is the crack found at the previous time step. A first attempt to local minimization has been carried out in \cite{DMT2} where a variant of this model is considered. At each time step the $L^2(\Om)$-distance to the previous displacement is penalized. More precisely, denoting by $u_{i-1}$ the displacement at the previous time step, one looks for minimizers of 
\begin{equation}\label{MSlambda}
(u,\Gamma) \mapsto\E_*(u,\Gamma)+ \lambda \|u - u_{i-1}\|^2_{L^2(\Om)}\,,
\end{equation}
on the same class of competitors than before, where $\lambda>0$ is a fixed parameter. We emphasize that this formulation only involves some kind of local minimality  with respect to the displacement. A notion of stability  which implies another local minimality criterion has been introduced in \cite{L}. It focuses on what the author  calls ``accessibility between two states". In the case of global minimization, when passing from one discrete time to the next, all states are accessible. From the point of view of \cite{DMT2}, a state $u$ is accessible from $u_{i-1}$ if and only if there is a certain gradient flow beginning at $u_{i-1}$ which approaches $u$ in the long-time limit. The main idea in \cite{L} is that a state $u$ is accessible from $u_{i-1}$ if and only if both states can be connected through a continuous path for which the total energy is never increased more than a fixed amount.

\vskip3pt

While static free discontinuity problems start to be well understood, many questions remain open concerning their evolutionary version. Apart from the quasi-static case, the closest evolution problem to statics consists in finding a steepest gradient descent of the energy, and thus in solving a gradient flow type equation. A major difficulty in this setting is to define a suitable notion of gradient since the functional is neither regular nor convex, and standard theories such as maximal monotone operators  \cite{Br} do not apply. However, using a time discretization, an implicit Euler scheme can always be defined. Letting  the time step tend to zero, the possible limits of such a discrete scheme are refered to as  {\sc De Giorgi}'s {\it minimizing movements} (see~\cite{A,DG}), and can be considered as solutions of the generalized gradient flow of the underlying functional. In the Mumford-Shah setting, this approach reduces to the minimization of the energy \eqref{MSlambda} exactly as in \cite{DMT2} with $\lambda$ replaced by $(2\delta)^{-1}$, $\delta>0$ being the time step. The minimizing movements of the Mumford-Shah functional have been first considered in \cite{AB}, and further developed in \cite{CD}. Motivated by the crack growth model as in  \cite{DMT2}, the authors apply the iterative scheme with respect to the variable $u$ while minimizing the energy with respect to $\Gamma$ under the constraint of  irreversibility. Showing compactness of the resulting discrete evolution as $\delta \to 0$, they obtain existence of  ``unilateral" minimizing movements of the Mumford-Shah energy (we add here the adjective unilateral to underline the irreversibility constraint on the evolution). In any space dimension, the limiting displacement $u(t)$ satisfies some kind of heat equation (in the weak sense), and an energy inequality with respect to the initial time holds. Assuming that  admissible cracks are compact and connected, they improve the result  in two dimensions showing that $u(t)$ solves a true heat equation in a fractured space-time domain, and that the energy inequality holds between arbitrary times. 

\vskip3pt

The Mumford-Shah functional enjoys good variational approximation properties by means of regular energies. Constructing $L^2(\Om)$-gradient flows for these regularized energies and taking the limit in the approximation parameter could be another way to derive a generalized gradient flow for $MS$. It was actually the path followed in \cite{Go} where a gradient flow equation for the one-dimensional Mumford-Shah functional is obtained as a limit of  ordinary differential equations derived from a non-local approximation of $MS$. Many other approximations are available, and the most famous one is certainly  the {\sc Ambrosio-Tortorelli} functional  defined for $(u,\rho) \in [H^1(\Om)]^2$ by
$$AT_\e(u,\rho):= \frac12\int_{\Om}(\eta_\e+\rho^2)|\nabla u|^2\,dx + \frac12\int_{\Om} \left(\e|\nabla\rho|^2 +\frac{1}{\e}(1-\rho)^2\right)dx\,.$$
The idea is to replace the discontinuity set $\Gamma$ by a (diffuse) phase field variable, denoted by $\rho:\Om \to [0,1]$, which is ``smooth'' and essentially $0$ in a $\e$-neighborhood of $\Gamma$. Such energies are of great importance for numerical simulations in imaging or brittle fracture, see \cite{BFM1,BFM}. 
From the mechanical point of view, it is interpreted as a non-local damage approximation of fracture models, where $\rho$ represents a damage density. 
The approximation result of \cite{AT,AT2} (see also \cite{Foc}) states that 
$AT_\e$ $\Gamma$-converges as $\varepsilon\to0$ to $MS$ (in the form~\eqref{MSSBV}) with respect to a suitable topology. For the static problem, it implies the convergence of $AT_\e$-minimizers towards $MS$-minimizers by standard results from $\Gamma$-convergence theory. However, the convergence of general critical points is {\it a priori} not guaranteed. Positive results in this direction have been obtained in \cite{FLS, Le}  for the one-dimensional case. The Ambrosio-Tortorelli approximation of quasi-static crack evolution is considered in \cite{G}, where the irreversibility constraint translates into the decrease of the phase field $t \mapsto \rho(t)$. The main result of~\cite{G} concerns the convergence of this regularized model towards the original one in \cite{FL}. Motivated by the formulation of a model of fracture dynamics, a hyperbolic evolution related to the Ambrosio-Tortorelli functional is also studied in \cite{LOS}, but the asymptotic behavior  of solutions as $\e\to 0$ is left open. A first step in that direction is made in \cite{DLL} where the analysis of a wave equation on a domain with growing cracks is performed. Concerning parabolic type evolutions, a standard gradient flow of the Ambrosio-Tortorelli functional is numerically investigated in~\cite{FP} for image segmentation and inpainting  purposes. 

\vskip3pt

The object of the present article is to study a {\it unilateral gradient flow} for the Ambrosio-Tortorelli functional taking into account the irreversibility constraint on the phase field variable. The idea is to construct minimizing movements starting from a discrete Euler scheme which is precisely an Ambrosio-Tortorelli regularization of the one studied in \cite{AB,CD}. As in~\cite{G}, the irreversibility of the process has to be encoded into the decrease of the phase field variable, and leads at each time step  to a constrained minimization problem. More precisely, 
given an initial data $(u_0,\rho_0)$, one may recursively  define pairs $(u_i,\rho_i)$ by minimizing at each  time $t_i\sim i \delta$,  
\begin{equation}\label{incremsch}
 (u,\rho) \mapsto  AT_\e(u,\rho) + \frac{1}{2\delta} \|u-u_{i-1}\|^2_{L^2(\Om)}\,,
\end{equation}
among all $u$ and $\rho \leq \rho_{i-1}$, where $(u_{i-1},\rho_{i-1})$ is a pair found at the previous time step. The objective is then to pass to the limit as the time step $\delta$ tends to $0$. A main difficulty is to deal with the asymptotics of the obstacle problems in the $\rho$ variable. It is known that such problems are not stable with respect to weak $H^1(\Om)$-convergence, and that ``strange terms" of capacitary type may appear \cite{CM,DML}. However, having uniform convergence of obstacles would be enough to rule out this situation. For that reason, instead of $AT_\e$, we consider a modified Ambrosio-Tortorelli functional with $p$-growth in $\nabla \rho$ with $p>N$. By the Sobolev Imbedding Theorem, with such a functional in hand, uniform convergence on the $\rho$ variable is now ensured.  We define for every $(u,\rho) \in H^1(\Om) \times W^{1,p}(\Om)$, 
$$
\E_\e(u,\rho):= \frac12 \int_{\Om}(\eta_\e+\rho^2)|\nabla u|^2\,dx + \int_{\Om} \left(\frac{\e^{p-1}}{p}|\nabla\rho|^p +\frac{\alpha}{p'\e}|1-\rho|^p\right)dx
\,,\quad p>N\,,
$$
where $\alpha>0$ is a suitable normalizing factor defined in \eqref{alpha}.  Note that an immediate adaptation of \cite{Foc} shows that $\E_\e$ is still an approximation of $MS$ in the sense of $\Gamma$-convergence. 

Considering the incremental scheme \eqref{incremsch} with $\E_\e$ instead of $AT_\e$, we prove that the discrete evolutions  converge as $\delta\to 0$ to  continous evolutions $t\mapsto (u_\e(t),\rho_\e(t))$ that we call {\it unilateral minimizing movements} (see Definition \ref{defminimov}). The first main result of the paper (Theorem \ref{thmGUMM}) gathers properties of unilateral minimizing movements. The limiting differential equation satisfied by $u_\e$ is 
\begin{equation}\label{equ}
\begin{cases}
\ds \partial_t u_\e - {\rm div}\big((\eta_\e+ \rho_\e^2)\nabla u_\e\big)=0  & \text{in }\Om \times (0,+\infty),\\[5pt]
\ds \frac{\partial u_\e}{\partial\nu}= 0 & \text{on }Ê\partial \Om \times (0,+\infty),
\end{cases}
\end{equation}
while the irreversibility and minimality conditions for $\rho_\e$ are 
\begin{equation}\label{equ2}
\begin{cases}
t \mapsto \rho_\e(t) \text{ is non-increasing}\,,\\[5pt]
\E_\e(u_\e(t),\rho_\e(t))\leq \E_\e(u_\e(t),\rho) \text{ for every $t\geq 0$ and $\rho\in W^{1,p}(\Om)$ such that $\rho\leq \rho_\e(t)$ in $\Om$} \,.
\end{cases}
\end{equation}
The system \eqref{equ}-\eqref{equ2} is by construction supplemented with the initial condition 
$$(u_\e(0), \rho_\e(0))= (u_0,\rho_0) \quad  \text{in }\Om\,.$$
In addition, we prove that the bulk and diffuse surface energies, defined by
$$
t \mapsto \frac12 \int_{\Om}(\eta_\e+\rho_\e(t)^2)|\nabla u_\e(t)|^2\,dx
$$
and
$$t \mapsto  \int_{\Om} \left(\frac{\e^{p-1}}{p}|\nabla\rho_\e(t)|^p +\frac{\alpha}{p'\e}|1-\rho_\e(t)|^p\right)dx$$
are respectively non-increasing and non-decreasing, a fact which is meaningful from the mechanical point of view. 
Moreover, the total energy is non-increasing, and it satisfies the following Lyapunov inequality: for a.e. $s\in [0,+\infty)$ and every $t\geq s$,
\begin{equation}\label{lyapu}
\E_\e(u_\e(t),\rho_\e(t))+\int_s^t \|\partial_t u_\e(r)\|^2_{L^2(\Om)}\,dr \leq \E_\e(u_\e(s),\rho_\e(s))\,.
\end{equation}

Note that the above inequality  is reminiscent of gradient flow type equations, and that it usually reduces to equality whenever the flow is regular enough. In any case, an energy equality would be equivalent to the absolute continuity in time of the total energy (see Proposition \ref{prop:identmodvit}). The reverse inequality might be obtained through an abstract infinite-dimensional chain-rule formula in the spirit of \cite{RoSa}. In our case, if we formally differentiate in time the total energy, we obtain
\begin{equation}\label{chainrule}
\frac{d}{dt}\,\E_\e(u_\e(t),\rho_\e(t)) = \big\langle \partial_u \E_\e(u_\e(t),\rho_\e(t)),\partial_t u_\e(t) \big\rangle +  \big\langle \partial_\rho \E_\e(u_\e(t),\rho_\e(t)),\partial_t \rho_\e(t) \big\rangle \, .
\end{equation}
From \eqref{equ2}  we could expect that 
\begin{equation}\label{Griff}
\langle \partial_\rho \E_\e(u_\e(t),\rho_\e(t)),\partial_t \rho_\e(t) \rangle=0\,,
\end{equation}
which would lead, together with \eqref{equ}, to the energy equality. Now observe that \eqref{Griff} is precisely the regularized version  of Griffith's criterion stating that a crack evolves if and only if the release of bulk energy is compensated by the increase of surface energy (see {\it e.g.} \cite[Section 2.1]{BFM}).  
Unfortunately, such a chain-rule is not available since we do not have enough control on the time regularity of $\rho_\e$. In the quasi-static case, one observes discontinuous time evolutions for the surface energy. Since the evolution law for $\rho_\e$ is quite similar to the quasi-static case (see~\cite{G}), we also expect here time discontinuities for the diffuse surface energy.  Adding a parabolic regularization in $\rho$, {\it i.e.}, a term of the form 
$$\frac{1}{\delta^{p-1}}\|\rho - \rho_{i-1}\|^p_{W^{1,p}(\Om)}$$
in \eqref{incremsch}, is a way to improve the time regularity of $\rho_\e$, and to get an energy equality. Unfortunately, it  also breaks the monotonicity of the surface energy, an undesirable fact in the modelling of the irreversibility. Moreover, this energy monotonicity is an essential ingredient in the analysis when $\varepsilon\to 0$.
\vskip3pt

As just mentioned, the natural continuation (and motivation) to the qualitative analysis of Ambrosio-Tortorelli minimizing movements is to understand the limiting behavior as $\e \to 0$,  and to compare the result with \cite{AB,CD}.  We stress that the general theory on $\Gamma$-convergence of gradient flows as presented in  \cite{SS,S} does not apply here since it requires a well defined gradient structure for the $\Gamma$-limit. A specific analysis thus seems to be necessary. In doing so, our second main result (Theorem \ref{BabMil}) states that $(u_\e,\rho_\e)$ tends to $(u,1)$ for some mapping $t\mapsto u(t)$ taking values in $SBV^2(\Om)$, and solving in the weak sense the equation 
\begin{equation}\label{1037}
\begin{cases}
\partial_t u - {\rm div}(\nabla u)=0 & \text{ in }Ê\Om\times(0,+\infty)\,,\\
\nabla u \cdot \nu =0 & \text{ on }\partial \Om\times(0,+\infty)\, ,\\
u(0)=u_0\,. 
\end{cases}
\end{equation}
In addition, using the monotonicity of the diffuse surface energy, we are able to pass to the limit in \eqref{equ2}. It yields the existence of a non-decreasing family of rectifiable subsets $\{\Gamma(t)\}_{t \geq 0}$ of $\Om$ such that 
$J_{u(t)}   \subset  \Gamma(t)$ for every $t \geq 0$, and for which the following energy inequality holds at any time: 
$$ 
\E_*\big(u(t),\Gamma(t)\big) + \int_0^t \|\partial_t u(s)\|_{L^2(\Om)}^2\, ds
\leq \frac12 \int_\Om |\nabla u_0|^2\, dx \,.$$
Comparing our result with \cite{CD}, we find that $u$ solves the same generalized heat equation with an improvement in the energy inequality where an increasing family of cracks appears. The optimality of this inequality and the convergence of energies remain open problems. Note that the (pointwise in time) convergence of the bulk energy usually follows by taking the solution as test function in the equation. In our case it asks the question wether $SBV^2(\Om)$ functions whose jump set is contained in $\Gamma(t)$ can be used in the variational formulation of \eqref{1037}, see \cite{DLL} and Subsection~\ref{subsc}. It would yield a weak form of the relation
$$\big((u^+(t)-u^-(t)\big) \frac{\partial u(t)}{\partial \nu}=0 \quad\text{ on }\Gamma(t)\,,$$
where $u^\pm(t)$ denote the one-sided traces of $u(t)$ on $\Gamma(t)$. This is indeed the missing equation to complement~\eqref{1037}, and it is intimately related to the finiteness of the unilateral slope of the Mumford-Shah functional (evaluated at $(u(t),\Gamma(t))$) defined in  \cite{DMT}.  

\vskip3pt

As already discussed, the nonlinear and nonconvex structure of  $MS$ prevents us to define a classical notion for its gradient flow. A possible approach, actually related to minimizing movements, is to make use of the general theory of gradient flows in metric spaces introduced in \cite{DGMT}. Here the notion of gradient is replaced by the concept of slope, and the standard gradient flow equation is recast in terms of {\it curves of maximal slope}Ê (see~\cite{AGS} for a detailed description of this subject). This idea was the starting point of \cite{DMT}, where the unilateral slope of $MS$ defined by
$$|\partial \E_*|(u,\Gamma):= \limsup_{v \to u\text{ in }L^2(\Omega)}Ê\frac{(\E_* (u,\Gamma) - \E_*(v,\Gamma \cup J_v))^+}{\|v-u\|_{L^2(\Om)}}\,,
$$
is investigated. By analogy we introduce the unilateral slope of the Ambrosio-Tortorelli functional 
$$|\partial \E_\e|(u,\rho):=\limsup_{v\to u\text{ in }L^2(\Omega)} \; \sup_{\hat \rho\leq \rho } \frac{\big(\E_\e(u,\rho)-\E_\e(v,\hat \rho)\big)^+}{\|v-u\|_{L^2(\Om)}} \,.$$
Then, {\it curves of maximal unilateral slope}  are essentially defined as curves for which  inequality \eqref{lyapu} holds  and the $L^2$-norm of the velocity coincides with the unilateral slope of the functional (see Definition~\ref{cms2}). 
In other words, these generalized evolutions are $L^2(\Om)$-steepest descents of $\E_\e$ with respect to $u$ in the direction of non-increasing $\rho$'s. In our third and last main result (Theorem \ref{cms}), we establish that any unilateral minimizing movement is a curve of maximal unilateral slope. As a matter of fact, any curve satisfying \eqref{equ}-\eqref{equ2}-\eqref{lyapu} has maximal unilateral slope. If one drops the energy inequality \eqref{lyapu}, system \eqref{equ}-\eqref{equ2}  admits infinitely many solutions which are not in general curves of maximal unilateral slope. The question wether or not curves of maximal unilateral slope are solutions of \eqref{equ}-\eqref{equ2}, is actually connected with the validity of the generalized chain-rule formula \eqref{chainrule}. Finally, we obtain some estimates in the spirit of \cite{DMT} for the limit as $\varepsilon\to 0$ of $|\partial \E_\e|$ along minimizing movements. However, a complete asymptotic analysis of $|\partial \E_\e|$ remains an open problem.

\vskip5pt

To conclude this introduction, let us briefly discuss some numerical aspects of our analysis. 
First, a practical drawback of the implicit Euler scheme defined in \eqref{incremsch} (with $\E_\e$ instead of $AT_\e$) is that the pair $(u_i,\rho_i)$ obtained at each time step might not be unique since $\E_\e$ is not strictly convex (although it is separately strictly convex). This lack of uniqueness may generate some troubles from the point of view of numerical approximations. For that reason, it is of interest to consider an alternate scheme as follows: given the initial data $(u_0,\rho_0)$, one recursively defines pairs $(u_i,\rho_i)$ at each time $t_i$ by
$$
\begin{cases}
\ds u_i:=\mathop{{\rm argmin}} \Big\{ \E_\e(u,\rho_{i-1}) +\frac{1}{2\delta} \|u-u_{i-1}\|^2_{L^2(\Om)} : 
u\in H^1(\Omega)\Big\}\,,\\[8pt]
\ds \rho_i:={\rm argmin} \Big\{ \E_\e(u_i,\rho) : \rho\in W^{1,p}(\Om)\,,\; \rho\leq \rho_{i-1} \text{  in }\Om\Big\} \,.
\end{cases}
$$
It turns out that this alternate minimization scheme is precisely the algorithm used in numerical experiments for quasi-static evolution in brittle fracture (see \cite{BFM1,BFM}). 
As the time step $\delta$ tends to zero, this scheme gives rise to the same time continuous model ({\it i.e.}, limiting evolutions satisfy \eqref{equ}-\eqref{equ2}-\eqref{lyapu}, see \cite{BM}).
Another difficulty for numerics is to deal with the  asymptotics when both $\varepsilon$ and $\delta$ tend to zero. Using the arguments developped in this paper together with \cite{CD} and \cite{G}, one should 
be able to prove a simultaneous convergence result similar to Theorem \ref{BabMil}, and that the limits commute. 

\vskip5pt

The paper is organized as follows. In Section \ref{sec2}, we provide the functional setting of the problem, and define in details the Ambrosio-Tortorelli and Mumford-Shah functionals.  
In Section \ref{sec3}, we introduce the implicit Euler scheme  generating unilateral minimizing movements. In Section \ref{sec4}, we establish an existence result for unilateral minimizing movements through a compactness result of discrete evolutions when the time step tends to zero. Then we study some qualitative properties where we establish the heat type equation,  the unilateral minimality of the phase field, and the energy inequality.  Section~\ref{sec5} is devoted to the asymptotic analysis as $\e \to 0$. 
Finally, Section \ref{sec6} is concerned with curves of maximal unilateral slope for the Ambrosio-Tortorelli functional. 

%%%%%%%%%%%%%%%%%%%%%%%%%%%%%%%%%%%%%%%%%%%%%%%%%%%%%%%%%%%%%%%%%%%%%%%%
%%%%%%%%%%%%%%%%%%%%%%%%%%%%%%%%%%%%%%%%%%%%%%%%%%%%%%%%%%%%%%%%%%%%%%%%
                                                              			  %%%%%%%%%%%%%%%%%%%%%%%%%%%%%%%%%%%%%%%%%%%%%%%%%    
\section{Preliminaries}\label{sec2}                        %%%%%%%%%%%%%%%%%%%%%%%%%%%%%%%%%%%%%%%%%%%%%%%%%
                                                                    		  %%%%%%%%%%%%%%%%%%%%%%%%%%%%%%%%%%%%%%%%%%%%%%%%    
%%%%%%%%%%%%%%%%%%%%%%%%%%%%%%%%%%%%%%%%%%%%%%%%%%%%%%%%%%%%%%%%%%%%%%%%
%%%%%%%%%%%%%%%%%%%%%%%%%%%%%%%%%%%%%%%%%%%%%%%%%%%%%%%%%%%%%%%%%%%%%%%%

\noindent {\bf {\it Notations.}} For  an open set  $U \subset \RR^N$, we denote by $\M(U;\RR^m)$  the space of all finite $\RR^m$-valued Radon measures on~$U$,  {\it i.e.}, the topological dual of the space $\C_0(U;\RR^m)$ of all $\RR^m$-valued continuous functions vanishing on $\partial U$. For $m=1$ we simply write $\M(U)$. The Lebesgue measure in $\RR^N$ is denoted by $\LL^N$, while $\HH^{N-1}$ stands for the $(N-1)$-dimensional Hausdorff measure. If $B_1$ is the open unit ball in $\RR^N$, we write $\omega_N:=\LL^N(B_1)$. We use the notations $\widetilde \subset$ and $\widetilde  =$ for inclusions or equalities between sets up to $\HH^{N-1}$-negligible sets.
For two real numbers $a$ and $b$, we denote by $a\wedge b$ and $a\vee b$ the minimum and maximum value between $a$ and $b$, respectively, and $a^+:=a \vee 0$.

\vskip10pt

\noindent{\bf {\it Absolutely continuous functions.}}
Throughout the paper, we consider the integration theory for Banach space valued functions in the sense of Bochner. All standard definitions and results we shall use can be found in \cite[Appendix]{Br} (see also \cite{DU}). We just recall here some basic facts.
If $X$ denotes a Banach space, we say that a mapping $u:[0,+\infty)\to X$ is absolutely continuous, and we write $u \in AC([0,+\infty);X)$, if there exists $m \in L^1(0,+\infty)$ such that
\begin{equation}\label{AC}
\|u(s)-u(t)\|_X \leq \int_s^tm(r)\, dr \quad \text{ for every }t\geq s \geq 0\,.
\end{equation}
If the space $X$ turns out to be reflexive, then any map $u \in AC([0,+\infty);X)$ is (strongly) differentiable almost everywhere. More precisely, for a.e. $t \in (0,+\infty)$, there exists $u'(t) \in X$ such that
$$\frac{u(t)-u(s)}{t-s} \to u'(t) \quad \text{ strongly in }X \text{ as }s \to t\,.$$
Moreover  
$u' \in L^1(0,+\infty;X)$, $u'$ coincides with distributional derivative of $u$, and the Fundamental Theorem of Calculus holds, {\it i.e.},
$$u(t)-u(s)=\int_s^t u'(r)\, dr \quad\text{ for every } t\geq s\geq 0\,.$$
If further the function $m$ in \eqref{AC}  belongs to $L^2(0,+\infty)$, then we write $u \in AC^2([0,+\infty);X)$, and in that case we have $u' \in L^2(0,+\infty;X)$.
\vskip10pt

\noindent {\bf {\it Special functions of bounded variation.}}
For  an open set  $U \subset \RR^N$, we denote by $BV(U)$ the space of functions of bounded variation, {\it i.e.}, 
the space of all functions $u \in L^1(U)$ whose distributional gradient $Du$ belongs to $\M(U;\RR^N)$. 
We shall also consider the subspace $SBV(U)$ of special functions of bounded variation made of functions $u \in BV(U)$ whose derivative $Du$  can be decomposed as
$$Du= \nabla u\, \LL^N + (u^+ - u^-) \nu_u\, \HH^{N-1} \res \, J_u\,.$$
In the previous expression, $\nabla u$ is the Radon-Nikod\'ym derivative of $Du$ with respect to $\LL^N$, and it is called 
approximate gradient of $u$. The Borel set $J_u$ is the (approximate) jump set of $u$. It is a countably $\HH^{N-1}$-rectifiable subset of $U$ oriented by the (normal) direction of jump $\nu_u:J_u\to \mathbb{S}^{N-1}$, and $u^\pm$ are the one-sided approximate limits of $u$ on $J_u$ according to $\nu_u$, see \cite{AFP}. We say that a measurable set $E$ has finite perimeter in~$U$ if $\chi_E \in BV(U)$, and we denote by $\partial^*E$ its reduced boundary. We also denote by $GSBV(U)$ the space of all measurable functions $u : U \to \RR$ such that $(-M \vee u) \wedge M \in SBV(U)$ for all $M>0$. Again, we refer to \cite{AFP} for an exhaustive treatment on the subject.  Finally we define the spaces
$$SBV^2(U):=\big\{ u \in SBV(U) \cap L^2(U) : \nabla u \in L^2(U;\RR^N) \text{ and } \HH^{N-1}(J_u)<\infty\big\} \,,$$
and 
$$GSBV^2(U):=\big\{ u \in GSBV(U) \cap L^2(U) : \nabla u \in L^2(U;\RR^N) \text{ and } \HH^{N-1}(J_u)<\infty\big\}\,.$$
Note that, according to the  chain rule formula for real valued $BV$-functions, we have the inclusion $SBV^2(U)\cap L^\infty(U)\subset GSBV^2(U)$ (see {\it e.g.} \cite[Theorem 3.99]{AFP}).
\vskip3pt

The following proposition will be very useful to derive a lower estimate for the Ambrosio-Tortorelli functional. It is a direct consequence of  the proof of \cite[Theorem 10.6]{B} (see \cite[Theorem 16]{BCS} for the original proof).

\begin{proposition}\label{lower-bound-AT}
Let $\Om \subset \RR^N$ be a bounded open set, let $\{u_{ n}\}_{n\in\NN} \subset H^1(\Om) \cap L^\infty(\Om)$ be such that $\sup_{n\in\NN}\|u_n\|_{L^\infty(\Om)} <\infty$,  and let $\{E_n\}_{n\in\NN}$ be a sequence of subsets of $\Om$ of finite perimeter in $\Omega$ such that $\sup_{n\in\NN} \mathscr H^{N-1}(\partial^* E_n\cap\Om) <\infty$. Assume that $u_n \to u$ strongly in $L^2(\Om)$, and  that $\mathscr L^N(E_n) \to 0$. 
Setting $\tilde u_n:=(1-\chi_{E_n}) u_n \in SBV^2(\Om)\cap L^\infty(\Omega)$,  and assuming in addition that 
$\sup_{n\in\NN} \| \nabla \tilde u_n\|_{L^2(\Om;\RR^N)}<\infty$,  
then $u \in SBV^2(\Om)\cap L^\infty(\Om)$ and
$$
\begin{cases}
\tilde u_n \to u \text{ strongly in }ÊL^2(\Om)\,,\\
\tilde u_n \wto u \text{ weakly* in }ÊL^\infty(\Om)\,,\\
\nabla \tilde u_n \wto \nabla u \text{ weakly in }ÊL^2(\Om;\RR^N)\,,\\
 \ds 2\, \mathscr H^{N-1}(J_u) \leq \liminf_{Ên \to \infty}Ê\mathscr H^{N-1}(\partial^* E_n\cap\Om)\,.
\end{cases}
$$
\end{proposition}

\vskip10pt

\noindent {\bf {\it The Ambrosio-Tortorelli \& Mumford-Shah functionals.}}
Throughout the paper, we assume that $\Om$ is a bounded open subset  of $\RR^N$ with at least Lipschitz boundary.  
We consider  $p>N$, $\beta>0$, and $g\in L^\infty(\Om)$ given.  
For $\e>0$ and $\eta_\e \in (0,1)$, we define the Ambrosio-Tortorelli functional $\E_\e:L^2(\Om)\times L^{p}(\Om)\to [0,+\infty]$ by
$$\E_\e(u,\rho):=
\begin{cases}
\displaystyle \begin{multlined}[9cm]
 \frac{1}{2}\int_{\Om}(\eta_\e+\rho^2)|\nabla u|^2\,dx + \int_{\Om} \left(\frac{\e^{p-1}}{p}|\nabla\rho|^p +\frac{\alpha}{p'\e}|1-\rho|^p\right)\,dx\\[-10pt]
+\frac{\beta}{2}\int_\Om (u-g)^2\,dx
\end{multlined}
& \text{if $(u,\rho)\in H^1(\Om)\times W^{1,p}(\Om)$}\,,\\[20pt]
+\infty & \text{otherwise}\,,
\end{cases}
$$
where $p':=p/(p-1)$ and $\alpha$ is the normalizing factor given by
\begin{equation}\label{alpha}
\alpha:= \left(\frac{p}{2}\right)^{p'}\,.
\end{equation}
The Mumford-Shah functional $\E : L^2(\Om) \to [0,+\infty]$ is in turn defined  by 
\begin{equation}\label{defMSfunct}
\E(u):=
\begin{cases}
\ds \frac{1}{2}\int_{\Om}|\nabla u|^2\,dx + \HH^{N-1}(J_u)+\frac{\beta}{2}\int_\Om (u-g)^2\,dx & \text{if }u \in GSBV^2(\Om)\,,\\[5pt]
+\infty & \text{otherwise}\,.
\end{cases}
\end{equation}
It is well known by now that the Ambrosio-Tortorelli functional approximates as $\varepsilon\to 0$ the Mumford-Shah functional in the sense of $\Gamma$-convergence, as stated in the following result, see \cite{AT,Foc}. Let us mention that Theorem~\ref{GconvAT} is not precisely a direct consequence of \cite{AT,Foc}. In \cite{AT}, the case $p=2$ is adressed, while \cite{Foc} deals with energies having the same $p$-growth in $\nabla u$ and $\nabla \rho$ (recall that $p>N\geq 2$). However, a careful inspection of the proof of \cite[Theorem 3.1]{Foc} shows that the  $\Gamma$-convergence result still holds for $\E_\e$.

\begin{theorem}\label{GconvAT}
Assume that $\eta_\e=o(\e)$. Then $\E_\e$ $\Gamma$-converges as $\e\to0$ (with respect to the strong $L^2(\Om) \times L^p(\Om)$-topology) to the functional $\E_0$ defined by
$$\E_0(u,\rho):=
\begin{cases}
\ds \E(u)  & \text{if }u \in GSBV^2(\Om) \text{ and }Ê\rho=1 \text{ in }Ê\Om\,,\\
+\infty & \text{otherwise}\,.
\end{cases}$$
\end{theorem}

%%%%%%%%%%%%%%%%%%%%%%%%%%%%%%%%%%%%%%%%%%%%%%%%%%%%%%%%%%%%%%%%%%%%%%%%
%%%%%%%%%%%%%%%%%%%%%%%%%%%%%%%%%%%%%%%%%%%%%%%%%%%%%%%%%%%%%%%%%%%%%%%%
                                                                                                                                                      %%%%%%%%%%%%%%%%%%%%%%%%%%%%%%%    
\section{Unilateral minimizing movements}\label{sec3}    %%%%%%%%%%%%%%%%%%%%%%%%%%%%%% 
                                                                                                                                                      %%%%%%%%%%%%%%%%%%%%%%%%%%%%%%%   
%%%%%%%%%%%%%%%%%%%%%%%%%%%%%%%%%%%%%%%%%%%%%%%%%%%%%%%%%%%%%%%%%%%%%%%%
%%%%%%%%%%%%%%%%%%%%%%%%%%%%%%%%%%%%%%%%%%%%%%%%%%%%%%%%%%%%%%%%%%%%%%%%

\subsection{The discrete evolution scheme}  
  
Throughout the paper, we shall say that a sequence of time steps $\boldsymbol{\delta}:=\{\delta^i\}_{i\in\mathbb{N}^*}$ is a partition of $[0,+\infty)$ if   
$$\delta^i>0\,,\quad \sup_{i\geq 1}\,\delta^i<+\infty\,,\quad\text{and}\quad \sum_{i\geq 1}\delta^i =+\infty\,. $$
To a partition  $\boldsymbol{\delta}$ we associate the sequence of discrete times $\{t^i\}_{i\in\mathbb{N}}$ given by $t^0:=0$,  $t^i:=\sum_{j=1}^i\delta^j$ for $i\geq 1$, and we define the {\it time step length} by 
$$|\boldsymbol\delta|:= \sup_{i\geq 1}\,\delta^i\,.$$

To an initial datum $u_0\in H^1(\Om)\cap L^\infty(\Om)$, we shall {\it always} associate (for simplicity) the initial state $\rho^\e_0$ determined by
\begin{equation}\label{rho0}
\rho_0^\e:=\mathop{{\rm argmin}}\limits_{\rho\in W^{1,p}(\Om)} \E_\e(u_0,\rho)\,.
\end{equation}
It is standard to check that the above minimization problem has a unique solution (by coercivity and strict convexity of the functional $\E_\e(u_0,\cdot)$), and it follows by minimality that $0\leq \rho^\e_0\leq 1$. Given a partition $\boldsymbol\delta$ of $[0,+\infty)$, we now introduce the discrete evolution Euler scheme starting from $(u_0,\rho_0^\e)$.

\vskip3pt

\noindent {\it \underline{\bf Global minimization:} \hskip5pt Set $(u^0,\rho^0):=(u_0,\rho_0^\e)$, and select recursively for all integer $i\geq 1$,  
\begin{equation}\label{minSchem1}
(u^i,\rho^i)\in {\rm argmin}\bigg\{ \E_\e(u,\rho) +\frac{1}{2\delta^i}\|u-u^{i-1}\|^2_{L^2(\Om)} :  (u,\rho)\in H^1(\Om)\times W^{1,p}(\Om)\,,\;
 \rho\leq \rho^{i-1}\text{  in }\Om\bigg\}\,.
\end{equation}}

The well-posedness of this scheme requires some care. 
Since the sublevel sets of $\E_\e$ are clearly relatively compact for the sequential weak $H^1(\Om)\times W^{1,p}(\Om)$-topology, one may apply 
the Direct Method of Calculus of Variations to solve \eqref{minSchem1}.  We only need to show that the constraint in \eqref{minSchem1} is closed, and that $\E_\e$ is lower semicontinuous with respect to weak convergence. 

\begin{lemma}\label{existsch1}
Let $\{(u_n,\rho_n)\}_{n\in\mathbb{N}}\subset H^1(\Om)\times W^{1,p}(\Om)$ be such that $(u_n,\rho_n)\rightharpoonup (u,\rho)$ weakly in $H^1(\Om)\times W^{1,p}(\Om)$. Then, 
\begin{equation}\label{liminfineqenerg}
\E_\e(u,\rho)\leq \liminf_{n\to\infty} \E_\e(u_n,\rho_n)\,.
\end{equation}
Moreover, if for each $n\in\mathbb{N}$, $\rho_n\leq \bar \rho$ in $\Om$ for some $\bar \rho\in W^{1,p}(\Om)$, then $\rho\leq \bar \rho$ in $\Om$.
 Finally, assuming that  
$ \E_\e(u_n,\rho_n)\to\E_\e(u,\rho)$ as $n\to\infty$, then $(u_n,\rho_n)\to (u,\rho)$ strongly in $H^1(\Om)\times W^{1,p}(\Om)$.
\end{lemma}

\begin{proof}
{\it Step 1.} The sequence $\{(u_n,\rho_n)\}$ being weakly convergent, it is bounded in $H^1(\Om)\times W^{1,p}(\Om)$. Therefore $\rho_n\to \rho$ in $\mathscr{C}^0(\overline\Om)$ by the Sobolev Imbedding Theorem.  
Hence $\rho \leq \bar \rho$ in $\Om$ whenever $\rho_n\leq \bar \rho$ in $\Om$ for every $n\in\mathbb{N}$. 
Then $\rho_n\nabla u_n\wto \rho\nabla u$ weakly in $L^2(\Om)$, and consequently, 
$$\int_\Om (\eta_\e+\rho^2)|\nabla u|^2\,dx\leq  \liminf_{n\to\infty} \int_\Om (\eta_\e+\rho_n^2)|\nabla u_n|^2\,dx\,.$$
Since all other terms in $\E_\e$ are clearly lower semicontinuous with respect to  the weak  convergence in $H^1(\Om)\times W^{1,p}(\Om)$, we have proved \eqref{liminfineqenerg}. 

\vskip5pt

\noindent{\it Step 2.} Let us now assume that $\E_\e(u_n,\rho_n)\to\E_\e(u,\rho)$. We first claim that 
\begin{equation}\label{convbulk}
\int_\Om (\eta_\e+\rho^2)|\nabla u|^2\,dx=\lim_{n\to\infty} \int_\Om (\eta_\e+\rho_n^2)|\nabla u_n|^2\,dx\,.
\end{equation}
Indeed, assume by contradiction that for a subsequence $\{n_j\}$ we have 
$$\int_\Om (\eta_\e+\rho^2)|\nabla u|^2\,dx< \liminf_{j\to\infty} \int_\Om (\eta_\e+\rho_{n_j}^2)|\nabla u_{n_j}|^2\,dx\,. $$
Using the fact that $u_n\to u$ strongly in $L^2(\Om)$, we deduce from Step 1 that 
\begin{align*}
 \lim_{j\to\infty} &\, \E_\e(u_{n_j},\rho_{n_j})\\
& \geq \liminf_{j\to\infty} \frac{1}{2}\int_\Om (\eta_\e+\rho_{n_j}^2)|\nabla u_{n_j}|^2\,dx+ 
\liminf_{j\to\infty} \int_{\Om} \left(\frac{\e^{p-1}}{p}|\nabla\rho_{n_j}|^p +\frac{\alpha}{p'\e}|1-\rho_{n_j}|^p\right)\,dx
+\frac{\beta}{2}\int_\Om (u-g)^2\,dx\\
& > \E_\e(u,\rho)\,,
\end{align*}
which is impossible. Therefore \eqref{convbulk} holds. Then, combining the convergence of $\E_\e(u_n,\rho_n)$ with \eqref{convbulk}, we deduce that 
$\|\rho_n\|_{W^{1,p}(\Om)}\to \|\rho\|_{W^{1,p}(\Om)}$, whence the strong $W^{1,p}(\Om)$-convergence of $\rho_n$.

It now remains to show that $u_n\to u$ strongly in $H^1(\Om)$. Using the uniform convergence of $\rho_n$ established in Step 1, we first estimate
$$\int_\Om |\rho^2-\rho_n^2||\nabla u_n|^2\,dx\leq \big(\sup_{ kÊ\in \NN} \|\nabla u_k\|_{L^2(\Om;\RR^N)}\big)\|\rho^2-\rho_n^2\|_{L^\infty(\Om)} \mathop{\longrightarrow}\limits_{n\to\infty} 0\,.$$
Then we infer from  \eqref{convbulk} that
$$\int_\Om(\eta_\e+\rho^2)|\nabla u_n|^2\,dx=  \int_\Om(\eta_\e+\rho_n^2)|\nabla u_n|^2\,dx+ \int_\Om (\rho^2-\rho_n^2)|\nabla u_n|^2\,dx \mathop{\longrightarrow}\limits_{n\to\infty}  
\int_\Om(\eta_\e+\rho^2)|\nabla u|^2\,dx\,.$$
Consequently $\|u_n\|_{H^1(\Om)}\to \|u\|_{H^1(\Om)}$, whence the strong $H^1(\Om)$-convergence of $u_n$. 
 \end{proof}

We state below a maximum principle on the iterates $\{(u^{i},\rho^i)\}_{i\in\NN}$ which easily follows from minimality and standard truncation arguments. 

\begin{lemma}\label{ptwbdmonotdisc}
For every $i\in\NN$, 
\begin{equation}\label{ptwbdmonotdisceq}
\|u^i\|_{L^\infty(\Om)}\leq  \max\{\|u_0\|_{L^\infty(\Om)}, \|g\|_{L^\infty(\Om)}\}\quad\text{and}\quad 0\leq \rho^{i+1}\leq \rho^{i}\leq 1\text{ in $\Om$}\,.
\end{equation}
\end{lemma}

\subsection{Generalized unilateral minimizing movements}

To a partition $\boldsymbol\delta$ of $[0,+\infty)$ and a sequence of iterates $\{(u^i,\rho^i)\}_{i\in\NN}$ given by  \eqref{minSchem1}, we associate a {\it\bf discrete trajectory} $(u_\bdel,\rho_\bdel):[0,+\infty) \to H^1(\Om)\times W^{1,p}(\Om)$ defined as the  
{\it left continuous piecewise constant interpolation} of the $(u^i,\rho^i)$'s below. More precisely, we set 
$$u_\bdel(0)=u_0\,,\quad\rho_\bdel(0)=\rho^\e_0\,,$$ 
and for $t>0$, 
\begin{equation}\label{discrTraj}
\begin{cases}
\displaystyle u_\bdel(t):= u^{i}\\
\rho_\bdel(t):=\rho^{i} 
\end{cases}
\qquad \text{if $t\in(t^{i-1},t^{i}]$}\,.
\end{equation}

By analogy with the standard notion of minimizing movements, we now introduce the following definition. 

\begin{definition}[Unilateral Minimizing Movements]\label{defminimov}
Let $u_0\in H^1(\Om)\cap L^\infty(\Om)$. We say that a pair $(u,\rho):[0,+\infty)\to L^2(\Om)\times L^p(\Om)$ is a {\it (generalized) unilateral minimizing movement} for $\E_\e$ starting from 
$(u_0,\rho_0^\e)$ if there exist a sequence $\{\bdel_k\}_{k\in\NN}$ of partitions of $[0,+\infty)$ satisfying $|\bdel_k|\to 0$,  and  associated discrete trajectories  
$\{(u_{\bdel_k},\rho_{\bdel_k})\}_{k\in\NN}$ such that 
$$(u_{\bdel_k}(t),\rho_{\bdel_k}(t))\mathop{\longrightarrow}\limits_{k\to\infty} (u(t),\rho(t))  \quad\text{strongly in $L^2(\Om)\times L^p(\Om)$ for every $t\geq 0$}\,.$$
We denote by $GUMM(u_0,\rho_0^\e)$ the collection of all (generalized) unilateral minimizing movements for $\E_\e$ starting from $(u_0,\rho_0^\e)$. 
\end{definition}

\begin{remark}
At this stage we do not claim that the collection $GUMM(u_0,\rho_0^\e)$  is not empty. This will be proved in the next section through a compactness result on discrete trajectories (see Lemmas~\ref{convrok} \& \ref{convuk}, and Corollary \ref{coro}).  
\end{remark}

\section{Existence of generalized unilateral minimizing movements}\label{sec4}

The object of this section is to provide an accurate information on the evolution laws of generalized unilateral minimizing movements. To avoid some technicalities in the analysis, we shall restrict ourselves to generalized unilateral minimizing movements arising from a sequence of discrete trajectories whose time partitions $\{\bdel_k\}_{k\in\NN}$ satisfy
\begin{equation}\label{condGUAMM}
\sup_{k\in\mathbb{N}}\left(\sup_{i\geq 1}\,\frac{\delta_k^{i+1}}{\delta_k^i}\right)<\infty\,.
\end{equation}
This condition is not essential and can be removed (the alternative argument, based on De Giorgi interpolations, can be found in the preliminary version of this paper  \cite{BM}). 
 
\vskip3pt

The main result of this section can be summarized in the following theorem. 

\begin{theorem}\label{thmGUMM}
Let $\Om \subset \RR^N$ be a bounded open set with $\C^{1,1}$ boundary. For an initial data $u_0\in H^1(\Om)\cap L^\infty(\Om)$ and $\rho_0^\e$ given by \eqref{rho0},  let $(u_\e,\rho_\e)\in GUMM(u_0,\rho_0^\e)$ be a strong $L^2(\Om)\times L^p(\Om)$-limit of some discrete trajectories $\{(u_k,\rho_k)\}_{k\in\NN}$ obtained from a sequence of partitions $\{\bdel_k\}_{k\in\NN}$ of $[0,+\infty)$ satisfying $|\bdel_k|\to 0$ and \eqref{condGUAMM}. 
Then, the following properties hold:
\begin{eqnarray*}
&\ds u_\e\in AC^2([0,+\infty);L^2(\Om))\cap L^\infty(0,+\infty;H^1(\Om)) \cap L^2_{\rm loc}(0,+\infty;H^2(\Om))\,,\\[5pt]
& \ds \rho_\e\in L^\infty(0,+\infty;W^{1,p}(\Om))\,,\;\text{  $0\leq \rho_\e(t)\leq \rho_\e(s)\leq 1$ for every $t\geq s\geq 0$}\,,
\end{eqnarray*}
and 
\begin{equation}\label{equa1}\begin{cases}
\ds u^\prime_\e = {\rm div}\big((\eta_\e+ \rho_\e^2)\nabla u_\e\big)-\beta(u_\e-g)  & \text{in $L^2(0,+\infty; L^2(\Om))$ ,}\\[5pt]
\ds\frac{\partial u_\e}{\partial \nu}=0 & \text{in $L^2(0,+\infty; H^{1/2}(\partial\Om))$ ,}\\[5pt]
u_\e(0)= u_0\, ,
\end{cases}\end{equation}
with 
\begin{equation}\label{equa2}
\begin{cases}
\E_\e(u_\e(t),\rho_\e(t))\leq \E_\e(u_\e(t),\rho) \text{ for every $t\geq 0$ and $\rho\in W^{1,p}(\Om)$ such that $\rho\leq \rho_\e(t)$ in $\Om$} \,,\\
\rho_\e(0)=\rho_0^\e  \,.
\end{cases}
\end{equation}
Moreover, $t\mapsto \E_\e(u_\e(t),\rho_\e(t))$ has finite pointwise variation in $[0,+\infty)$, and there exists an (at most) countable set 
$\mathcal{N}_\e\subset (0,+\infty)$ such that 
\begin{itemize}
\item[(i)] $(u_\e,\rho_\e):[0,+\infty)\setminus\mathcal{N}_\e\to H^1(\Om)\times W^{1,p}(\Om)$ is strongly continuous;
\vskip3pt
\item[(ii)] for every $s\in [0,+\infty)\setminus\mathcal{N}_\e$, and every $t\geq s$,
\begin{equation}\label{equa3}
\E_\e(u_\e(t),\rho_\e(t))+\int_s^t \|u^\prime_\e(r)\|^2_{L^2(\Om)}\,dr 
\leq \E_\e(u_\e(s),\rho_\e(s))\,.
\end{equation}
\end{itemize}
\end{theorem}

The entire section is devoted to the proof of this result, and we now describe the main steps.  
We first obtain  suitable compactness results on  discrete trajectories which prove in particular that the collection $GUMM(u_0,\rho_0^\e)$ is not empty (Corollary \ref{coro}). Then we consider an arbitrary element $(u_\e,\rho_\e)$ in $GUMM(u_0,\rho_0^\e)$ arising from discrete trajectories  $\{(u_{k},\rho_{k})\}_{k\in\NN}$  and partitions $\{\bdel_k\}_{k\in\NN}$ satisfying~\eqref{condGUAMM} and the established compactness properties. 
Defining the (diffuse) surface energy at a time $t \geq 0$ by
\begin{equation}\label{se}
\mathfrak S_\e(t):= \int_{\Om} \left(\frac{\e^{p-1}}{p}|\nabla\rho_\e(t)|^p +\frac{\alpha}{p'\e}(1-\rho_\e(t))^p\right)\,dx\, ,
\end{equation}
and the bulk energy
\begin{equation}\label{be}
\mathfrak B_\e(t):=\frac{1}{2}\int_\Om (\eta_\e + \rho_\e^2(t)) |\nabla u_\e(t)|^2\, dx+\frac{\beta}{2}\int_\Om(u_\e(t)-g)^2\,dx\, ,
\end{equation}
we prove a preliminary minimality property of the phase field variable $\rho_\e$ leading to the increase of the surface energy $\mathfrak S_\e$. In turn, it implies the strong $W^{1,p}(\Om)$-continuity of $t \mapsto \rho_\e(t)$ outside a countable set. Then we establish  the inhomogeneous heat equation satisfied by $u_\e$. Exploiting a semi-group property for this equation, we show the decrease of the bulk energy $\mathfrak B_\e$ and, as a byproduct, 
the strong $H^1(\Om)$-continuity of $t \mapsto u_\e(t)$ outside a countable set. At this stage, we are able to derive the pointwise in time strong convergence in $H^1(\Om)\times W^{1,p}(\Om)$ of the sequence $\{(u_k,\rho_k)\}_{k\in\NN}$ away from a countable set. The announced minimality property of $\rho_\e$ as well as the Lyapunov inequality on the total energy are mainly consequences of these strong convergences.

\subsection{Compactness of discrete trajectories} \label{subape}

We fix an arbitrary $u_0\in H^1(\Om)\cap L^\infty(\Om)$, and we consider the function $\rho_0^\e$ determined by \eqref{rho0}. Let $\{\bdel_k\}_{k\in\NN}$ be an arbitrary sequence of partitions of $[0,+\infty)$ satisfying $|\bdel_k|\to 0$. We write 
$$\bdel_k=:\{\delta^i_k\}_{i\in\mathbb{N}^*}\,, \quad t_k^0:=0\,, \quad\text{and} \quad  t_k^i:=\sum_{j=1}^i \delta_k^j \;\text{ for $i\geq 1$}\,.$$ 
For each $k\in\mathbb{N}$ we consider a discrete trajectory $(u_k,\rho_k) \equiv (u_{\bdel_k},\rho_{\bdel_k}) :[0,+\infty)\to H^1(\Om)\times W^{1,p}(\Om)$  associated to $\bdel_k$ which is obtained from  \eqref{discrTraj}. We  next define for every $k\in\NN$ a further  left-continuous piecewise constant interpolation $\rho^-_k:[0,+\infty)\to W^{1,p}(\Om)$ of the iterates $\{\rho_k^i\}_{i\in\mathbb{N}}$ setting $\rho^-_k(0)=\rho^\e_0$, and for $t>0$, 
\begin{equation}\label{interp}
\rho^-_k(t):=\rho_k^{i-1}
\quad \text{if $t\in(t^{i-1}_k,t^{i}_k]$}\,.
\end{equation}
We also consider the piecewise affine interpolation $v_k:[0,+\infty)\to H^1(\Om)$ of the $u_k^i$'s defined  for each $k\in\NN$ by 
\begin{equation}\label{interpu}
v_k(t):=u_k^{i-1} +\frac{t-t^{i-1}_k}{\delta^i_k}(u_k^{i}-u_k^{i-1}) \qquad \text{if $t\in[t^{i-1}_k,t^{i}_k]$}\,. 
\end{equation}

We first state  {\it a priori} estimates based on a (non optimal) discrete energy inequality. It is obtained by taking  the solution at time $t_k^{i-1}$ as competitor in the minimization problem at time $t_k^i$. An optimal energy inequality will be proved later on (see Proposition \ref{energyineq}). The higher order estimate on the sequence $\{u_k\}$ is obtained by means of an elliptic regularity result postponed to the appendix (see Lemma \ref{ellipreg}).

\begin{lemma}\label{inegenergdiscr}
There exists a constant $C_\e>0$ (independent of $k$) such that
\begin{equation}\label{enineqdiscr}
\sup_{t \geq 0}Ê\|\nabla u_k(t)\|_{L^2(\Om)} + \sup_{t \geq 0}Ê\|\nabla \rho_k(t)\|_{L^p(\Om)} + \int_0^{+\infty} \|v_k'(t)\|^2_{L^2(\Om)}\, dt \leq C_\e\,.
\end{equation}
Moreover, $u_k(t)\in H^2(\Om)$ and $\frac{\partial u_k(t)}{\partial \nu}= 0$ in $H^{1/2}(\partial\Om)$ for every $t > 0$, and  for each $T>0$, 
\begin{equation}\label{L2H2bound}
\int_0^T\|u_k(t)\|^2_{H^2(\Om)}\,dt\leq C_{\e,T}
\end{equation}
for a constant $C_{\e,T}>0$ (independent of $k$).
\end{lemma}

\begin{proof}
Taking $(u^{i-1}_k,\rho^{i-1}_k)$ as a competitor in the minimization problem \eqref{minSchem1} yields
$$\E_\e(u^i_k,\rho^i_k) +\frac{1}{2\delta_k^i}\|u^i_k-u^{i-1}_k\|^2_{L^2(\Om)} \leq \E_\e(u^{i-1}_k,\rho^{i-1}_k)\,.$$
Summing up for $i=1$ to $j$ leads to
$$\E_\e(u^j_k,\rho^j_k) +\sum_{i=1}^j \frac{1}{2\delta_k^i}\|u^i_k-u^{i-1}_k\|^2_{L^2(\Om)} \leq \E_\e(u_0,\rho^\e_0)\,,$$
or still, for every $t \geq 0$,
\begin{equation}\label{1046}
\E_\e(u_k(t),\rho_k(t)) +\frac12 \int_0^t \|v'_k(s)\|^2_{L^2(\Om)}\, ds \leq \E_\e(u_0,\rho^\e_0)\,.
\end{equation}
The first upper bound \eqref{enineqdiscr} now follows from the expression of the energy $\E_\e$.

To show estimate \eqref{L2H2bound}, let $t>0$ be such that $t\in (t_k^{i-1},t_k^i)$ for some integer $i\geq 1$. By minimality $u_k(t)$ solves
\begin{equation}\label{1101}
\begin{cases}
-{\rm div}((\eta_\e+\rho_k^2(t))\nabla u_k(t)) = -v_k'(t) -\beta(u_k(t)-g) & \text{in $H^{-1}(\Om)$}\,,\\
(\eta_\e+\rho_k^2(t))\nabla u_k(t)\cdot\nu = 0 & \text{in $H^{-1/2}(\partial\Om)$}\,.
\end{cases}
\end{equation}
From Lemma \ref{ellipreg} we deduce that $u_k(t)\in H^2(\Om)$ and $\frac{\partial u_k(t)}{\partial \nu}= 0$ in $H^{1/2}(\partial\Om)$ with the estimate
$$\|u_k(t)\|_{H^2(\Om)}\leq  C_\e(1+\|\nabla\rho_k(t)\|_{L^p(\Om;\RR^N)})^{ \gamma}\big(\|v_k'(t)\|_{L^2(\Om)}+\beta \|u_k(t)-g\|_{L^2(\Om)}+\|u_k(t)\|_{H^1(\Om)}\big)\,.$$
In view of \eqref{enineqdiscr} and Lemma \ref{ptwbdmonotdisc}, we infer that \eqref{L2H2bound} holds.
\end{proof}

We are now in position to  establish a compactness result for the sequences $\{u_k\}_{k\in\NN}$ and $\{\rho_k\}_{k\in\NN}$. We start with $\{\rho_k\}_{k\in\NN}$ and $\{\rho^-_k\}_{k\in\NN}$. 

\begin{lemma}\label{convrok}
There exist a subsequence $k_n\to\infty$ and a strongly measurable map $\rho_\e:[0,+\infty)\to W^{1,p}(\Om)$ such that $\rho_{k_n}(t)\rightharpoonup \rho_\e(t)$ weakly in $W^{1,p}(\Om)$ for every $t\geq 0$. In addition, $\rho_\e\in L^\infty(0,+\infty;W^{1,p}(\Om))$, $\rho_\e(0)=\rho_0^\e$, and $0\leq \rho_\e(t)\leq \rho_\e(s)\leq 1$  in~$\Om$ for every $t \geq s\geq 0 \,$. 
\end{lemma}

\begin{proof}
By Lemma \ref{ptwbdmonotdisc}, $\rho_k:[0,+\infty)\to L^1(\Om)$ is monotone non-increasing, and $0\leq \rho_k(t)\leq 1$ in $\Om$ for every $t\geq 0$. 
By a generalized version of Helly's selection principle (see \cite[Theorem~3.2]{MM}), we deduce that there exists a subsequence $k_n\to\infty$ and a map $\rho_\e:[0,+\infty)\to L^1(\Omega)$ such that $\rho_{k_n}(t) \rightharpoonup \rho_\e(t)$ weakly in $L^1(\Omega)$ for every $t\geq 0$. 
On the other hand, since
$$\E_\e(u_0,\rho^\e_0) \leq \E_\e(u_0,1) \leq \|\nabla u_0\|^2_{L^2(\Om;\RR^N)}+\frac{\beta}{2}\|u_0-g\|^2_{L^2(\Om)}\,,$$
we derive from Lemma \ref{inegenergdiscr} that
$$\sup_{t\geq 0}\|\rho_k(t)\|_{W^{1,p}(\Om)}\leq C_\e\,, $$
for some constant $C_\e>0$ independent of $k$. Therefore, $\rho_{k_n}(t)\wto \rho_\e(t)$ weakly in $W^{1,p}(\Om)$, and  $\rho_{k_n}(t)\to \rho_\e(t)$ in $\mathscr{C}^0(\overline\Om)$ for every $t\geq 0$ by the Sobolev Imbedding Theorem.
In particular $\rho_\e(t) \in W^{1,p}(\Om)$ for every $t\geq 0$, and by lower semicontinuity, 
$$\sup_{t\geq 0}\|\rho_\e(t)\|_{W^{1,p}(\Om)}\leq C_\e\,.$$ 
Moreover, since $0\leq \rho_k(t)\leq \rho_k(s)\leq 1$  in $\Omega$ whenever $s\leq t$,  we deduce from the  uniform convergence that $0\leq \rho_\e(t)\leq \rho_\e(s)\leq 1$  in $\Om$ for every $t\geq s\geq 0$.   

Since $\rho_\e:[0,+\infty)\to W^{1,p}(\Om)$ is a pointwise weak limit of a sequence of measurable (locally) simple functions, we deduce that $\rho_\e:[0,+\infty)\to W^{1,p}(\Om)$ is weakly measurable, hence strongly measurable thanks to the separability of  $W^{1,p}(\Om)$ and Pettis~Theorem. 
\end{proof}

Through the same argument we obtain the convergence of the sequence $\{\rho^-_k\}_{k\in\NN}$ (defined in \eqref{interp}).

\begin{lemma}\label{convrok-}
Let $\{k_n\}_{n\in\NN}$ be the subsequence given by Lemma~\ref{convrok}.  There exist a further subsequence (not relabeled) and a strongly measurable map $\rho^-_\e:[0,+\infty)\to W^{1,p}(\Om)$ such that $\rho^-_{k_n}(t)\wto \rho^-_\e(t)$ weakly in $W^{1,p}(\Om)$ for every $t\geq 0$. In addition, $\rho^-_\e\in L^\infty(0,+\infty;W^{1,p}(\Om))$, $\rho^-_\e(0)=\rho_0^\e$, and $0\leq  \rho_\e(t) \leq \rho^-_\e(t)\leq \rho^-_\e(s)\leq 1$  in~$\Om$ for every $t \geq s\geq 0 \,$. 
\end{lemma}

We continue with the compactness of the sequences $\{u_k\}_{k\in\NN}$ and $\{v_k\}_{k\in\NN}$ (defined by \eqref{interpu}). 

\begin{lemma}\label{convuk}
Let $\{k_n\}_{n\in\NN}$ be the subsequence given by Lemma~\ref{convrok-}. There exist a further subsequence (not relabeled) and a strongly measurable map $u_\e:[0,+\infty)\to H^1(\Om)$ such that $u_{k_n}(t)\wto u_\e(t)$ and $v_{k_n}(t)\wto u_\e(t)$ weakly in $H^1(\Om)$ for every $t\geq 0$. 
In addition, 
\begin{itemize}
\item[(i)] $u_\e(0)=u_0\,$;
\vskip3pt
\item[(ii)] $\|u_\e(t)\|_{L^{\infty}(\Om)}\leq  \max\{\|u_0\|_{L^\infty(\Omega)}, \|g\|_{L^\infty(\Om)}\}$ for every $t\geq 0\,$;
\vskip3pt
\item[(iii)] $u_\e \in  L^\infty(0,+\infty;H^1(\Om))\cap L^2_{\rm loc}([0,+\infty);H^2(\Om))\,$;
\vskip3pt
\item[(iv)] $\displaystyle \frac{\partial u_\e}{\partial\nu}=0$ in $L^2(0,+\infty;H^{1/2}(\partial\Om))$;
\vskip3pt
\item[(v)] $u_\e\in AC^2([0,+\infty);L^2(\Om))$ and 
$$
\int_0^{+\infty}\|u^\prime_\e(t)\|^2_{L^2(\Om)}\,dt \leq \|\nabla u_0\|^2_{L^2(\Om;\RR^N)}+\frac{\beta}{2}\|u_0-g\|^2_{L^2(\Om)}\,;
$$
\vskip3pt
\item[(vi)] $v^\prime_{k_n}\wto u_\e^\prime$ weakly in  $L^2(0,+\infty;L^2(\Om))$. 
\end{itemize}
\end{lemma}

\begin{proof}
We start by establishing the compactness of the sequence $\{v_k\}$. First Lemma~\ref{ptwbdmonotdisc} yields for every  $t\geq 0$, 
\begin{equation}\label{linftyvk}
\|v_{k_n}(t)\|_{L^\infty(\Om)}\leq \|u_{k_n}(t)\|_{L^\infty(\Om)}\leq  \max\{\|u_0\|_{L^\infty(\Om)}, \|g\|_{L^\infty(\Om)}\}\,.
\end{equation}
Then, combining the  bounds in \eqref{enineqdiscr} together with \eqref{interpu},  we infer that 
\begin{equation}\label{bdgradl2vk}
\sup_{t\geq 0}\|\nabla v_{k_n}(t)\|_{L^2(\Om;\RR^N)}\leq \sup_{t\geq 0}\|\nabla u_{k_n}(t)\|_{L^2(\Om;\RR^N)}\leq C_\e\,, 
\end{equation}
for some constant $C_\e>0$ independent of $k_n$. Consequently, for every $T>0$ the set $\bigcup_{n}v_{k_n}([0,T])$ is relatively compact in $L^2(\Om)$. On the other hand, \eqref{enineqdiscr} yields
\begin{equation}\label{bdtimedervk}
\int_0^{+\infty}\|v^\prime_{k_n}(r)\|^2_{L^2(\Om)}\,dr\leq  \E_\e(u_0,\rho_0^\e)\leq  \E_\e(u_0,1) \leq  \|\nabla u_0\|^2_{L^2(\Om;\RR^N)}+\frac{\beta}{2}\|u_0-g\|^2_{L^2(\Om)}\,. 
\end{equation}
Since for any  $t\geq s\geq 0$ we have 
\begin{equation}\label{abscontvk}
\|v_{k_n}(t)-v_{k_n}(s)\|_{L^2(\Om)}\leq \int_s^t\|v^\prime_{k_n}(r)\|_{L^2(\Om)} \,dr\,,
\end{equation}
we deduce from \eqref{bdtimedervk} and Cauchy-Schwarz inequality that 
\begin{equation}\label{holdcontrvk}
\|v_{k_n}(t)-v_{k_n}(s)\|^2_{L^2(\Om)}\leq (t-s)\int_s^t\|v^\prime_{k_n}(r)\|^2_{L^2(\Om)}\,dr \leq (t-s)\left( \|\nabla u_0\|^2_{L^2(\Om;\RR^N)}
+\frac{\beta}{2}\|u_0-g\|^2_{L^2(\Om)}\right)\,.
\end{equation}
By the Arzela-Ascoli Theorem we can find a subsequence of $\{k_n\}$ (not relabeled) such that 
\begin{equation}\label{convhold}
v_{k_n}\to u_\e \quad \text{in $\C^{0}([0,T];L^2(\Om))$ for every $T>0$} \,,
\end{equation}
for some $u_\e\in \C^{0,1/2}([0,+\infty);L^2(\Om))$. In particular, $v_{k_n}(t)\to u_\e(t)$ strongly in $L^2(\Om)$ for every $t\geq 0$, which yields {\it (i)} since $v_{k_n}(0)=u_0$. 

\vskip3pt

On the other hand, in view of estimates \eqref{linftyvk} and \eqref{bdgradl2vk}, we obtain {\it (ii)} and the fact that $v_{k_n}(t) \wto u_\e(t)$ weakly in $H^1(\Omega)$ for every $t\geq 0$. By lower semicontinuity we also deduce from \eqref{bdgradl2vk} that 
\begin{equation}\label{1005}
\sup_{t\geq 0}\|u_\e(t)\|_{H^1(\Om)}\leq C_\e\,.
\end{equation}

\vskip3pt

We next show the compactness of the sequence $\{u_{k_n}\}$. Let us now consider an arbitrary $t>0$. For each $n\in\NN$ there is a unique $i\in \NN$ such that $t\in(t^{i-1}_{k_n}, t^i_{k_n}]$. We then have $u_{k_n}(t)= u_{k_n}(t^i_{k_n})=v_{k_n}(t^i_{k_n})$. Consequently, by \eqref{holdcontrvk}, 
\begin{align}
\nonumber\|u_{k_n}(t)-u_\e(t)\|_{L^2(\Om)} & = \|v_{k_n}(t^i_{k_n})-u_\e(t)\|_{L^2(\Om)} \\
\nonumber& \leq \|v_{k_n}(t^i_{k_n})-v_{k_n}(t)\|_{L^2(\Om)} + \|v_{k_n}(t)-u_\e(t)\|_{L^2(\Om)} \\
\label{coucou3}& \leq C\sqrt{|\bdel_{k_n}|} +  \|v_{k_n}(t)-u_\e(t)\|_{L^2(\Om)} \mathop{\longrightarrow}\limits_{n\to\infty} 0\,.
\end{align}
Hence $u_{k_n}(t)\to u_\e(t)$ strongly in $L^2(\Om)$ for every $t\geq 0$, and in view of \eqref{bdgradl2vk} we infer  that 
$u_{k_n}(t)\wto u_\e(t)$ weakly in $H^1(\Om)$ for every $t\geq 0$.
The mappings $t \mapsto u_{k_n}(t)$ being (locally) simple and measurable, we conclude as in the proof of Lemma~\ref{convrok} that $u_\e:[0,+\infty)\to H^1(\Om)$ is strongly measurable. Moreover $u_\eÊ\in L^\infty(0,+\infty;H^1(\Om))$ by \eqref{1005}.  In view of \eqref{L2H2bound},  $u_{k_n} \wto u_\e$ weakly in $L^2_{\rm loc}(0,+\infty;H^2(\Om))$ which shows that $u_\e\in L^2_{\rm loc}(0,+\infty;H^2(\Om))$. Item {\it (iii)} is thus proved.

\vskip3pt

We next show that the Neumann boundary condition  $\frac{\partial u_\e}{\partial \nu}= 0$ in $L^2(0,+\infty;H^{1/2}(\partial\Om))$ holds. To this purpose, let us fix $T>0$ and select an arbitrary test function $\varphi\in L^2(0,T;H^1(\Om))$.  By Lemma \ref{inegenergdiscr} we have  $\frac{\partial u_k(t)}{\partial \nu}= 0$ in $H^{1/2}(\partial \Om)$ for every $t>0$, and consequently
\begin{multline*}
\int_0^T \int_\Om (-\Delta u_\e) \varphi\,dx\,dt =\lim_{n\to\infty} \int_0^T \int_\Om (-\Delta u_{k_n}) \varphi\,dx\,dt\\
= \lim_{n\to\infty} \int_0^T \int_\Om \nabla u_{k_n} \cdot \nabla \varphi\,dx\,dt
=  \int_0^T \int_\Om \nabla u_{\e} \cdot \nabla \varphi\,dx\,dt\,.
\end{multline*}
From the arbitrariness of $\varphi$ and $T$, we conclude that  $\frac{\partial u_\e}{\partial \nu}= 0$ in $L^2(0,T;H^{1/2}(\partial \Om))$ for every $T>0$ which completes the proof of item {\it (iv)}.

\vskip3pt

We next show the absolute continuity in time of $u_\e$. We note that \eqref{bdtimedervk} tells us that the functions $A_{k_n}:t\in(0,+\infty)\mapsto \|v^\prime_{k_n}(t)\|_{L^2(\Om)}$ are bounded in $L^2(0,+\infty)$. Hence we can find a further subsequence (not relabeled) such that $A_{k_n}\wto A$ weakly in $L^2(0,+\infty)$, for a non-negative function $A\in L^2(0,+\infty)$ satisfying 
$$\int_0^{+\infty} A^2(t)\,dt \leq   \|\nabla u_0\|^2_{L^2(\Om;\RR^N)}+\frac{\beta}{2}\|u_0-g\|^2_{L^2(\Om)}\,. $$
Letting $n\to \infty$ in \eqref{abscontvk}, we conclude that for every $t\geq s\geq 0$, 
$$\|u_{\e}(t)-u_{\e}(s)\|_{L^2(\Om)}\leq \int_s^t A(r) \,dr\,,$$
which shows that $u_\e\in AC^2([0,+\infty);L^2(\Om))$, whence {\it (v)}. 

\vskip3pt

Now, since $\{v^\prime_{k_n}\}$ is bounded in $L^2(0,+\infty;L^2(\Om))$, up to a subsequence,  
$\{v^\prime_{k_n}\}$ converges weakly in  $L^2(0,+\infty;L^2(\Om))$ to some element in $L^2(0,+\infty;L^2(\Om))$ which has to agree with $u^\prime_\e$ by \eqref{convhold}. This implies that {\it (vi)} holds. 
\end{proof}

\begin{remark}\label{weakcontueps}
As a consequence of {\it (iii)} and {\it (v)} in the previous lemma, $u_\e:[0,+\infty)\to H^1(\Om)$ is weakly continuous. 
\end{remark}

As an immediate consequence of Lemmas  \ref{convrok} \& \ref{convuk}, we obtain that unilateral minimizing movements starting from $(u_0,\rho_0^\e)$ do exist.

\begin{corollary}\label{coro}
The collection $GUMM(u_0,\rho_0^\e)$ is not empty. 
\end{corollary}

\subsection{Time continuity for the phase field variable and the surface energy}\label{subapost}

For the rest of this section, we consider an arbitrary element $(u_\e,\rho_\e)\in GUMM(u_0,\rho_0^\e)$ and discrete trajectories  $\{(u_{k},\rho_{k})\}_{k\in\NN}$  associated to partitions $\{\bdel_k\}_{k\in\NN}$ satisfying~\eqref{condGUAMM}. Without loss of generality, we also assume that all the results of the previous subsection do hold.
\vskip5pt

We first establish several properties of the phase field $\rho_\e$, starting from a (weak) minimality principle with respect to the diffuse surface energy.

\begin{proposition}\label{decenergysurf}
For every $t\geq 0$,  
$$\int_{\Om} \left(\frac{\e^{p-1}}{p}|\nabla\rho_\varepsilon(t)|^p +\frac{\alpha}{p'\e}(1-\rho_\varepsilon(t))^p\right)\,dx \leq  \int_{\Om} \left(\frac{\e^{p-1}}{p}|\nabla\rho|^p +\frac{\alpha}{p'\e}(1-\rho)^p\right)\,dx$$
for all $\rho\in W^{1,p}(\Om)$ such that $\rho\leq \rho_\e(t)$  in $\Om$. 
In  particular, the surface energy $\mathfrak S_\e$ defined in \eqref{se} is non-decreasing on $[0,+\infty)$, and thus 
continuous outside an (at most) countable set $\mathcal{S}_{\e}\subset[0,+\infty)$.
\end{proposition}

\begin{proof}
Fix $t >0$ and let $i \in\mathbb{N}$ be such that $t \in (t_k^{i-1},t_k^i]$. Consider a function $\rho \in W^{1,p}(\Om)$ such that $\rho \leq \rho_\e(t)$  in~$\Om$, and define $\hat \rho_k:= \rho \wedge \rho_k(t)$. Then $\hat \rho_k \in W^{1,p}(\Om)$ and $\hat \rho_k \leq \rho_k(t) \leq \rho_k^{i-1}$. By the minimality properties of the pair $(u_k(t),\rho_k(t))$,
$$\E_\e(u_k(t),\rho_k(t)) \leq \E_\e(u_k(t),\hat \rho_k)\,,$$
and  since $\hat \rho_k \leq \rho_k(t)$,
\begin{equation}\label{surfenergy}
\int_{\Om} \left(\frac{\e^{p-1}}{p}|\nabla\rho_k(t)|^p +\frac{\alpha}{p'\e}(1-\rho_k(t))^p\right)\,dx
\leq \int_{\Om} \left(\frac{\e^{p-1}}{p}|\nabla \hat \rho_k|^p +\frac{\alpha}{p'\e}(1-\hat \rho_k)^p\right)\,dx\,.
\end{equation}
Let us now define the measurable sets $A_k:=\{ \rho \leq \rho_k(t)\}$. By definition of $\hat \rho_k$, we have
$$\int_\Om |\nabla \hat \rho_k|^p\, dx = \int_{A_k} |\nabla \rho|^p\, dx + \int_{\Om \setminus A_k}|\nabla \rho_k(t)|^p\, dx\,,$$
and thanks to \eqref{surfenergy}, we infer that 
\begin{equation}\label{surfenergy2}
 \frac{\e^{p-1}}{p}\int_{A_k}|\nabla\rho_k(t)|^p\, dx  + \frac{\alpha}{p'\e}\int_\Om (1-\rho_k(t))^p\,dx
\leq \frac{\e^{p-1}}{p}\int_{A_k} |\nabla \rho|^p\, dx  + \frac{\alpha}{p'\e}\int_\Om(1-\hat \rho_k)^p\,dx\,.
\end{equation}
Since $\rho_k(t) \to \rho_\e(t)$ strongly in $L^p(\Om)$ and $\rho \leq \rho_\e(t)$ in $\Om$, we deduce that $\mathscr L^N(\Om \setminus A_k) \to 0$. As a consequence,
$$ \int_{A_k} |\nabla \rho|^p\, dx \to  \int_{\Om} |\nabla \rho|^p\, dx\,,$$
and $\chi_{A_k} \nabla \rho_k(t) \wto \nabla \rho_\e(t)$ weakly in $L^p(\Om;\RR^N)$ which in turn leads to
$$ \liminf_{k \to \infty}\int_{A_k} |\nabla\rho_k(t)|^p\, dx \geq \int_{\Om} |\nabla\rho_\e(t)|^p\, dx\,.$$
Passing to the limit in \eqref{surfenergy2} as $k \to \infty$ yields
$$
\int_{\Om} \left( \frac{\e^{p-1}}{p}|\nabla\rho_\e(t)|^p +\frac{\alpha}{p'\e}(1-\rho_\e(t))^p\right)\,dx
\leq \int_{\Om} \left(\frac{\e^{p-1}}{p}|\nabla \rho|^p\, + \frac{\alpha}{p'\e}(1-\rho)^p\right)\,dx\,.
$$
In particular, taking $\rho=\rho_\e(s)$ with $s \geq t$ leads to the announced monotonicity of the function $\mathfrak S_{\e}$.
\end{proof}

At this stage we  do not have any {\it a priori} time-regularity for  $t \mapsto \rho_\e(t)$ except that it is non-increasing, and thus it has finite pointwise  variation (with values in $L^1(\Om)$). In the following result we show that this mapping is actually strongly continuous in $W^{1,p}(\Om)$ outside a countable subset of $(0,+\infty)$ containing the discontinuity points of the surface energy $\mathfrak S_\e$.

\begin{lemma}\label{contH1}
There exists an (at most) countable set $\mathcal R_\e \subset (0,+\infty)$ containing $\mathcal{S}_\e$ such that the mapping $t \mapsto \rho_\e(t)$ is strongly continuous in $W^{1,p}(\Om)$ on $[0,+\infty) \setminus \mathcal R_\e$. In particular, $\rho_\e$ is strongly continuous at time $t=0$.  
\end{lemma}

\begin{proof}
Let $\mathcal R_\e$ be the union of the set $\mathcal{S}_{\e}$ given by Proposition \ref{decenergysurf} and the set of all discontinuity points of 
\begin{equation}\label{f}
t \mapsto \int_\Om \rho_\e(t)\, dx\,.
\end{equation}
Note that $\mathcal R_\e$ is at most countable by the decreasing property of the latter function. Let $t \in[0,+\infty)\setminus \mathcal R_\e$, we claim that $\rho_\e$ is strongly continuous in $W^{1,p}(\Om)$ at $t$. Consider a sequence $t_n \to t$ and extract a subsequence $\{t_{n_j}\} \subset \{t_n\}$ such that $\rho_\e(t_{n_j}) \wto \rho_\star$ weakly in $W^{1,p}(\Om)$ for some $\rho_\star \in W^{1,p}(\Om)$. Upon extracting a further subsequence, we may assume without loss of generality that $t_{n_j} >t$ for each $j \in \NN$ (the other case $t_{n_j} <t$ can be treated in a similar way). Then $\rho_\e(t_{n_j}) \leq \rho_\e(t)$  in $\Om$, and passing to the limit yields  $\rho_\star \leq \rho_\e(t)$  in $\Om$. On the other hand, by our choice of $t$ as a continuity point of  the mapping \eqref{f}, we have
$$\int_\Om \rho_\e(t)\, dx = \lim_{j \to \infty} \int_\Om \rho_\e(t_{n_j})\, dx = \int_\Om \rho_\star \, dx\,,$$
and thus $\rho_\star=\rho_\e(t)$. As a consequence, the limit is independent of the choice of the subsequence, and the full sequence $\{\rho_\e(t_n)\}$ weakly converges to $\rho_\e(t)$ in $W^{1,p}(\Om)$. Finally, using the fact that $t$ is a continuity point of $\mathfrak S_{\e}$, we get that $\mathfrak S_{\e}(t_n) \to \mathfrak S_{\e}(t)$, and thus $\|\rho_\e(t_n)\|_{W^{1,p}(\Om)} \to \|\rho_\e(t)\|_{W^{1,p}(\Om)}$. We then deduce that $\rho_\e(t_n) \to \rho_\e(t)$ strongly in $W^{1,p}(\Om)$.

It now remains to show that $\rho_\e$ is continuous at $t=0$. Let $t_n \downarrow 0$ be an arbitrary sequence. By Remark~\ref{weakcontueps} 
we have $u_\e(t_n) \wto u_0$ weakly in $H^1(\Om)$. By Lemma \ref{convrok}, $\rho_\e \in L^\infty(0,+\infty;W^{1,p}(\Om))$, and we can extract a (not relabeled) subsequence  such that $\rho_\e(t_n) \wto \rho_*$ weakly in $W^{1,p}(\Om)$ for some $\rho_* \in W^{1,p}(\Om)$.
According to the energy inequality \eqref{1046} proved in Lemma \ref{inegenergdiscr}, we have 
$\E_\e(u_k(t_n),\rho_k(t_n)) \leq \E_\e(u_0,\rho_0^\e)$ for all $n \in \NN$ and all $k \in \NN$. 
Now we apply Lemma~\ref{existsch1} to pass to the limit first as $k \to \infty$ and then as $n \to \infty$, which yields  $\E_\e(u_0,\rho_*) \leq \E_\e(u_0,\rho_0^\e)$. From the minimality property \eqref{rho0} satisfied by $\rho_0^\e$, we deduce that $\E_\e(u_0,\rho_*) = \E_\e(u_0,\rho_0^\e)$.
By uniqueness of the solution of the minimization problem \eqref{rho0}, we have $\rho_*=\rho_0^\e$. Moreover, we infer from the discussion above that 
$\lim_{n} \E_\e(u_\e(t_n),\rho_\e(t_n))=  \E_\e(u_0,\rho_0^\e)$, 
which implies that 
$\rho_\e(t_n)\to \rho_0^\e$ strongly in $W^{1,p}(\Om)$ by Lemma \ref{existsch1}. 
This convergence holds for the full sequence $\{t_n\}$ by uniqueness of the limit.
\end{proof}

Thanks to the just established continuity of $t \mapsto \rho_\e(t)$, we deduce that $\rho_\e^-$ and $\rho_\e$ actually coincide almost everywhere in time whenever \eqref{condGUAMM} holds.

\begin{corollary}\label{rho-}
There exists an $\mathscr{L}^1$-negligible set $\mathcal{M}_\e\subset[0,+\infty)$ such that  
$\rho^-_\e(t)= \rho_\e(t)$ for every  $t\in[0,+\infty)\setminus~\mathcal{M}_\e$. 
\end{corollary}

\begin{proof}
Let us consider the function $\ell_k:[0,+\infty)\to[0,+\infty)$ defined by 
\begin{equation}\label{1117}\ell_k(t):=\begin{cases} 
0 & \text{if $t\in[0,t_k^1]$}\,,\\
\displaystyle t_k^{i-1} + \frac{\delta_k^i}{\delta_k^{i+1}}(t-t_k^i) & \text{if $t\in (t_k^i,t_k^{i+1}]$ with $i\geq 1$}\,.
\end{cases}
\end{equation}
Notice that
$$\sup_{t\geq 0} |\ell_k(t)-t| \leq 3|\bdel_k| \mathop{\longrightarrow}\limits_{k\to\infty} 0 \,.$$
Setting 
$$\rho_\e^k(t):=\rho_\e(\ell_k(t))\,,$$
we infer from Lemma \ref{contH1} that $\rho_\e^k(t)\to \rho_\e(t)$ strongly in $L^1(\Om)$ for every $t\in[0,+\infty)
\setminus\mathcal{R}_\e$.
Since $0\leq \rho_\e\leq 1$, by dominated convergence we have $\rho^k_\e\to\rho_\e$ strongly in $L^1(0,T;L^1(\Om))$ for every $T>0$. 
Similarly,  by \eqref{ptwbdmonotdisceq} we have  that $\rho_k\to\rho_\e$ and $\rho_k^-\to \rho^-_\e$ strongly in $L^1(0,T;L^1(\Om))$ for every $T>0$.  
Given $T>0$ arbitrary,  we estimate 
$$\|\rho^-_\e-\rho_\e\|_{L^1(0,T;L^1(\Om))}\leq \|\rho^-_\e-\rho^-_k\|_{L^1(0,T;L^1(\Om))} +\|\rho^-_k-\rho^k_\e\|_{L^1(0,T;L^1(\Om))}  + \|\rho^k_\e-\rho_\e\|_{L^1(0,T;L^1(\Om))} \mathop{\longrightarrow}\limits_{k\to\infty} 0.  $$
Indeed, observing that $\rho_k^-(t)=\rho_k(\ell_k(t))$ and $\ell_k(t)\leq t$, we deduce from  \eqref{condGUAMM} that
\begin{multline*}
\int_0^T\|\rho_k^-(t)-\rho^k_\e(t)\|_{L^1(\Om)}\,dt = \int_{\delta^1_k}^T\|\rho_k(\ell_k(t))-\rho_\e(\ell_k(t))\|_{L^1(\Om)}\,dt \\
\leq \left(\sup_{i\geq 1}\,\frac{\delta_k^{i+1}}{\delta_k^i}\right)\int_0^{T}\|\rho_k(t)-\rho_\e(t)\|_{L^1(\Om)}\,dt\mathop{\longrightarrow}\limits_{k\to\infty} 0 \,. 
\end{multline*}
Hence $\|\rho^-_\e-\rho_\e\|_{L^1(0,T;L^1(\Om))}=0$ for every $T>0$, whence $\rho^-_\e(t)=\rho_\e(t)$ for all $t\in[0,+\infty)\setminus\mathcal{M}_\e$ for some  
$\mathscr{L}^1$-negligible set $\mathcal{M}_\e\subset[0,+\infty)$. 
\end{proof}

\subsection{Time continuity for $u_\e$ and  the bulk energy}\label{subbe}

We start  proving that $u_\e$ solves the inhomogeneous heat equation.

\begin{proposition}\label{convdiv}
The function $u_\e\in AC^2([0,+\infty);L^2(\Om))$ solves
\begin{equation}\label{parabolic}
\begin{cases}
\ds u^\prime_\e = {\rm div}\big((\eta_\e+ \rho_\e^2)\nabla u_\e\big)-\beta(u_\e-g)  & \text{in $L^2(0,+\infty;L^2(\Om))$}\,,\\
\ds \frac{\partial u_\e}{\partial \nu}=0 & \text{in }L^2(0,+\infty;H^{1/2}(\partial\Om))\,.\\
u_\e(0)= u_0\,.
\end{cases}
\end{equation}
\end{proposition}

\begin{proof}
In view of \eqref{1101}, $u_k$ satisfies 
$$\int_0^T \int_\Om \Big(v'_k \varphi + (\eta_\e+\rho_k^2)\nabla u_k\cdot \nabla \varphi + \beta (u_k-g) \varphi \Big) \, dx \, dt =0$$
for every $\varphi \in L^2(0,T;H^1(\Om))$ and $T>0$. 
Using the convergences established in Lemmas \ref{convrok} and \ref{convuk}, we can pass to the limit $k\to\infty$ in the previous formula to derive  
$$\int_0^T \int_\Om \Big(u'_\e \varphi + (\eta_\e+\rho_\e^2)\nabla u_\e\cdot \nabla \varphi + \beta (u_\e-g) \varphi \Big) \, dx \, dt =0\,.$$
The proof of \eqref{parabolic} is now an immediate consequence of the previous variational formulation together with Lemma \ref{convuk}, items {\it (i)} and {\it (iv)}. 
 \end{proof}

We are now in position to prove the decrease of the bulk energy $\mathfrak B_\e$. 

\begin{proposition}\label{semigroup}
Let $t_0>0$ and set $\rho^{t_0}_\e(t):=\rho_\e(t+t_0)$.  
For any  $w_0 \in H^1(\Om) \cap L^\infty(\Om)$ there exists a unique solution $w_\e \in AC^2([0,+\infty);L^2(\Om)) \cap L^\infty(0,+\infty; H^1(\Om))$ of 
\begin{equation}\label{parabolic2}
\begin{cases}
\displaystyle w^\prime_\e ={\rm div}\big(\big(\eta_\e+(\rho_\e^{t_0})^2\big)\nabla w_\e\big)-\beta(w_\e-g)  & \text{in $L^2_{\rm loc}([0,+\infty);H^{-1}(\Om))$}\,,\\
\displaystyle \big(\eta_\e+(\rho_\e^{t_0})^2\big) \nabla w_\e\cdot\nu= 0 & \text{in $L^2_{\rm loc}([0,+\infty);H^{-1/2}(\partial\Om))$}\,,\\
w_\e(0)= w_0\,, 
\end{cases}
\end{equation}
and $w_\e$ satisfies the following energy inequality for every $t\geq 0$, 
\begin{multline}\label{energineqwe}
\frac{1}{2}\int_\Om \big(\eta_\e + (\rho_\e^{t_0}(t))^2\big) |\nabla w_\e(t)|^2\, dx+\frac{\beta}{2}\int_\Om (w_\e(t)-g)^2\,dx\\
\leq  \frac{1}{2}\int_\Om \big(\eta_\e + \rho_\e^2(t_0)) |\nabla w_0|^2\, dx+\frac{\beta}{2}\int_\Om (w_0-g)^2\,dx\,.
\end{multline}
In particular, for any $t_0>0$ the function $u_\e(\,\cdot \,+t_0)$ is the unique solution of \eqref{parabolic2} with initial datum $w_0:=u_\e(t_0)$. 
As a consequence, the bulk energy $\mathfrak B_\e$ defined in \eqref{be} is non increasing on $[0,+\infty)$, and thus continuous outside an (at most) countable subset $\mathcal{B}_\e$ of $[0,+\infty)$.
\end{proposition}

\begin{proof}
{\it Step 1, Uniqueness.} Let $w_{\e,1}$ and $w_{\e,2}$ be two solutions of \eqref{parabolic2}, and set $z_\e:=w_{\e,1}-w_{\e,2}$. Then $z_\e(0)=0$.   
The variational formulation of \eqref{parabolic2} implies that for any $T>0$ and any test function $\phi \in L^2(0,T;H^1(\Om))$,
$$
\int_0^T \int_\Om \left( z'_\e\phi + \big(\eta_\e +(\rho^{t_0}_\e)^2\big) \nabla z_\e \cdot \nabla \phi +\beta z_\e\phi\right) \,dx \, dt =0\,.
$$
Choosing $\phi(t):=z_\e(t)\chi_{[0,T]}(t)$ as test function above yields
$$ \int_0^T \int_\Om  z'_\e z_\e\, dx \, dt \leq 0\quad\text{for every $T>0$}\,.$$
On the other hand, since $z_\e \in AC^2([0,+\infty);L^2(\Om))$, we have $\|z_\e(\cdot)\|^2_{L^2(\Om)}\in AC([0,+\infty))$ and 
$$\frac{d}{dt}\|z_\e(t)\|^2_{L^2(\Om)} = 2\int_\Om z'_\e(t) z_\e(t)\,dx \quad\text{for a.e. $t\in(0,+\infty)$}\,.$$
Therefore,
$$0\geq \int_0^T \int_\Om  z'_\e z_\e\, dx \, dt =\frac{1}{2}\|z_\e(T)\|^2_{L^2(\Om)}\quad\text{for every $T>0$}\,, $$
which shows that $z_\e\equiv 0$, {\it i.e.}, $w_{\e,1}=w_{\e,2}$. 

\vskip5pt

\noindent{\it Step 2, Existence.} 
For what concerns existence, we reproduce a minimizing movement scheme as before. More precisely, given a sequence $\tau_k\downarrow0$, we set  $\tau_k^i:=i\tau_k$ for $i \in \mathbb{N}$. Taking $w_k^0:=w_0$, we define recursively for all integer $i \geq 1$, $w_k^i \in H^1(\Om)$ as the unique solution of the minimization problem
$$
\min_{v \in H^1(\Om)} \left\{\frac{1}{2}\int_\Om \big(\eta_\e +(\rho^{t_0}_\e(t_k^{i-1}))^2\big) |\nabla v|^2\, dx + \frac{\beta}{2}\int_\Om (v-g)^2\,dx+ \frac{1}{2\tau_k} \int_\Om (v-w_k^{i-1})^2\, dx\right\}\,.
$$
Using the minimality of $w_k^i$ at each step and the fact that $0\leq \rho^{t_0}_\e(\tau_k^i)\leq \rho^{t_0}_\e(\tau_k^{i-1})$, we obtain that for every integer $i\geq 1$,  
\begin{multline}\label{energineqwk}
\frac{1}{2}\int_\Om \big(\eta_\e + (\rho_\e^{t_0}(\tau_k^{i-1}))^2\big) |\nabla w_k^i|^2\, dx+\frac{\beta}{2}\int_\Om (w_k^i-g)^2\,dx + \sum_{j=1}^i \frac{1}{2\tau_k}\int_\Om (w_k^j - w_k^{j-1})^2\, dx \\
\leq \frac{1}{2}\int_\Om (\eta_\e + \rho_\e^2(t_0)) |\nabla w_0|^2\, dx+\frac{\beta}{2}\int_\Om(w_0-g)^2\,dx\,.
\end{multline}
Let us now define the following piecewise constant and piecewise affine interpolations. Set $w_k(0)=\hat w_k(0)=w_0$, and for $t \in (\tau_k^{i-1},\tau_k^{i}]$, 
$$\begin{cases}
w_k(t):=w_k^i\,,\\
\varrho^{t_0}_k(t):=\rho^{t_0}_\e(\tau_k^{i-1})\,,\\
\hat w_k(t):=w_k^{i-1} +\tau_ k^{-1}(t-\tau_k^{i-1})(w_k^{i} - w_k^{i-1})\,.
\end{cases}$$
By Lemma \ref{contH1}, we have $\varrho^{t_0}_k(t) \to \rho^{t_0}_\e(t)$ strongly in $W^{1,p}(\Om)$ for all $t \in [0,+\infty) \setminus (-t_0+\mathcal R_\e)$. 
Arguing exactly as in the proof of Lemma \ref{convuk} we prove that (for a suitable subsequence) $w_k(t)\rightharpoonup w_\e(t)$ weakly in $H^1(\Om)$ for every $t\geq 0$ and $\hat w'_k \rightharpoonup w_\e'$ weakly in $L^2(0,+\infty;L^2(\Om))$, for 
some $w_\e\in AC^2([0,+\infty);L^2(\Om))\cap L^\infty(0,+\infty;H^1(\Om)$. Then we can reproduce with minor modifications the proof of Proposition~\ref{convdiv} to show that $w_\e$ is a solution of \eqref{parabolic2}. 

Since $0\leq \rho^{t_0}_\e(t) \leq \varrho^{t_0}_k(t)$ and $w_k(t)\to w_\e(t)$ strongly in $L^2(\Om)$ for every $t\geq 0$, we infer from \eqref{energineqwk} that for every $t \geq 0$, 
\begin{align*}
\frac{1}{2}\int_\Om (\eta_\e + \rho_\e^2(t_0)) |\nabla w_0|^2\, dx & +\frac{\beta}{2}\int_\Om(w_0-g)^2\,dx \\
&  \geq \liminf_{k\to\infty} \left(\frac{1}{2} \int_\Om \big(\eta_\e + (\rho^{t_0}_k(t))^2\big) |\nabla w_k(t)|^2\, dx+\frac{\beta}{2}
\int_\Om(w_k(t)-g)^2\,dx\right) \\
& \geq  \frac{1}{2} \int_\Om \big(\eta_\e + (\rho^{t_0}_\e(t))^2\big) |\nabla w_\e(t)|^2\, dx+\frac{\beta}{2}
\int_\Om(w_\e(t)-g)^2\,dx\,,
\end{align*}
and \eqref{energineqwe} is proved. 
\end{proof}

\begin{remark}
We notice that the proof of Lemma \ref{contH1} together with Remark \ref{weakcontueps} show that the function $\mathfrak B_\e$ is actually continuous at time $t=0$, {\it i.e.}, $0\not\in\mathcal B_\e$. 
\end{remark}

As a consequence of Lemma \ref{contH1} and Proposition \ref{semigroup}, we obtain the strong continuity in $H^1(\Om)$ of the mapping $t \mapsto u_\e(t)$ outside a countable subset of $(0,+\infty)$ containing the discontinuity points of $\mathfrak S_\e$ and~$\mathfrak B_\e$.

\begin{corollary}\label{contH1ue}
The mapping $u_\e:[0,+\infty)\to H^1(\Om)$ is strongly continuous on $[0,+\infty) \setminus (\mathcal R_\e \cup \mathcal{B}_\e)$.
\end{corollary}

\begin{proof}
Let us consider $t_0\in [0,+\infty) \setminus (\mathcal R_\e \cup \mathcal{B}_\e)$ and $\{t_n\}\subset [0,+\infty)$ an arbitrary sequence such that $t_n\to~t_0$. 
Since $t_0\not\in \mathcal R_\e \cup\mathcal{B}_\e$ we have  $\mathfrak B_\e(t_n)\to \mathfrak B_\e(t_0)$ and $\rho_\e(t_n)\to \rho_\e(t_0)$ strongly in $W^{1,p}(\Om)$. 
Therefore $\E_\e(u_\e(t_n),\rho_\e(t_n))\to \E_\e(u_\e(t_0),\rho_\e(t_0))$. On the other hand $u_\e(t_n)\rightharpoonup u_\e(t_0)$ weakly in $H^1(\Om)$ by Remark~\ref{weakcontueps}, and 
the conclusion follows from Lemma \ref{existsch1}. 
\end{proof}

\subsection{Strong convergences and limiting minimality}\label{subsc}

Thanks to the equation solved by $u_\e$, we are now able to improve the weak $H^1(\Om)$-convergence of the sequence $\{u_k(t)\}_{k\in\NN}$ into a strong convergence. We start proving that the bulk energy converges in time averages. 

\begin{lemma}\label{strongconv}
For every $t>s\geq 0$, 
\begin{equation}\label{convbulkenerg}
\lim_{k\to\infty}\int_s^t \int_\Om (\eta_\e+\rho^2_k(r))|\nabla u_k(r)|^2\,dx\,dr= \int_s^t \int_\Om (\eta_\e+\rho^2_\e(r))|\nabla u_\e(r)|^2\,dx\,dr\,.
\end{equation}
\end{lemma}

\begin{proof}
Taking $u_k(r)$ as test function in the variational formulation of \eqref{1101} and integrating in time between $s$ and $t$ leads to
$$\int_{s}^t \int_\Om (\eta_\e + \rho^2_k(r)) |\nabla u_k(r)|^2\, dx\, dr =- \int_s^{t}\int_\Om v'_k(r) u_k(r)\, dx\, dr-\beta\int_s^t\int_\Om(u_k(r)-g)u_k(r)\,dx\,dr\,.$$
From Lemma \ref{convuk}  we have  $u_k \to u_\e$ strongly in $L^2_{\rm loc}([0,+\infty);L^2(\Om))$  and $v'_k\rightharpoonup u'_\e$ weakly in $L^2(0,+\infty;L^2(\Om))$.  
Therefore, 
$$\lim_{k \to \infty} \int_s^t \int_\Om (\eta_\e + \rho_k^2(r)) |\nabla u_k(r)|^2\, dx \, dt =
 -\int_s^{t}\int_\Om u'_\e(r) u_\e(r)\, dx\, dr-\beta\int_s^t\int_\Om(u_\e(r)-g)u_\e(r)\,dx\,dr\,.$$
On the other hand, according to equation \eqref{parabolic} solved by $u_\e$,  
we have 
$$-\int_s^{t}\int_\Om u'_\e(r) u_\e(r)\, dx\, dr-\beta\int_s^t\int_\Om(u_\e(r)-g)u_\e(r)\,dx\,dr= \int_s^t \int_\Om (\eta_\e + \rho_\e^2(r)) |\nabla u_\e(r)|^2\, dx \, dr\,,$$
which leads  to \eqref{convbulkenerg}.
\end{proof}

Starting from Lemma \ref{strongconv}, we now localize in time the convergence of the sequence $\{\nabla u_k\}_{k \in \NN}$ by showing that  $\nabla u_k(t) \to \nabla u_\e(t)$ strongly in $L^2(\Om)$ at every continuity times $t$ of the bulk energy $\mathcal B_\e$. The proof is inspired from \cite[Lemma 5]{CD}.

\begin{lemma}\label{convforte}
For every $t \in [0,+\infty) \setminus \mathcal{B}_\e$,  $u_k(t) \to u_\e(t)$ strongly in $H^1(\Om)$.
\end{lemma}

\begin{proof} 
Let $t_0 \in[0,+\infty) \setminus \mathcal{B}_\e$. Since $\mathfrak B_\e$ is continuous at $t_0$, for every $\alpha>0$ there exists $\delta_\alpha>0$ such that 
$\mathfrak B_\e(t) \leq \mathfrak B_\e(t_0) +\alpha$ for all $t \in [t_0-\delta_\alpha,t_0]$. 

Let us fix $\alpha>0$ arbitrary. Since  $\E_\e(u_k^i,\rho_k^i) \leq \E_\e(u_k^{i-1},\rho_k^i)$ and $\rho_k^i \leq \rho_k^{i-1}$ in $\Om$ 
for each integers $k$ and $i\geq 1$, we infer that the function
$$t \mapsto \frac{1}{2}\int_\Om (\eta_\e + \rho_k^2(t)) |\nabla u_k(t)|^2\, dx+\frac{\beta}{2}\int_\Om(u_k(t)-g)^2\,dx$$
is non-increasing on $[0,+\infty)$, and thus 
\begin{multline*}
\delta_\alpha\left( \frac{1}{2}\int_\Om (\eta_\e + \rho^2_k(t_0))|\nabla u_k(t_0)|^2\, dx +\frac{\beta}{2}\int_\Om(u_k(t_0)-g)^2\,dx\right)\\
\leq \int_{t_0-\delta_\alpha}^{t_0}\left( \frac{1}{2}\int_\Om (\eta_\e + \rho^2_k(t))|\nabla u_k(t)|^2\, dx +\frac{\beta}{2}\int_\Om(u_k(t)-g)^2\,dx\right)\, dt\,.
\end{multline*}
By Lemma \ref{strongconv} and  the strong convergence of $u_k$ to $u_\e$ in $L^2_{\rm loc}([0,+\infty);L^2(\Om))$, we infer that
\begin{multline*}
\delta_\alpha \limsup_{k \to \infty}\left( \frac{1}{2}\int_\Om (\eta_\e + \rho^2_k(t_0))|\nabla u_k(t_0)|^2\, dx +\frac{\beta}{2}\int_\Om(u_k(t_0)-g)^2\,dx\right)\\ 
\leq \lim_{k \to \infty}\int_{t_0-\delta_\alpha}^{t_0}\left( \frac{1}{2}\int_\Om (\eta_\e + \rho^2_k(t))|\nabla u_k(t)|^2\, dx +\frac{\beta}{2}\int_\Om(u_k(t)-g)^2\,dx\right)\, dt\\
=\int_{t_0-\delta_\alpha}^{t_0}\mathfrak B_\e(t)\, dt \leq (\mathfrak B_\e(t_0)+\alpha)\delta_\alpha\,.
\end{multline*}
Dividing the previous inequality by $\delta_\alpha$ and using the  strong convergence of $u_k(t_0)$ in $L^2(\Om)$, we derive in view of the arbitrariness of $\alpha$ that 
$$ \limsup_{k \to \infty} \int_\Om (\eta_\e + \rho^2_k(t_0))|\nabla u_k(t_0)|^2\, dx\leq \int_\Om (\eta_\e + \rho^2_\e(t_0))|\nabla u_\e(t_0)|^2\, dx\,.$$
As in the proof of Lemma \ref{existsch1}  we obtain 
$$\int_\Om (\eta_\e + \rho^2_\e(t_0))|\nabla u_\e(t_0)|^2\, dx\leq   \liminf_{k \to \infty} \int_\Om (\eta_\e + \rho^2_k(t_0))|\nabla u_k(t_0)|^2\, dx\,.$$
Combining the last two inequalities we conclude 
\begin{equation}\label{convpttimeenerguk}
  \lim_{k \to \infty} \int_\Om (\eta_\e + \rho^2_k(t_0))|\nabla u_k(t_0)|^2\, dx= \int_\Om (\eta_\e + \rho^2_\e(t_0))|\nabla u_\e(t_0)|^2\, dx\,.
  \end{equation}
Finally, using \eqref{convpttimeenerguk}, the weak convergence of $\rho_k(t_0)$ to $\rho_\e(t_0)$ in $W^{1,p}(\Om)$, and the weak convergence of $u_k(t_0)$ to $u_\e(t_0)$ in $H^{1}(\Om)$, 
we can argue as in the proof of Lemma  \ref{existsch1} to show that $u_k(t_0) \to  u_\e(t_0)$ strongly in $H^1(\Om)$.
\end{proof}

We now derive the (strong) minimality property for $\rho_\e(t)$ at all  times, and as a byproduct, the $W^{1,p}(\Om)$-convergence of $\{\rho_k(t)\}_{k\in\NN}$ at all continuity points of the bulk energy.  We would like to stress that the proof of the minimality property of $\rho_\e$ strongly relies on the fact that the surface energy in the Ambrosio-Tortorelli functional has $p$-growth with $p>N$ (which ensures the convergence of the underlying obstacle problems).

\begin{proposition}\label{minrhoeps}
For every $t\geq 0$ the function $\rho_\e(t)$ satisfies 
\begin{equation}\label{eqlocminro}
\E_\e(u_\e(t),\rho_\e(t))  \leq \E_\e(u_\varepsilon(t),\rho) \quad \text{for all $\rho\in W^{1,p}(\Om)$ such that $\rho\leq \rho_\e(t)$     in $\Om$}\, .
\end{equation}
In addition, if $t\in [0,+\infty)\setminus \mathcal B_\e$ then 
\begin{equation}\label{superminro}
\rho_\e(t)={\rm argmin} \left\{\E_\e(u_\e(t),\rho) : \rho\in W^{1,p}(\Om) \text{ such that } \rho\leq \rho^-_\e(t) \text{ in }\Om\right\}\,,
\end{equation}
and $\rho_k(t)\to \rho_\e(t)$ strongly in $W^{1,p}(\Om)$. 
\end{proposition}

\begin{proof} Let us fix an arbitrary $t\geq 0$. Since $u_k(t)\rightharpoonup u_\e(t)$ weakly in $H^1(\Om)$, we can find a (not relabeled) subsequence and a nonnegative Radon measure $\mu\in\mathscr{M}(\mathbb{R}^N)$ supported in $\overline \Om$ such that 
$$  |\nabla u_k(t)|^2\mathscr L^N\res \, \Om \wto |\nabla u_\e(t)|^2\mathscr L^N \res \, \Om+ \mu$$
weakly* in $\M(\RR^N)$. Then we consider the functionals $\mathcal F_k$ and $\mathcal F$ defined on $W^{1,p}(\Om)$ by 
$$\mathcal F_k(\rho):=\begin{cases}
\E_\e(u_k(t),\rho) & \text{if $\rho\leq \rho_k^-(t)$}\,,\\
+\infty & \text{otherwise}\,,
\end{cases}\;\text{ and }\;
\mathcal F(\rho):=\begin{cases}
\displaystyle \E_\e(u_\e(t),\rho)+\frac{1}{2}\int_{\overline\Om} (\eta_\e+\rho^2)\,d\mu & \text{if $\rho\leq \rho_\e^-(t)$}\,,\\[8pt]
+\infty & \text{otherwise}\,.
\end{cases}$$
Note that by the Sobolev Imbedding $W^{1,p}(\Om) \hookrightarrow \mathscr C^0(\overline \Om)$, the functional $\mathcal F$ is well defined on the space $W^{1,p}(\Om)$. 
\vskip5pt

\noindent{\it Step 1.} We claim that $\mathcal F_k$ $\Gamma$-converges to $\mathcal F$ for the sequential weak $W^{1,p}(\Om)$-topology. For what concerns the lower bound, if $\{\hat \rho_k\} \subset W^{1,p}(\Om)$ is such that $\liminf_{k}Ê\mathcal F_k(\hat \rho_k)<\infty$ and $\hat \rho_{k} \wto \hat \rho$ weakly in $W^{1,p}(\Om)$,  then for a subsequence $\{k_j\}$ we have $\lim_{j}Ê\mathcal F_{k_j}(\hat \rho_{k_j})=\liminf_{k}Ê\mathcal F_k(\hat \rho_k)$, and $\hat \rho_{k_j} \to \hat \rho$  in $\mathscr{C}^0(\overline \Om)$ by the compact imbedding $W^{1,p}(\Om) \hookrightarrow \mathscr C^0(\overline \Om)$. 
Consequently $\hat \rho \leq \rho_\e^-(t)$  in~$\Om$, and 
$$\int_\Om \hat \rho_{k_j}^2 |\nabla u_{k_j}(t)|^2\, dx \to \int_\Om \hat \rho^2 |\nabla u_\e(t)|^2\, dx + \int_{\overline \Om} \hat \rho^2 \, d\mu\,.$$
Since the remaining terms in the energy $\mathcal F_{k}$ are independent of $k$ and lower semicontinuous for the weak $W^{1,p}(\Om)$-convergence, we deduce that 
$$\mathcal F(\hat \rho) \leq \liminf_{k \to \infty}Ê\mathcal F_k(\hat \rho_k)\,.$$
To show the upper bound, it is enough to consider $\hat \rho\in W^{1,p}(\Om)$ satisfying $\hat \rho \leq \rho^-_\e(t)$ in $\Om$. Let us set $\hat\rho_\delta:=\hat \rho - \delta$ where $\delta>0$ is small. Since $\rho_{k}^-(t)Ê\to \rho_\e^-(t)$ uniformly in $\overline \Om$, we have $\hat \rho_\delta \leq \rho_{k}^-(t)$ in $\Om$ whenever $k$ large enough (depending only on $\delta$). Hence, 
$$\lim_{\delta\downarrow 0}\, \limsup_{k\to \infty} \mathcal F_k(\hat \rho_\delta) \leq \mathcal F(\hat \rho)\,,$$ 
and we obtain from $\{\hat \rho_\delta\}_{\delta>0}$ a suitable recovery sequence for $\hat\rho$ through a diagonalization argument, which completes the proof of the $\Gamma$-convergence.
\vskip5pt

\noindent{\it Step 2.} Since
$$\rho_k(t)=\mathop{{\rm argmin}}\limits_{\rho\in W^{1,p}(\Om)} \mathcal{F}_k(\rho) \,,$$
and $\rho_k(t)\rightharpoonup \rho_\e(t)$ weakly in $W^{1,p}(\Om)$, we infer from Step 1  that 
\begin{equation}\label{presupermin}
\rho_\e(t)=\mathop{{\rm argmin}}\limits_{\rho\in W^{1,p}(\Om)} \mathcal{F}(\rho) \,. 
\end{equation}
Let us now fix an arbitrary $\rho\in W^{1,p}(\Om)$  such that $\rho\leq \rho_\e(t)$ in $\Om$, and set $\rho^+:=\rho\wedge 0$. Then $ \rho^+\in W^{1,p}(\Om)$, 
$0\leq\rho^+\leq \rho_\e(t)$ in $\Om$, and $\E_\e(u_\e(t),\rho^+)\leq \E_\e(u_\e(t),\rho)$. 
Since $\rho^+\leq \rho_\e(t)\leq \rho^-_\e(t)$ in $\Om$,   we have $\mathcal F(\rho_\e(t))\leq \mathcal F(\rho^+)$ which leads to  
$$\E_\e(u_\e(t),\rho_\e(t)) \leq\E_\e(u_\e(t),\rho_\e(t))+\frac{1}{2}\int_\Om (\rho^2_\e(t)-(\rho^+)^2)\,d\mu\leq  \E_\e(u_\e(t),\rho^+)\leq  \E_\e(u_\e(t),\rho)\,,$$
and \eqref{eqlocminro} is proved.

Next we observe that if $t\in [0,+\infty)\setminus \mathcal{B}_\e$, then $\mu=0$ by Lemma~\ref{convforte}. Hence $\mathcal F(\rho)=\E_\e(u_\e(t),\rho)$ for every 
$\rho\in W^{1,p}(\Om)$ such that $\rho\leq \rho_\e^-(t)$ in $\Om$, and \eqref{superminro} is a consequence of  \eqref{presupermin}. 
From the $\Gamma$-convergence of $\mathcal{F}_k$ to $\mathcal{F}$ we also have 
$\min \mathcal F_k\to \min \mathcal F$, and thus 
$$\E_\e(u_k(t),\rho_k(t))\mathop{\longrightarrow}\limits_{k\to\infty} \E_\e(u_\e(t),\rho_\e(t))\,, $$
and the strong convergence in $W^{1,p}(\Om)$ of $\rho_k(t)$ follows from Lemma \ref{existsch1}.
\end{proof}

\subsection{Energy inequality}\label{subee}

We are now in position to establish the Lyapunov type inequality between almost every two arbitrary times. The argument below is inspired from  \cite[Proposition 3]{CD}. It formally consists in taking $u'_\e$ as test function in the variational formulation of \eqref{parabolic}. Since we do not have enough time regularity, we will make this argument rigorous by working at the time-discrete level and approximating $u'_\e$ by a sequence of smooth functions.

\begin{proposition}\label{energyineq}
For every $s \in [0,+\infty)\setminus \mathcal{B}_\e$ and every $t\geq s$, 
$$\E_\e(u_\e(t),\rho_\e(t))+\int_{s}^{t} \|u^\prime_\e(r)\|^2_{L^2(\Om)}\,dr \leq \E_\e(u_\e(s),\rho_\e(s))\,. $$
\end{proposition}

\begin{proof} 
Let us fix $\phi \in \C^\infty_c(\Om \times (0,T))$ arbitrary. We define for each integers $k\geq 0$ and $i \geq 0$,
$$\phi_k^i(x):=\frac{1}{\delta_k^{i+1}}\int_{t_k^i}^{t_k^{i+1}}\phi(x,r)\, dr\,.$$
At a step $i+1$, we can test the minimality of the pair $(u_k^{i+1},\rho_k^{i+1})$ against the competitor $(u_k^i+\delta_k^{i+1} \phi_k^i,\rho_k^i)$. It yields 
\begin{multline*}
\E_\e(u_k^{i+1},\rho_k^{i+1}) + \frac{1}{2\delta_k^{i+1}}\|u_k^{i+1} - u_k^i\|^2_{L^2(\Om)} \leq \E_\e(u_k^i,\rho_k^i)\\
 + \delta_k^{i+1} \int_\Om \left[(\eta_\e +(\rho_k^i)^2)\nabla u_k^i \cdot \nabla \phi_k^i+\beta(u_k^i-g)\phi_k^i + \frac{(\phi_k^i)^2}{2}\right] dx\\ 
 + (\delta_k^{i+1})^2 \int_\Om \big(|\nabla \phi_k^i|^2+\beta(\phi_k^i)^2\big)\, dx\,.
\end{multline*}
According to Jensen's inequality, we get that
\begin{multline*}
\E_\e(u_k^{i+1},\rho_k^{i+1}) + \frac{1}{2\delta_k^{i+1}}\|u_k^{i+1} - u_k^i\|^2_{L^2(\Om)} \leq \E_\e(u_k^i,\rho_k^i)\\
 + \int_{t_k^i}^{t_k^{i+1}} \int_\Om \left[(\eta_\e +(\rho_k^i)^2)\nabla u_k^i \cdot \nabla \phi + \beta (u_k^i-g)\phi+ \frac{\phi^2}{2}\right] dx\, dr \\
 + |\bdel_k|\int_{t_k^i}^{t_k^{i+1}} \int_\Om \big( |\nabla \phi|^2+\beta\phi^2\big)\, dx\, dr\,.
\end{multline*}
Let us consider the time-shift function $\ell_k:[0,+\infty)\to[0,+\infty)$ defined in \eqref{1117}. Setting  $u^-_k(t):=u_k(\ell_k(t))$ we rewrite the previous inequality as 
\begin{multline}\label{coucou}
\E_\e(u_k^{i+1},\rho_k^{i+1}) +
  \frac12\int_{t_k^i}^{t_k^{i+1}}\|v'_k(r)\|^2_{L^2(\Om)}\, dr \leq \E_\e(u_k^i,\rho_k^i)\\
 + \int_{t_k^i}^{t_k^{i+1}} \int_\Om \left[(\eta_\e +(\rho^-_k)^2)\nabla u^-_k \cdot \nabla \phi +\beta(u_k^--g)\phi+ \frac{\phi^2}{2}\right] dx\, dr\\ 
 + |\bdel_k|\int_{t_k^i}^{t_k^{i+1}} \int_\Om\big( |\nabla \phi|^2+\beta\phi^2\big)\, dx\, dr\,.
\end{multline}
\vskip3pt

Let us now fix $s \in (0,+\infty)\setminus \mathcal{B}_\e$ and $t\geq s$. Given $k$, we consider the two integers $j \geq i\geq 1$ such that $s \in (t_k^{i-1},t_k^i]$ and $t \in (t_k^{j-1},t_k^j]$. Iterating estimate \eqref{coucou}, we are led to 
\begin{multline}\label{coucou2}
\E_\e(u_k(t),\rho_k(t)) +  \frac12\int_{t_k^i}^{t_k^j}\|v'_k(r)\|^2_{L^2(\Om)}\, dr  \leq \E_\e(u_k(s),\rho_k(s))\\
 + \int_{t_k^i}^{t_k^j} \int_\Om \left[(\eta_\e +(\rho^-_k)^2)\nabla u^-_k \cdot \nabla \phi +\beta(u_k^--g)\phi + \frac{\phi^2}{2}\right] dx\, dr\\
  + |\bdel_k|\int_0^T\int_\Om \big(|\nabla \phi|^2+\beta \phi^2\big)\, dx\, dr\,.
\end{multline}
As in estimate \eqref{coucou3}, we have 
$$\|u^-_{k}(r)-u_\e(r)\|_{L^2(\Om)} 
 \leq C\sqrt{|\bdel_{k}|} +  \|v_{k}(r)-u_\e(r)\|_{L^2(\Om)} \mathop{\longrightarrow}\limits_{k\to\infty} 0$$
for every $r\geq0$.  On the other hand, $\sup_{r\geq 0}\|u_k^-(r)\|_{H^1(\Om)}<\infty$ by Lemma \ref{inegenergdiscr}. Therefore $u_k^-(r)\rightharpoonup u_\e(r)$ weakly in $H^1(\Om)$ for every $r\geq 0$. Since  $\rho_\e^-(r)=\rho_\e(r)$ for a.e. $r \geq 0$ by Corollary \ref{rho-}, we infer that $\rho^-_k(r) \wto \rho_\e(r)$ weakly in
$W^{1,p}(\Om)$ for a.e. $r\geq 0$. In particular, $\rho^-_k(r) \to \rho_\e(r)$ uniformly in $\Om$ for a.e. $r \geq 0$. Using those convergences, the uniform bound $0\leq \rho^-_k\leq 1$, Lemma \ref{existsch1}, Lemma \ref{convforte}, and Proposition \ref{minrhoeps}, we can pass to the limit in $k$ in inequality~\eqref{coucou2} (invoking Lebesgue's dominated convergence theorem)  to get 
\begin{multline*}
\E_\e(u_\e(t),\rho_\e(t)) +  \frac12\int_{s}^{t}\|u'_\e(r)\|^2_{L^2(\Om)}\, dr  \leq \E_\e(u_\e(s),\rho_\e(s))\\
 + \int_{s}^{t} \int_\Om \left[(\eta_\e +\rho_\e^2)\nabla u_\e \cdot \nabla \phi +\beta(u_\e-g)\phi + \frac{\phi^2}{2}\right] dx\, dr\,.
\end{multline*}
 Using  equation \eqref{parabolic}, we now infer that 
$$\E_\e(u_\e(t),\rho_\e(t)) +  \frac12 \int_{s}^{t}\|u'_\e(r)\|^2_{L^2(\Om)}\, dr  \leq \E_\e(u_\e(s),\rho_\e(s))
 + \int_{s}^{t} \int_\Om \left[-u'_\e \phi + \frac{\phi^2}{2}\right] dx\, dr\,.$$
By density, the previous inequality actually holds for any $\phi \in L^2(0,T;L^2(\Om))$. Choosing $\phi=u'_\e$ yields the announced energy inequality.
\end{proof}

%%%%%%%%%%%%%%%%%%%%%%%%%%%%%%%%%%%%%%%%%%%%%%%%%%%%%%%%%%%%%%%%%%%%%%%%
%%%%%%%%%%%%%%%%%%%%%%%%%%%%%%%%%%%%%%%%%%%%%%%%%%%%%%%%%%%%%%%%%%%%%%%%
                                                                                                                                                                         %%%%%%%%%%%%%%%%%%%%%%%%%% 
\section{Asymptotics for unilateral minimizing movements in the Mumford-Shah limit}\label{sec5}   %%%%%%%%%%%%%%%%%%%%%%%%%%
         																		%%%%%%%%%%%%%%%%%%%%%%%%%%
%%%%%%%%%%%%%%%%%%%%%%%%%%%%%%%%%%%%%%%%%%%%%%%%%%%%%%%%%%%%%%%%%%%%%%%%
%%%%%%%%%%%%%%%%%%%%%%%%%%%%%%%%%%%%%%%%%%%%%%%%%%%%%%%%%%%%%%%%%%%%%%%%

The main goal of this section is to analyse the behavior of a unilateral minimizing movement as $\e$ tends to zero. We prove that in the limit $\e\to 0$, we recover a parabolic type evolution for the Mumford-Shah functional under the irreversible growth constraint on the crack set  similar to \cite{CD}. The result rests on the approximation of the Mumford-Shah functional by the Ambrosio-Tortorelli functional by means of $\Gamma$-convergence proved in \cite{AT,AT2,Foc}. The main result of this section is the following theorem.

\begin{theorem}\label{BabMil}
Let $\e_n\downarrow 0$ be an arbitrary sequence, 
$u_0\in H^1(\Om)\cap L^\infty(\Om)$, and $\rho_0^{\e_n}$ determined by \eqref{rho0}. Let $\{(u_{\e_n},\rho_{\e_n})\}_{n\in\NN}$ be a sequence in  $GUMM(u_0,\rho_0^{\e_n})$. Then there exist a (not relabeled) subsequence  and $u \in AC^2([0,+\infty);L^2(\Om))$ such that
\begin{equation}\label{preconvueps}
\begin{cases}
\rho_{\e_n}(t)\to 1 \text{ strongly in } L^p(\Om) \text{ for every $t\geq 0$}\,,\\
u_{\e_n}(t) \to u(t) \text{ strongly in } L^2(\Om) \text{ Êfor every }t\geq 0\,,\\
u'_{\e_n} \wto u'  \text{ weakly in } L^2(0,+\infty;L^2(\Om))\,.
\end{cases}
\end{equation}
For every $t\geq 0$  the function $u(t)$ belongs to $SBV^2(\Om) \cap L^\infty(\Om)$ with  
\begin{equation}\label{ptwbdu*}
\|u(t)\|_{L^\infty(\Om)}
\leq \max\{\|u_0\|_{L^\infty(\Om)}, \|g\|_{L^\infty(\Om)}\}\,,
\end{equation}
and $\nabla u \in L^\infty(0,+\infty;L^2(\Om;\RR^N))$. Moreover $u$ solves 
$$
\begin{cases}
u' = {\rm div}(\nabla u)-\beta(u-g) & \text{ in }ÊL^2(0,+\infty;L^2(\Om))\,,\\
\nabla u \cdot \nu =0 & \text{ in }L^2(0,+\infty;H^{-1/2}(\partial \Om))\, ,\\
u(0)=u_0\,, 
\end{cases}
$$
and there exists a family of countably $\HH^{N-1}$-rectifiable subsets $\{\Gamma(t)\}_{t \geq 0}$ of $\Omega$ such that 
\begin{enumerate}
\item[(i)] $\Gamma(s)  \subset \Gamma(t)$ for every $0 \leq s \leq t$;
\vskip3pt
\item[(ii)] $J_{u(t)}  \; \widetilde \subset \; \Gamma(t)$ for every $t \geq 0$;
\vskip3pt
\item[(iii)] for every $t \geq 0$,
\begin{multline*}
 \frac12 \int_\Om|\nabla u(t)|^2\, dx + \HH^{N-1}(\Gamma(t)) +\frac{\beta}{2} \int_\Om (u(t)-g)^2\, dx+ \int_0^t \|u'(s)\|_{L^2(\Om)}^2\, ds\\
\leq \frac12 \int_\Om |\nabla u_0|^2\, dx + \frac{\beta}{2} \int_\Om (u_0-g)^2\, dx\,.
\end{multline*}
\end{enumerate}
\end{theorem}
\vskip5pt

This section is thus essentially devoted to the proof of this theorem. To this purpose, we consider for the rest of the section a sequence $\e_n\downarrow0$, and an arbitrary sequence $\{(u_{\e_n},\rho_{\e_n})\}_{n\in\NN}$ in $GUMM(u_0,\rho_0^\e)$.  

\subsection{Compactness and the limiting heat equation}

We start by proving compactness properties for the sequence $\{(u_{\e_n},\rho_{\e_n})\}_{n\in\NN}$.

\begin{proposition}\label{compaciteeps}
There exist a (not relabeled) subsequence $\{u_{\e_n}\}_{n\in\NN}$ and a function $u \in AC^2([0,+\infty);L^2(\Om))$ such that
\eqref{preconvueps} holds. 
In addition, 
$u(t) \in SBV^2(\Om)\cap L^\infty(\Om)$ with \eqref{ptwbdu*} for every $t \geq 0$, and  the mapping $t \mapsto \nabla u(t) \in L^2(\Om;\RR^N)$ is strongly measurable with $\nabla u \in L^\infty(0,+\infty;L^2(\Om;\RR^N))$. 
Moreover, for every $t\geq 0$ and any $0<\delta_1<\delta_2<1$, there exists $s_n=s_n(t,\delta_1,\delta_2) \in (\delta_1, \delta_2)$ such that the set $E_n:= \{\rho_{\e_n}(t)<s_n\}$ has finite perimeter in~$\Om$, $\tilde u_{\e_n}(t):=(1-\chi_{E_n})u_{\e_n}(t) \in SBV^2(\Om)\cap L^\infty(\Om)$, and 
$$
\begin{cases}
\tilde u_{\e_n}(t) \to u(t) \text{ strongly in }L^2(\Om)\,,\\
\tilde u_{\e_n}(t) \wto u(t) \text{ weakly* in }L^\infty(\Om)\,,\\
\nabla \tilde u_{\e_n}(t) \wto \nabla u(t) \text{ weakly } L^2(\Om;\RR^N)\,.
\end{cases}
$$
Finally, for any open subset $A \subset \Om$, 
\begin{equation}\label{SBVconv2}
\begin{cases}
\ds \HH^{N-1}(J_{u(t)} \cap A) \leq \liminf_{n \to \infty}\, \frac{p}{2}\int_A (1-\rho_{\e_n}(t))^{p-1}|\nabla \rho_{\e_n}(t)|\, dx\,,\\[10pt]
\ds \int_A |\nabla u(t)|^2\, dx \leq \liminf_{n \to \infty} \int_A(\eta_{\e_n} + \rho^2_{\e_n}(t))|\nabla u_{\e_n}(t)|^2\, dx\,.
\end{cases}
\end{equation}
\end{proposition}

\begin{proof}{\it Step 1.}
We first derive {\it a priori} estimates from the energy inequality obtained in Proposition \ref{energyineq}.  Indeed according to that result together with the minimality property \eqref{rho0} of $\rho_0^{\e_n}$, we infer that for every $t  \geq 0$,
\begin{multline}\label{nrjineq}
\frac12 \int_\Om (\eta_{\e_n} +\rho^2_{\e_n}(t)) |\nabla u_{\e_n}(t)|^2\, dx+ \frac{\e_n^{p-1}}{p} \int_\Om |\nabla \rho_{\e_n}(t)|^p\, dx+\frac{\alpha}{p'{\e_n}} \int_\Om (1-\rho_{\e_n}(t))^p\, dx \\
 +\int_0^t \|u'_{\e_n}(s)\|_{L^2(\Om)}^2\, ds 
 \leq \E_{\e_n}(u_0,\rho_0^{\e_n}) \leq \E_{\e_n}(u_0,1)\leq \int_\Om |\nabla u_0|^2\, dx+\frac{\beta}{2} \int_\Om (u_0-g)^2\, dx\,.
\end{multline}
Then, applying Young's inequality and using \eqref{alpha}, we obtain 
\begin{equation}\label{Young}
\frac{\e_n ^{p-1}}{p} \int_\Om |\nabla \rho_{\e_n} (t)|^p\, dx +\frac{\alpha}{p'{\e_n} } \int_\Om (1-\rho_{\e_n} (t))^p\, dx \geq \frac{p}{2} \int_\Om (1-\rho_{\e_n} (t))^{p-1}|\nabla \rho_{\e_n} (t)|\, dx\,,
\end{equation}
from which we deduce the following uniform bound
\begin{equation}\label{estimation-energy}
\|u'_{\e_n}\|^2_{L^2(0,+\infty;L^2(\Om))} + \int_\Om (\eta_{\e_n} +\rho^2_{\e_n}(t)) |\nabla u_{\e_n}(t)|^2\, dx + \int_\Om (1-\rho_{\e_n}(t))^{p-1}|\nabla \rho_{\e_n}(t)|\, dx \leq C_0\,,
\end{equation}
for some constant $C_0>0$ independent of ${\e_n}$ and $t$. 
\vskip5pt

\noindent{\it Step 2.} We now establish the weak convergence of $\{u_{\e_n}\}$ and the bound  \eqref{ptwbdu*}. Recalling that $0\leq \rho_{\e_n}\leq 1$, the fact that 
$$\rho_\e(t) \to 1 \quad \text{ strongly in }ÊL^p(\Om) \text{ for every $t\geq 0$}\,,$$
is a direct consequence of \eqref{nrjineq}. According to items {\it (v)} and {\it (ii)} in Lemma~\ref{convuk}, the sequence $\{u_{\e_n}\}$ in uniformly equi-continuous in $L^2(\Om)$, and for each $t \in [0,+\infty)$, the sequence $\{u_{\e_n}(t)\}$ is sequentially weakly relatively compact in $L^2(\Om)$. Therefore, according to Ascoli-Arzela Theorem, we can find a (not relabeled) subsequence and $u \in AC^2([0,+\infty);L^2(\Om))$ such that $u_{\e_n}(t) \wto u(t)$ weakly in $L^2(\Om)$ (and also weakly* in $L^\infty(\Om)$) for every $tÊ\geq 0$, and $u'_{\e_n} \wto u'$ weakly in $L^2(0,+\infty;L^2(\Om))$. In particular, \eqref{ptwbdu*} follows from item {\it (ii)} in Lemma~\ref{convuk}.  

\vskip5pt

\noindent{\it Step 3.} We now examine more accurately the asymptotic behavior of the sequence $\{u_{\e_n}\}$ as in \cite{AT,AT2,Foc}, 
and prove~\eqref{SBVconv2}. 
Let us fix $t\geq 0$, $0<\delta_1<\delta_2<1$ and an arbitrary open subset $A$ of $\Om$. According to the $BV$-coarea formula (see \cite[Theorem 3.40]{AFP}), 
\begin{multline}\label{ee1}
\int_{A} (1-\rho_{\e_n}(t))^{p-1}|\nabla \rho_{\e_n}(t)|\, dx = \int_0^1 (1-s)^{p-1}\HH^{N-1}(\partial^* \{\rho_{\e_n}(t) <s\}\cap A)\, ds\\ \geq \int_{\delta_1}^{\delta_2}(1-s)^{p-1}\HH^{N-1}(\partial^* \{\rho_{\e_n}(t) <s\} Ê\cap A)\, ds\,.
\end{multline}
Consequently, by the mean value theorem there exists some $s_n=s_n(t,\delta_1,\delta_2,A) \in (\delta_1,\delta_2)$ such that
\begin{equation}\label{ee2}
\int_{ A} (1-\rho_{\e_n}(t))^{p-1}|\nabla \rho_{\e_n}(t)|\, dx \geq \frac{\delta_2^p-\delta_1^p}{p} \HH^{N-1}(\partial^* E_n \cap A)\,,
\end{equation}
where $ E_n:=\{ \rho_{\e_n}(t) <s_n \} \cap A$. Note that from \eqref{nrjineq} we have
\begin{equation}\label{measure}
\LL^N(E_n) \leq \frac{1}{(1-s_n)^p} \int_\Om (1-\rho_{\e_n}(t))^p\, dx\leq \frac{C\e_n}{(1-\delta_2)^p} \to 0 \text{ as } n \to \infty\,,
\end{equation}
for some constant $C>0$ independent of $n$. 

Let us define the new sequence 
\begin{equation}\label{utilde}
\tilde u_{\e_n}(t):=(1-\chi_{E_n})u_ {\e_n}(t)\,.
\end{equation}
By \eqref{measure} we have
\begin{equation}\label{ueps-tildeueps}
\|u_{\e_n}(t) - \tilde u_{\e_n}(t)\|_{L^2(A)} \leq \|u_{\e_n}(t)\|_{L^\infty(\Om)}\sqrt{ \LL^N(E_n)} \to 0
\end{equation}
as $n \to \infty$, from which we deduce that $\tilde u_{\e_n}(t) \wto u(t)$ weakly in $L^2(A)$. On the other hand, according to \cite[Theorem 3.84]{AFP} we have $\tilde u_{\e_n}(t) \in SBV^2(A)\cap L^\infty(A)$ with
$$
\begin{cases}
J_{\tilde u_{\e_n}(t)} \; \widetilde \subset \; \partial^*E_n \,,\\[3pt]
\nabla \tilde u_{\e_n}(t) =(1-\chi_{E_n}) \nabla u_{\e_n}(t) \,,\\[3pt]
\|\tilde u_{\e_n}(t)\|_{L^\infty(A)} \leq \max\{\|u_0\|_{L^\infty(\Om)},\|g\|_{L^\infty(\Om)}\} \,.
\end{cases}
$$
By the energy estimate \eqref{estimation-energy} together with \eqref{ee1} and \eqref{ee2},
$$
\|\nabla \tilde u_{\e_n}(t)\|^2_{L^2(A;\RR^N)} \leq \frac{1}{s_n^2}\int_\Om \rho^2_{\e_n}(t)|\nabla u_{\e_n}(t)|^2\, dx \leq \frac{C_0}{\delta_1^2} \,,
$$
and
$$
\HH^{N-1}(J_{\tilde u_{\e_n}(t)}\cap A) \leq \HH^{N-1}(\partial^* E_n \cap A) \leq \frac{C_0p }{\delta_2^p-\delta_1^p}\,.
$$
We are now in position to apply Ambrosio's compactness Theorem in 
$SBV$ (see Theorems 4.7 and 4.8 in~\cite{AFP}) to deduce that $u(t) \in SBV^2(\Om)$ (by arbitrariness of $A$), and that 
$$
\begin{cases}
\tilde u_{\e_n}(t) \to u(t) \text{ strongly in }L^2(A)\,,\\
\tilde u_{\e_n}(t) \wto u(t) \text{ weakly* in }L^\infty(A)\,,\\
\nabla \tilde u_{\e_n}(t) \wto \nabla u(t) \text{ weakly } L^2(A;\RR^N)\,.
\end{cases}
$$
In view of \eqref{ueps-tildeueps} we deduce that $u_{\e_n}(t) \to u(t)$ strongly in $L^2(\Om)$ for each $t\geq 0$ (again by arbitrariness of $A$). Next Proposition \ref{lower-bound-AT} yields
$$2 \HH^{N-1}(J_{u(t)}\cap A) \leq \liminf_{n \to \infty} \HH^{N-1}(\partial^* E_n\cap A)\,.$$
Combining this inequality with \eqref{ee2} we get that
$$(\delta_2^p-\delta_1^p)\HH^{N-1}(J_{u(t)} \cap A) \leq \liminf_{n \to \infty}\, \frac{p}{2}\int_A (1-\rho_{\e_n}(t))^{p-1}|\nabla \rho_{\e_n}(t)|\, dx\,,$$
and the first inequality of \eqref{SBVconv2} follows by letting $\delta_1 \to 0$ and $\delta_2\to 1$.

For what concerns the bulk energy, we have
$$ \int_A(\eta_{\e_n} + \rho^2_{\e_n}(t))|\nabla u_{\e_n}(t)|^2\, dx \geq s_n^2 \int_{A \setminus E_n} |\nabla u_{\e_n}(t)|^2\, dx \geq \delta_1^2\int_{A \setminus E_n} |\nabla u_{\e_n}(t)|^2\, dx\,.$$
Since $\tilde u_{\e_n}(t)=(1-\chi_{E_n})u_{\e_n}(t)$, we have $\nabla \tilde u_{\e_n}(t)=(1-\chi_{E_n})\nabla u_{\e_n}(t) \wto \nabla u(t)$ weakly in $L^2(A;\RR^N)$, and thus
$$\liminf_{n \to \infty} \int_A (\eta_{\e_n} + \rho^2_{\e_n}(t))|\nabla u_{\e_n}(t)|^2\, dx \geq \delta_1^2\liminf_{n \to \infty} \int_A |\nabla \tilde u_{\e_n}(t)|^2\, dx \geq \delta_1^2 \int_A |\nabla u(t)|^2\, dx\,,$$
and the second inequality of \eqref{SBVconv2} follows by letting $\delta_1 \to 1$. 

\vskip5pt

\noindent{\it Step 4.} It now remains to prove the  strong measurability in $L^2(\Om;\RR^N)$ of $t \mapsto \nabla u(t)$, and that $\nabla u \in L^\infty(0,+\infty;L^2(\Om;\RR^N))$. Given $t\geq 0$ and $0<\delta_1<\delta_2<1$ arbitrary, let us consider as in Step 3 the set $E_n$ and the function $\tilde u_{\e_n}(t) \in SBV^2(\Om)$ given by \eqref{utilde} with $A=\Om$. Then,  
$$(\eta_{\e_n}+\rho^2_{\e_n}(t)) (1-\chi_{E_n}) \nabla u_{\e_n}(t) = (\eta_{\e_n}+\rho^2_{\e_n}(t))\nabla \tilde u_{\e_n}(t)\,.$$ 
Note that this last sequence is bounded in $L^2(\Om;\RR^N)$. Since $\rho_{\e_n}(t) \to 1$ strongly in $L^p(\Om)$ with 
$0\leq \rho_{\e_n}\leq 1$, and $\nabla \tilde u_{\e_n}(t) \wto \nabla u(t)$ weakly in $L^2(\Om;\RR^N)$, we deduce that 
\begin{equation}\label{eq2}
(\eta_{\e_n}+\rho^2_{\e_n}(t)) (1-\chi_{E_n}) \nabla u_{\e_n}(t) \wto  \nabla u(t) \text{ weakly in }L^2(\Om;\RR^N)\,.
\end{equation}
On the other hand, from the {\it a priori} estimate \eqref{estimation-energy}, the Cauchy-Schwarz inequality, and \eqref{measure}, we infer that for every $\Phi \in \C^\infty_c (\Om;\RR^N)$,
\begin{multline}\label{eq1}
\int_{E_n}(\eta_{\e_n} +\rho^2_{\e_n}(t)) \nabla u_{\e_n}(t) \cdot  \Phi\, dx\\
\leq \|\Phi \|_{L^\infty(\Om;\RR^N)} \left(\int_\Om (\eta_{\e_n} + \rho^2_{\e_n}(t))|\nabla u_{\e_n}(t)|^2\, dx \right)^{1/2}\left( \int_{E_n}(\eta_{\e_n}+\rho^2_{\e_n}(t))\, dx\right)^{1/2}\\
\leq    \|\Phi\|_{L^\infty(\Om;\RR^N)} \sqrt{C_0(1+\eta_{\e_n}) \LL^N(E_n)} \to 0\,.
\end{multline}
By \eqref{estimation-energy} and the boundedness of $\rho_{\e_n}$, the sequence $\{ (\eta_{\e_n}+\rho_{\e_n}(t)^2)  \nabla u_{\e_n}(t)\}$ is thus bounded in $L^2(\Om;\RR^N)$, so that \eqref{eq2} and \eqref{eq1} yield
\begin{equation}\label{weaknablaue}
(\eta_{\e_n}+\rho^2_{\e_n}(t))  \nabla u_{\e_n}(t) \wto  \nabla u(t) \text{ weakly in }L^2(\Om;\RR^N)\,.
\end{equation}

Finally, Lemmas \ref{convrok} and \ref{convuk} ensure that, for each $nÊ\in \NN$, the mappings $t \mapsto (\eta_{\e_n}+\rho_{\e_n}(t)^2)  \nabla u_{\e_n}(t)$ are strongly measurable in $L^2(\Om;\RR^N)$. Hence $t \mapsto  \nabla u(t)$ is weakly measurable in $L^2(\Om;\RR^N)$, and thus strongly measurable owing to Pettis Theorem. The fact that $\nabla u \in L^\infty(0,+\infty;L^2(\Om;\RR^N))$ is a consequence of the second relation in \eqref{SBVconv2} together with the uniform bound \eqref{estimation-energy}.
\end{proof}

Our next goal is to pass to the limit as $\e_n\to 0$ in the inhomogeneous heat equation solved by $u_{\e_n}$. 

\begin{proposition}\label{limitheatequation}
The function $u$ solves the generalized heat equation
$$
\begin{cases}
u' ={\rm div}(\nabla u)-\beta(u-g) & \text{ in }ÊL^2(0,+\infty;L^2(\Om)),\\
\nabla u \cdot \nu =0 & \text{ in }L^2(0,+\infty;H^{-1/2}(\partial \Om)),\\
u(0)=u_0. &
\end{cases}
$$
\end{proposition}

\begin{proof}
By Proposition \ref{convdiv}, $u_{\e_n}$ is the solution of the following variational formulation: for every $T>0$
$$\int_0^{T}\int_\Om \big(u'_{\e_n} \, \phi +(\eta_{\e_n}Ê+\rho^2_{\e_n}) \nabla u_{\e_n} \cdot \nabla \phi+\beta(u_{\e_n}-g)\phi\big)\,dx\, dt=0 \quad \text{ for all }\phi \in L^2(0,T;H^1(\Om))\,.$$
According to Proposition \ref{compaciteeps},  $u'_{\e_n}Ê+\beta(u_{\e_n} -g) \wto u'+\beta(u-g)$ weakly in $L^2(0,T;L^2(\Om))$,  
so that it remains to pass to the limit in the divergence term. Thanks to \eqref{weaknablaue},  for a.e. $t \geq 0$ we have 
$$\int_\Om(\eta_{\e_n} +\rho^2_{\e_n}(t)) \nabla u_{\e_n}(t) \cdot \nabla \phi(t)\, dx  \mathop{\longrightarrow}\limits_{n\to\infty} \int_\Om  \nabla u(t) \cdot \nabla \phi(t)\, dx\,,$$
and by the dominated convergence theorem, we deduce that 
$$\int_0^T \int_\Om(\eta_{\e_n} +\rho_{\e_n}^2) \nabla u_{\e_n} \cdot \nabla \phi\, dx\, dt  \mathop{\longrightarrow}\limits_{n\to\infty} \int_0^T \int_\Om  \nabla u \cdot \nabla \phi\, dx\, dt\,.$$
Hence, passing to the limit as $\e_n\to0$ in the variational formulation yields
$$\int_0^{T}\int_\Om \big(u' \, \phi +\nabla u \cdot \nabla \phi+\beta(u-g)\phi\big)\,dx\, dt=0\quad \text{ for all }\phi \in L^2(0,T;H^1(\Om))\,.$$
Finally, the initial condition $u(0)=u_0$ is a consequence of the fact that $u_{\e_n}(0)=u_0$ together with the strong convergence in $L^2(\Om)$ of $u_{\e_n}(0)$ to $u(0)$.
\end{proof}

\subsection{Limiting crack set and the energy inequality}

Our main goal is now to pass to the limit as $\e_n\to 0$ in the energy inequality established in Proposition~\ref{energyineq}. 
We first notice that Theorem \ref{GconvAT} and Proposition~\ref{compaciteeps} (with $A=\Om$) immediately imply that for every $t \geq 0$,
$$
\E(u(t))=\frac{1}{2}\int_{\Om}|\nabla u(t)|^2\,dx + \HH^{N-1}(J_{u(t)})+\frac{\beta}{2}\int_\Om (u(t)-g)^2\,dx \leq \liminf_{n \to \infty} \E_{\e_n}(u_{\e_n}(t),\rho_{\e_n}(t))\,.
$$
We emphasize that this lower bound only involves the measure of the jump set of $u(t)$.  
It will be later improved in Proposition \ref{finalenergyest} by replacing $J_{u(t)}$ by a countably $\HH^{N-1}$-rectifiable set $\Gamma(t)$ containing $J_{u(t)}$ and increasing with respect to $t$. This monotonicity property of the crack acts as a memory of the irreversibility of the process characterized by the non-increasing property of $t \mapsto \rho_{\e_n}(t)$ together with the non-decreasing property of the diffuse surface energy $\mathfrak S_{\e_n}$ established in Proposition \ref{decenergysurf}.
\vskip5pt

To prove the assertion above, {\it we fix an arbitrary countable dense subset $D$ of $[0,+\infty)$}, and we consider 
for each $t\in D$ and $n\in\NN$ the bounded Radon measure
$$\mu_n(t):= \left( \frac{{\e_n}^{p-1}}{p}|\nabla\rho_{\e_n}(t)|^p +\frac{\alpha}{p'{\e_n}}(1-\rho_{\e_n}(t))^p \right) \LL^N \res\, \Om\,.$$
By the energy inequality \eqref{nrjineq},  we infer that the sequences $\{\mu_n(t)\}_{n\in\NN}$ are uniformly bounded with respect to $t\in D$. 
Then, a standard diagonalization procedure together with the metrizability of bounded subsets of~$\M(\RR^N)$ yields the existence of a subsequence (not relabeled) and a family of bounded non negative Radon measures $\{\mu(t)\}_{t\in D}$ (supported in $\overline\Om$)
such that 
$$\mu_n(t) \wto \mu(t)\quad \text{weakly* in $\M(\RR^N)$ for every $t \in D$}\,.$$
\vskip5pt

We first claim that the mapping $t \in D\mapsto \mu(t)$ inherits the increase of the diffuse surface energy. 

\begin{lemma}\label{mono}
For every $s$ and $t \in D$ with $0 \leq s \leq t$ we have
$$\mu(s) \leq \mu(t)\,.$$
\end{lemma}

\begin{proof}
Let us fix $s$ and $t \in D$ with $0 \leq s \leq t$. Let $B \subset \RR^N$ be an arbitrary Borel set, and $K \subset B \subset A$ where $A$ is open and $K$ is compact. Let us consider a cut-off function $\zeta \in \C^\infty_c(\RR^N;[0,1])$ such that $\zeta=1$ on $K$ and $\zeta=0$ on $\RR^N\setminus A$, and let us define
$$\hat \rho_n:= \zeta \rho_{\e_n}(t) + (1-\zeta) \rho_{\e_n}(s).$$
Note that $\hat \rho_n \in W^{1,p}(\Om)$, and since $t \geq s$, we have $\hat \rho_n \leq \rho_{\e_n}(s)$ in $\Om$. As a consequence of the minimality property established in Proposition \ref{decenergysurf}, we have
 $$
\int_{\Om} \left( \frac{\e_n^{p-1}}{p}|\nabla\rho_{\e_n}(s)|^p +\frac{\alpha}{p'\e_n}(1-\rho_{\e_n}(s))^p\right)\,dx
\leq \int_{\Om} \left(\frac{\e_n^{p-1}}{p}|\nabla \hat\rho_n|^p\, + \frac{\alpha}{p'{\e_n}}(1-\hat \rho_n)^p\right)\,dx.
$$
Since $\nabla \hat \rho_n=\zeta \nabla \rho_{\e_n}(t) + (1-\zeta)\nabla \rho_{\e_n}(s) +  (\rho_{\e_n}(t) - \rho_{\e_n}(s))\nabla \zeta$, there exists a constant $C>0$ (independent of~$n$) such that
\begin{multline*}
\int_\Om |\nabla \hat\rho_n|^p\, dx \leq \int_\Om |\zeta \nabla \rho_{\e_n}(t) + (1-\zeta)\nabla \rho_{\e_n}(s) |^p \\
+ C \int_\Om |\nabla \zeta| (\rho_{\e_n}(s) - \rho_{\e_n}(t))\left(1+ |\nabla \rho_{\e_n}(t)|^{p-1} + |\nabla \rho_{\e_n}(s)|^{p-1} +  |\nabla \zeta|^{p-1} |\rho_{\e_n}(t) - \rho_{\e_n}(s) |^{p-1}\right)\, dx\\
\leq\int_\Om \big(\zeta |\nabla \rho_{\e_n}(t)|^p + (1-\zeta) |\nabla \rho_{\e_n}(s)|^p \big) \, dx + C\big(1+\|\nabla \rho_{\e_n}(t)\|^{p-1}_{L^p(\Om;\RR^N)}+\|\nabla \rho_{\e_n}(s)\|^{p-1}_{L^p(\Om;\RR^N)}\big)\,,
\end{multline*}
where we used H\"older's inequality and the fact that $0\leq \rho_{\e_n}\leq 1$.  Hence, 
\begin{multline}\label{muest}
\mu_n(s)(\RR^N)\leq \int_{\RR^N} \zeta \, d\mu_n(t) + \int_{\RR^N} (1-\zeta) \, d\mu_n(s) + C\e_n^{p-1}\\
+ C\e_n^{(p-1)/p} \big(\| \e_n^{(p-1)/p}\nabla \rho_{\e_n}(t)\|^{p-1}_{L^p(\Om;\RR^N)} +\| \e_n^{(p-1)/p}\nabla \rho_{\e_n}(s)\|^{p-1}_{L^p(\Om;\RR^N)} \big),
\end{multline}
and passing to the limit as $n \to \infty$ yields
$$\mu(s)(\RR^N) \leq  \int_{\RR^N} \zeta \, d\mu(t) - \int_{\RR^N} \zeta\, d\mu(s) + \mu(s)(\RR^N)\,.$$
From this inequality we deduce that 
$$\mu(s)(K) \leq \int_{\RR^N} \zeta\, d\mu(s) \leq \int_{\RR^N} \zeta\, d\mu(t) \leq \mu(t)(A)\,.$$
Taking the supremum among all compact sets $K \subset B$, the infimum among all open sets $A \supset B$, and using the outer-inner regularity of the measures $\mu(s)$ and $\mu(t)$ leads to $\mu(s)(B) \leq \mu(t)(B)$.
\end{proof}

We can now define a family of increasing cracks for times in the countable dense set $D$.

\begin{lemma}\label{hatmu}
There exists a family of countably $\HH^{N-1}$-rectifiable subsets $\{\hat \Gamma(t)\}_{t \in D}$ of $\overline\Omega$ such that
\begin{itemize}
\item $\hat \Gamma(s) \; \widetilde \subset \; \hat \Gamma(t)$ for every $ s \leq t$ with $s$, $t \in D$;
\vskip3pt
\item $J_{u(s)} \; \widetilde \subset \; \hat \Gamma(t)$ for every $t \in D$ and $0\leq s \leq t$;
\vskip3pt
\item $\mu(t) \geq \HH^{N-1} \res \hat \Gamma(t)$ for every $t \in D$.
\end{itemize}
\end{lemma}

\begin{proof}
For $t \in D$, let us define the upper density of $\mu(t)$ at $x$ by
$$\Theta^*(t,x):=\limsup_{r \to 0} \frac{\mu(t)(B_r(x))}{\omega_{N-1}r^{N-1}}=\limsup_{r \to 0} \frac{\mu(t)(\overline B_r(x))}{\omega_{N-1}r^{N-1}}$$
for all $x \in \RR^N$, and the Borel set
$$K(t):=\left\{x \in \RR^N : \Theta^*(t,x) \geq 1\right\}\subset\overline\Om\,.$$
Note that the monotonicity property established in Lemma \ref{mono} ensures that $D \ni t \mapsto \Theta^*(t,x)$ is non-decreasing for every $x \in \overline\Om$. Consequently, 
$$K(s) \subset K(t)\text{ for every }0 \leq s \leq t \text{ with $s,t\in D$}\,.$$
Moreover, from standard properties of densities (see \cite[Theorem 2.56]{AFP}) we infer that for every $t\in D$,  
\begin{equation}\label{compHKmu}
\HH^{N-1}\res K(t) \leq \mu(t)\,.
\end{equation}

Let us now fix $t\in D$ and $s\in[0,t]$ (not necessarily in $D$), and let $A$ and $A' \subset \RR^N$ be open sets such that $\overline A\subset A' $. We consider a cut-off function $\zeta \in \C^\infty_c(\RR^N;[0,1])$ such that $\zeta=1$ on $A\cap\Omega$ and $\zeta=0$ on $\RR^N\setminus A'$.  Arguing exactly as in the proof of Lemma \ref{mono}, we obtain inequality \eqref{muest} from which we deduce that
\begin{equation}\label{1706}
\mu_n(s)(A)\leq \int_{\RR^N} \zeta \, d\mu_n(s) \leq \int_{\RR^N} \zeta \, d\mu_n(t) + C\e_n^{(p-1)/p}\, ,
\end{equation}
for some constant $C>0$ independent of $n$. By \eqref{Young} and \eqref{SBVconv2}, we infer that
$$\liminf_{n \to \infty} \mu_n(s)(A) \geq \HH^{N-1}(J_{u(s)} \cap A)\,.$$
Passing to the limit in \eqref{1706} then leads to
$$ \HH^{N-1}(J_{u(s)} \cap A) \leq \int_{\RR^N} \zeta \, d\mu(t)\leq \mu(t)(A')\,.$$
Taking the infimum with respect to all open sets $A' $ containing $\overline A$ yields 
$$\mu(t)(\overline A) \geq \HH^{N-1}(J_{u(s)}\cap A) \quad\text{ for every open set $A$}\, .$$ 
In particular, since $J_{u(s)}$ is countably $\HH^{N-1}$-rectifiable, we infer from the Besicovitch-Mastrand-Mattila Theorem (see \cite[Theorem 2.63]{AFP}) that $\Theta^*(t,x) \geq 1$ for $\HH^{N-1}$-a.e. $x \in J_{u(s)}$, and hence 
\begin{equation}\label{inclJK}
J_{u(s)} \; \widetilde \subset \; K(t)\quad\text{for every $s\in[0,t]$}\,.
\end{equation}

The Borel sets $\{K(t)\}_{t\in D}$ have all the required properties, except that they might not be countably $\HH^{N-1}$-rectifiable. However, since $\HH^{N-1}(K(t))<+\infty$ by \eqref{compHKmu}, it is possible to decompose each $K(t)$ into the union of a countably $\HH^{N-1}$-rectifiable set $\hat \Gamma(t)$, and a purely $\HH^{N-1}$-unrectifiable set $K(t) \setminus \hat \Gamma(t)$ (see {\it e.g.} \cite[page~83]{AFP}). This decomposition being unique up to $\HH^{N-1}$-negligible sets, and $J_{u(s)}$ being countably $\HH^{N-1}$-rectifiable, we deduce from \eqref{inclJK} that 
$$J_{u(s)} \; \widetilde \subset \; \hat \Gamma(t)\quad\text{for every $t\in D$ and $s\in[0,t]$}\,.$$
Moreover, for $s,t \in D$ with $s\leq t$ we have $\hat \Gamma(s) \subset K(s) \subset K(t)$, and since $\hat \Gamma(s)$ is countably $\HH^{N-1}$-rectifiable we finally conclude that $\hat \Gamma(s) \; \widetilde \subset \; \hat \Gamma(t)$. 
\end{proof}

We now extend our definition of crack set for arbitrary times. We set for each $t \geq 0$,
$$\Gamma(t):= \bigcap_{\tau>t, \, \tau\in D} \hat \Gamma(\tau)\cap\Om\,.$$

\begin{lemma}\label{mu}
For every $t \geq 0$, the set $ \Gamma(t)$ is countably $\HH^{N-1}$-rectifiable, and it satisfies
\begin{itemize}
\item $\Gamma(s)  \subset \Gamma(t)$ for every $0 \leq s \leq t$;
\vskip3pt 
\item $J_{u(t)} \; \widetilde \subset \;  \Gamma(t)$ for every $t \geq 0$.
\end{itemize}
\end{lemma}

\begin{proof} 
Clearly $\{\Gamma(t)\}_{t \geq 0}$ is a family of countably $\HH^{N-1}$-rectifiable sets satisfying $\Gamma(s) \subset \Gamma(t)$ for every $0 \leq s \leq t$. Moreover, for $t \geq 0$ we have
\begin{multline*}
\HH^{N-1}(J_{u(t)} \setminus \Gamma(t)) = \HH^{N-1} \left(J_{u(t)} \setminus \bigcap_{\tau>t, \, \tau \in D} \hat \Gamma(\tau)\right)\\
=  \HH^{N-1} \left( \bigcup_{\tau>t, \, \tau\in D} \big(J_{u(t)} \setminus \hat \Gamma(\tau)\big)\right) \leq \sum_{\tau>t,\; \tau\in D} \HH^{N-1} 
\left( J_{u(t)} \setminus \hat \Gamma(\tau)\right)=0\,,
\end{multline*}
since $J_{u(t)}Ê\; \widetilde\subset\; \hat \Gamma(\tau)$ for all $\tau \in D$ such that $\tau>t$ by Lemma \ref{hatmu}. Consequently, $J_{u(t)}Ê\; \widetilde \subset \; \Gamma(t)$.
\end{proof}

We are now in position to improve the energy inequality by replacing the jump set of $u(t)$  by the increasing family of cracks $\Gamma(t)$ constructed before.

\begin{proposition}\label{finalenergyest}
For every $t \geq 0$,
\begin{multline*}
 \frac12 \int_\Om|\nabla u(t)|^2\, dx + \HH^{N-1}(\Gamma(t)) +\frac{\beta}{2} \int_\Om (u(t)-g)^2\, dx+ \int_0^t \|u'(s)\|_{L^2(\Om)}^2\, ds\\
\leq \frac12 \int_\Om |\nabla u_0|^2\, dx + \frac{\beta}{2} \int_\Om (u_0-g)^2\, dx\,.
\end{multline*}
\end{proposition}

\begin{proof}
{\it Step 1.} We first consider the case $t \in D$. According to the energy inequality in Proposition \ref{energyineq}, we have 
$$\frac12\int_\Om (\eta_{\e_n} + \rho^2_{\e_n}(t))|\nabla u_{\e_n}(t)|^2\, dx + \mu_n(t)(\RR^N) +\frac{\beta}{2} \int_\Om (u_{\e_n}(t)-g)^2\, dx +
\int_0^t \|u'_{\e_n}(s)\|_{L^2(\Om)}^2 \, ds
\leq \E_{\e_n}(u_0,\rho_0^{\e_n})\,.$$
Since $\mu_n(t) \wto \mu(t)$ weakly* in $\M(\RR^N)$ and $\mu(t) \geq \HH^{N-1} \res\, \hat \Gamma(t)$ by Lemma \ref{hatmu}, we have
$$\liminf_{n \to \infty}Ê\mu_n(t)(\RR^N) \geq \mu(t)(\RR^N) \geq \HH^{N-1}(\hat \Gamma(t))\,.$$
On the other hand the second inequality in \eqref{SBVconv2}  with $A=\Om$ yields
$$\liminf_{n \to \infty}Ê\int_\Om (\eta_{\e_n} + \rho_{\e_n}(t)^2)|\nabla u_{\e_n}(t)|^2\, dx \geq \int_\Om |\nabla u(t)|^2\, dx\, .$$
Therefore, 
\begin{multline*}
\frac12 \int_\Om|\nabla u(t)|^2\, dx + \HH^{N-1}(\hat \Gamma(t))+\frac{\beta}{2} \int_\Om (u(t)-g)^2\, dx+\int_0^t \|u'(s)\|^2_{L^2(\Om)}\, ds\\
 \leq \liminf_{n \to \infty}\left\{ \frac12\int_\Om (\eta_{\e_n} + \rho_{\e_n}(t)^2)|\nabla u_{\e_n}(t)|^2\, dx + \mu_n(t)(\RR^N) +\frac{\beta}{2}  \int_\Om (u_{\e_n}(t)-g)^2\, dx+ \int_0^t \|u'_{\e_n}(s)\|^2_{L^2(\Om)} \, ds\right\}.
\end{multline*}
Then, by the minimality property \eqref{rho0} of $\rho_0^{\e_n}$, we have 
$$\E_{\e_n}(u_0, \rho_0^{\e_n}) \leq \E_{\e_n}(u_0,1) =\frac{1+\eta_{\e_n}}{2} \int_\Om |\nabla u_0|^2\, dx+\frac{\beta}{2} \int_\Om (u_0-g)^2\, dx \to \frac12 \int_\Om |\nabla u_0|^2\, dx+\frac{\beta}{2} \int_\Om (u_0-g)^2\, dx\,,$$
which leads to 
\begin{multline}\label{tinD}
\frac12 \int_\Om|\nabla u(t)|^2\, dx + \HH^{N-1}(\hat \Gamma(t)) +\frac{\beta}{2} \int_\Om (u(t)-g)^2\, dx+ \int_0^t \|u'(s)|^2_{L^2(\Om)}\, ds \\
\leq \frac12 \int_\Om |\nabla u_0|^2\, dx + \frac{\beta}{2} \int_\Om (u_0-g)^2\, dx\,.
\end{multline}
\vskip5pt

\noindent{\it Step 2.} We now extend the inequality above to the case where $t \geq 0$ is arbitrary. In that case, there exists a sequence $\{t_j\} \subset D$ such that $t_j \to t$ with $t_j>t$. By \eqref{tinD} we have 
$$
\sup_{j \in \NN} \big\{\|u(t_j)\|_{L^\infty(\Om)} + \|\nabla u(t_j)\|_{L^2(\Om;\RR^N)}+\HH^{N-1}(J_{u(t_j)})Ê\big\}<\infty\,,
$$
since  $J_{u(t_j)} \; \widetilde \subset\; \hat \Gamma(t_j)$ by Lemma \ref{hatmu}. On the other hand,  $u \in AC^2(0,+\infty;L^2(\Om))$ by Proposition \ref{compaciteeps}, and thus $u(t_j) \to u(t)$ strongly in $L^2(\Om)$. Applying Ambrosio's compactness Theorem (Theorems 4.7 and 4.8 in~\cite{AFP}), we deduce that $\nabla u(t_j) \wto \nabla u(t)$ weakly in $L^2(\Om;\RR^N)$. Since $\Gamma(t) \subset \hat \Gamma(t_j)$ for all $j \in \NN$,
we finally conclude that
\begin{multline*}
\frac12 \int_\Om|\nabla u(t)|^2\, dx + \HH^{N-1}(\Gamma(t)) +\frac{\beta}{2} \int_\Om (u(t)-g)^2\, dx+ \int_0^t \|u'(s)\|^2_{L^2(\Om)}\, ds \\
\leq \liminf_{j \to \infty}Ê\left\{\frac12 \int_\Om|\nabla u(t_j)|^2\, dx + \HH^{N-1}(\hat \Gamma(t_j)) +\frac{\beta}{2} \int_\Om (u(t_j)-g)^2\, dx + \int_0^{t_j} \|u'(s)\|_{L^2(\Om)}^2\, ds
\right\}\,,
\end{multline*}
which, in view of  \eqref{tinD}, completes the proof of the energy inequality.
\end{proof}

\begin{remark}
Note that the limiting bulk and surface energies are continuous at time $t=0$, {\it i.e.},  
$$\lim_{t\downarrow 0}\, \frac{1}{2}\int_\Om |\nabla u(t)|^2\,dx+\frac{\beta}{2} \int_\Om (u(t)-g)^2\, dx = \frac{1}{2}\int_\Om |\nabla u_0|^2\,dx+\frac{\beta}{2} \int_\Om (u_0-g)^2\, dx\,,$$
and 
$$\lim_{t\downarrow 0}\HH^{N-1}(\Gamma(t))=0\,. $$
Indeed, arguing as in the previous proof (Step 2), we obtain that   $\nabla u(t) \wto \nabla u_0$ weakly in $L^2(\Om;\RR^N)$ as $t\to0$. Since $u \in AC^2(0,+\infty;L^2(\Om))$, we then infer from the energy inequality that 
\begin{align*}
\frac12 \int_\Om|\nabla u_0|^2\, dx +\frac{\beta}{2} \int_\Om (u_0-g)^2\, dx & \geq  
  \limsup_{t \to 0}Ê\left\{\frac12 \int_\Om|\nabla u(t)|^2\, dx + \HH^{N-1}( \Gamma(t)) +\frac{\beta}{2} \int_\Om (u(t)-g)^2\, dx 
\right\} \\
& \geq  \limsup_{t \to 0}Ê\left\{\frac12 \int_\Om|\nabla u(t)|^2\, dx +\frac{\beta}{2} \int_\Om (u(t)-g)^2\, dx 
\right\} \\
& \geq \liminf_{t \to 0}Ê\left\{\frac12 \int_\Om|\nabla u(t)|^2\, dx +\frac{\beta}{2} \int_\Om (u(t)-g)^2\, dx 
\right\} \\
& \geq \frac12 \int_\Om|\nabla u_0|^2\, dx +\frac{\beta}{2} \int_\Om (u_0-g)^2\, dx  \,,
\end{align*}
and the conclusion follows. 
\end{remark}

\section{Curves of maximal unilateral slope for the Ambrosio-Tortorelli functional}\label{sec6}

In the spirit of \cite{AGS}, we introduce in this section the notion of $L^2(\Om)$-unilateral gradient flow for the Ambrosio-Tortorelli functional in terms of {\it curves of maximal unilateral slope}, accounting for the quasi-stationnarity and the decrease  constraint on the phase field variable $\rho$. To this aim, we first  define the unilateral slope of $\E_\e$ by analogy with  the unilateral slope of  the Mumford-Shah functional, see \cite{DMT}. Then we prove that (generalized) unilateral minimizing movements provide specific examples of curves of maximal unilateral slope. We conclude this section with a preliminary step toward the asymptotic behavior of the unilateral slope as $\e\to 0$, and 
a discussion on the related results of \cite{DMT}.  

\subsection{The maximal unilateral slope}

\begin{definition}\label{1739}
The {\it unilateral slope} of $\E_\e$ at $(u,\rho)\in H^1(\Om)\times W^{1,p}(\Om)$ is defined by
$$|\partial \E_\e|(u,\rho):=\limsup_{v\to u \text{ in }ÊL^2(\Om)} \; \sup_{\hat \rho}\left\{ \frac{\big(\E_\e(u,\rho)-\E_\e(v,\hat \rho)\big)^+}{\|v-u\|_{L^2(\Om)}} : \hat \rho \in W^{1,p}(\Om), \; \hat \rho \leq \rho \text{ in }Ê\Om\right\}\,.$$
The functional $|\partial \E_\e|$ is then extended to $L^2(\Om)\times L^p(\Om)$ by setting 
$|\partial \E_\e|(u,\rho):=+\infty$ for $(u,\rho)\not\in H^1(\Om)\times W^{1,p}(\Om)$. 
\end{definition}

The unilateral slope being defined, we can now define curves of maximal unilateral slope for the Ambrosio-Tortorelli functional. 

\begin{definition}\label{cms2}
We say that a pair $(u,\rho):(a,b)\to L^2(\Om)\times L^p(\Om)$  is a {\it curve of maximal unilateral slope} for $\E_\e$  if 
$u\in AC^2(a,b;L^2(\Om))$, $\rho$ is non-increasing, and if there exists a non-increasing function $\lambda:(a,b) \to [0,+\infty)$ such that for a.e. $t\in (a,b)$, $\E_\e(u(t),\rho(t))=\lambda(t)$, and
\begin{equation}\label{lambda}
\lambda'(t)\leq -\frac{1}{2}\|u'(t)\|^2_{L^2(\Om)} -\frac{1}{2}|\partial \E_\e|^2\big(u(t),\rho(t)\big)\,. 
\end{equation}
\end{definition}

This definition is motivated by the following proposition which parallels \cite[Theorem 1.2.5]{AGS}.

\begin{proposition}\label{prop:identmodvit}
If $(u,\rho):(a,b)\to L^2(\Om)\times L^p(\Om)$  is a curve of maximal unilateral slope for $\E_\e$, then 
\begin{equation}\label{identmodvit}
\|u'(t)\|_{L^2(\Om)}=|\partial \E_\e|(u(t),\rho(t)) \quad \text{for a.e. $t\in (a,b)$}\,.
\end{equation}
Moreover, if $t\mapsto \E_\e(u(t),\rho(t))$ is absolutely continuous on $(a,b)$, then 
$$\E_\e\big(u(t),\rho(t))+\int_s^t\|u'(r)\|^2_{L^2(\Om)}\,dr = \E_\e\big(u(s),\rho(s)\big)\quad\text{for every $s$ and $t\in (a,b)$ with $s\leq t$}\,. $$
\end{proposition}

\begin{proof}
Let $\lambda$ be as in Definition \ref{cms2}. Since $\lambda$ is non-increasing, $\lambda$ has finite pointwise variation in $(a,b)$. Let us consider the set
$$A:=\big\{t \in (a,b) : \E_\e(u(t),\rho(t))=\lambda(t), \; \lambda \text{ and } u \text{ are derivable at } t\big\}\,,$$
and observe that $\LL^1((a,b) \setminus A)=0$. 

Let $t \in A$. Since $\lambda$ is non-increasing, we have $\lambda'(t)\leq 0$, and thus
$$|\lambda'(t)|=-\lambda'(t) = \lim_{s \downarrow t, \; s \in A}\frac{\lambda(t)-\lambda(s)}{s-t}
= \lim_{s \downarrow t, \; s \in A} \frac{\E_\e(u(t),\rho(t))-\E_\e(u(s),\rho(s))}{s-t}\,.
$$
Using the fact that $\rho(s)\leq \rho(t)$ when $s>t$ (by the non-increasing property of $t \mapsto \rho(t)$) and the strong $L^2(\Om)$-continuity of $u$, we infer that 
\begin{multline*}
|\lambda'(t)| \leq \limsup_{s \downarrow t\; s \in A} \sup_{\hat \rho \leq \rho(t)} \frac{\big(\E_\e(u(t),\rho(t))-\E_\e(u(s),\hat \rho)\big)^+}{\|u(s)-u(t)\|_{L^2(\Om)}} \frac{\|u(s)-u(t)\|_{L^2(\Om)}}{s-t}
\leq |\partial \E_\e|\big(u(t),\rho(t)\big) \|u'(t)\|_{L^2(\Om)}\,.
\end{multline*}
On the other hand, $|\lambda'(t)|\geq \frac{1}{2} \|u'(t)\|^2_{L^2(\Om)}+\frac{1}{2}|\partial \E_\e|^2\big(u(t),\rho(t)\big)$ by \eqref{lambda}, and \eqref{identmodvit} follows as well as the fact that  
$\lambda'(t)=-\|u'(t)\|^2_{L^2(\Om)}$. 
\vskip3pt
Finally, if $t\mapsto \E_\e(u(t),\rho(t))$ is absolutely continuous on $(a,b)$, then for every $s,t\in (a,b)$ with $s\leq t$, 
$$\E_\e(u(t),\rho(t))-\E_\e(u(s),\rho(s))=\int_{s}^t \lambda'(r)\,dr=-\int_s^t\|u'(r)\|^2_{L^2(\Om)}\,dr\,,$$
which completes the proof of the proposition.
\end{proof}

We state below necessary and sufficient conditions for the finiteness of the slope, as well as an explicit formula to represent it. 
\begin{proposition}\label{slope+domain}
Assume that  $\Om$ has a $\C^{1,1}$-boundary, and 
let  $ D(|\partial \E_\e|)$ be the proper domain of  $|\partial \E_\e|$. Then,
 \begin{multline}\label{domainslip}
 D(|\partial \E_\e|)=\Big\{(u,\rho)\in H^2(\Om)\times W^{1,p}(\Om): \frac{\partial u}{\partial \nu}=0 \text{ in }H^{1/2}(\partial\Om)\,,\text{ and}  \\
 \E_\e(u,\rho)\leq \E_\e(u,\hat \rho) \;\text{ for all $\hat \rho\in W^{1,p}(\Om)$ such that $\hat\rho\leq \rho$  in $\Om$}\Big\}\,.
 \end{multline} 
In addition, for $(u,\rho)\in  D(|\partial \E_\e|)$,
$$|\partial \E_\e|(u,\rho)=\big\|{\rm div}((\eta_\e+\rho^2)\nabla u)-\beta(u-g)\big\|_{L^2(\Om)} \,,$$
and 
$$
\|u\|_{H^2(\Om)}\leq C_\e (1+\|\nabla\rho\|_{L^p(\Om;\RR^N)})^{\gamma} \big(|\partial \E_\e|(u,\rho)+ \beta\|u-g\|_{L^2(\Om)}+\|u\|_{H^1(\Om)}\big)\,,
$$
where $\gamma\in\mathbb{N}$ is the smallest integer larger than or equal to $p/(p-N)$, and $C_\e$ only depends on $\eta_\e$, $p$, $N$, and~$\Om$.   
\end{proposition}

\begin{proof}
{\it Step 1.} Let us consider a pair $(u,\rho)$ such that $|\partial \E_\e|(u,\rho)<\infty$. For $\varphi\in H^1(\Om)$ with $\varphi\not=0$, we estimate 
\begin{equation}\label{firstlwbdslip}
 |\partial \E_\e|(u,\rho)\geq \limsup_{\delta\downarrow 0} \frac{\E_\e(u,\rho)-\E_\e(u-\delta\varphi,\rho)}{\delta \|\varphi\|_{L^2(\Om)}}\geq \frac{1}{\|\varphi\|_{L^2(\Om)}}\bigg( \int_\Om (\eta_\e+\rho^2)\nabla u\cdot \nabla\varphi\,dx+\beta\int_\Om (u-g)\varphi\,dx\bigg)\,.
 \end{equation}
By density of $H^1(\Om)$ in $L^2(\Om)$ and 
the Riesz representation Theorem in $L^2(\Om)$, we deduce that there exists $\tilde f\in L^2(\Om)$ such that 
$$
\int_\Om (\eta_\e+\rho^2)\nabla u\cdot \nabla\varphi\,dx+\beta\int_\Om (u-g)\varphi\,dx=\int_\Om \tilde f\varphi\,dx
$$
for all $\varphi\in H^1(\Om)$. Hence $u$ solves \eqref{eqellipt} with $f=\tilde f-\beta(u-g)$. We then infer from Lemma~\ref{ellipreg} that $u\in H^2(\Om)$, 
and that $\frac{\partial u}{\partial \nu}=0$ in $H^{1/2}(\partial\Om)$. Next, taking $\varphi \in H^1(\Om)$ such that $\|\varphi\|_{L^2(\Om)}=1$, integrating by parts in \eqref{firstlwbdslip}, and passing to the supremum over all such $\varphi$'s yields the  lower bound 
$$|\partial \E_\e|(u,\rho)\geq\big\|{\rm div}((\eta_\e+\rho^2)\nabla u)-\beta(u-g)\big\|_{L^2(\Om)}\,.$$

\vskip5pt

We now claim that the following minimality property for $\rho$ holds: 
\begin{equation}\label{minrhoseccurv}
\E_\e(u,\rho)\leq \E_\e(u,\hat \rho) \;\text{ for all $\hat \rho\in W^{1,p}(\Om)$ such that $\hat\rho\leq \rho$  in $\Om$}\,.
\end{equation}
Since $ |\partial \E_\e|(u,\rho)<+\infty$ we can find sequences $\{v_n\} \subset H^1(\Om)$ and $\{\rho_n\} \subset W^{1,p}(\Om)$ such that 
$v_n\to u$ strongly in $L^2(\Om)$, 
$$\rho_n ={\rm argmin}\left\{ \E_\e(v_n,\hat \rho) : \hat \rho \in W^{1,p}(\Om),\, \hat \rho \leq \rho \text{ in }Ê\Om \right\}\quad\text{for each $n \in \NN$}\,,$$
\begin{equation}\label{cond2}
\E_\e(v_n,\rho_n) \leq \E_\e(u,\rho)\,,
\end{equation}
and
\begin{equation}\label{cond3}
\limsup_{n \to \infty}Ê\frac{\E_\e(u,\rho) - \E_\e(v_n,\rho_n)}{\|v_n-u\|_{L^2(\Om)}} \leq |\partial \E_\e|(u,\rho)\,.
\end{equation}
By \eqref{cond2} the sequence $\{\nabla v_n\}$ is uniformly bounded in $L^2(\Om;\RR^N)$. Hence, for a  suitable subsequence (not relabeled), 
$$|\nabla v_n|^2 \mathscr L^N\res \, \Om \wto |\nabla u|^2\mathscr L^N \res \, \Om+ \mu$$ 
weakly* in $\M(\RR^N)$ for some nonnegative Radon measure $\mu \in \M(\RR^N)$ supported in $\overline\Om$. Let us now consider the following functionals on $W^{1,p}(\Om)$ defined by 
$$
\mathcal F_n(\hat \rho):=
\begin{cases}
\E_\e(v_n,\hat \rho) & \text{ if }\hat \rho \leq \rho\,,\\
+\infty & \text{ otherwise}\,,
\end{cases}
\;\text{ and }\;
\mathcal F(\hat \rho):=
\begin{cases}
\ds \E_\e(u,\hat \rho) +\frac12 \int_{\overline \Om} (\eta_\e + \hat \rho^2)\, d\mu& \text{ if }\hat \rho \leq \rho\,,\\
+\infty & \text{ otherwise}\,.
\end{cases}
$$
A similar argument to the one used in the proof of Proposition \ref{minrhoeps} (Step 1) ensures that $\mathcal F_n$ $\Gamma$-converges to $\mathcal F$ for the sequential weak $W^{1,p}(\Om)$-topology.
Since the sublevel sets of $\mathcal F_n$ are  relatively compact for the sequential weak $W^{1,p}(\Om)$-topology (uniformly in~$n$), we infer from the $\Gamma$-convergence of $\mathcal F_n$ 
towards $\mathcal F$ that
$$\E_\e(v_n,\rho_n) = \min_{W^{1,p}(\Om)} \mathcal F_n \to\min_{W^{1,p}(\Om)} \mathcal F\,.$$
On the other hand, by \eqref{cond2} and \eqref{cond3} we have $\E_\e(v_n,\rho_n) \to \E_\e(u,\rho)$ from which we deduce that 
$$
\E_\e(u,\rho)=\min_{W^{1,p}(\Om)} \mathcal F=\min\left\{ \mathcal F(\hat \rho) : \hat \rho \in W^{1,p}(\Om), \, \hat \rho \leq \rho \text{ in }Ê\Om\right\}\,.
$$
We conclude from this last relation that $\mu=0$ and that \eqref{minrhoseccurv} holds.

\vskip5pt

\noindent{\it Step 2.} Conversely, we show that if a pair $(u,\rho)$ belongs to the set in the right hand side of \eqref{domainslip}, then $|\partial \E_\e|(u,\rho)<\infty$ and $ |\partial \E_\e|(u,\rho)\leq\big\|{\rm div}((\eta_\e+\rho^2)\nabla u)-\beta(u-g)\big\|_{L^2(\Om)}$. 

Consider a pair $(u,\rho)\in H^2(\Om)\times W^{1,p}(\Om)$ satisfying $\frac{\partial u}{\partial \nu}=0$ in $H^{1/2}(\partial\Om)$ and
$$\E_\e(u,\rho)\leq \E_\e(u,\hat \rho)$$
for all $\hat \rho\in W^{1,p}(\Om)$ such that $\hat\rho\leq \rho$ in $\Om$. Note that  $u\in W^{1,r}(\Om)$ for every $r\leq 2^*$ by the Sobolev Imbedding, and since $p>N$, the product $\nabla u \cdot \nabla \rho$  belongs to $L^2(\Om)$ and $\rho\in L^\infty(\Om)$. Hence, 
$${\rm div}((\eta_\e+\rho^2)\nabla u) = (\eta_\e+\rho^2) \Delta u + 2\rho \nabla \rho \cdot \nabla u \in L^2(\Om),$$
and consequently, it is enough to check that
$$|\partial \E_\e|(u,\rho)\leq\big\|{\rm div}((\eta_\e+\rho^2)\nabla u)-\beta(u-g)\big\|_{L^2(\Om)}.$$
Consider a sequence $\{v_n\}Ê\subset H^1(\Om)$ converging strongly to $u$ in $L^2(\Om)$ such that
$$|\partial \E_\e|(u,\rho) = \lim_{n \to \infty}\,\sup\left\{ \frac{ \big(\E_\e(u,\rho) - \E_\e(v_n,\hat \rho) \big)^+}{\|v_n - u\|_{L^2(\Om)}} : \hat \rho \in W^{1,p}(\Om) ,\; \hat \rho \leq \rho \text{ in }Ê\Om\right\}\, ,$$
and let 
$$\rho_n={\rm argmin}Ê\left\{ \E_\e(v_n,\hat \rho) :  \hat \rho\in W^{1,p}(\Om) \text{ such that }\hat\rho\leq \rho\text{ in }\Om\right\}.$$
Then 
$$\sup_{\hat \rho \leq \rho} \big(\E_\e(u,\rho) - \E_\e(v_n,\hat\rho) \big)^+ \leq \big( \E_\e(u,\rho) - \E_\e(v_n,\rho_n) \big)^+\,,$$
so that 
\begin{equation}\label{1700}
|\partial \E_\e|(u,\rho) = \lim_{n \to \infty}   \frac{ \big(\E_\e(u,\rho) - \E_\e(v_n,\rho_n) \big)^+}{\|v_n - u\|_{L^2(\Om)}}\, .
\end{equation}
If for infinitely many $n$'s we have $\E_\e(v_n ,\rho_n) > \E_\e(u,\rho)$, then $|\partial \E_\e|(u,\rho) = 0$ and there is nothing to prove.  Hence we can assume without loss of generality that $\E_\e(v_n ,\rho_n) \leq \E_\e(u,\rho)$. In particular, $\{\rho_n\}$ is uniformly bounded in $W^{1,p}(\Om)$, and $\{v_n\}$ is uniformly bounded in $H^1(\Om)$. As a consequence, for a subsequence $v_n \wto u$ weakly in $H^1(\Om)$ and $\rho_n \wto  \rho_*$ weakly in $W^{1,p}(\Om)$. From Lemma \ref{existsch1} we infer that $ \rho_* \leq \rho$  in $\Om$, and
\begin{equation}\label{lsccaractslop}
\E_\e(u, \rho_*) \leq \liminf_{n \to \infty}\E_\e(v_n,\rho_n)  \leq \limsup_{n \to \infty}\E_\e(v_n,\rho_n) \leq \E_\e(u,\rho)\,.
\end{equation}
By the minimality property of $\rho$, we have that $\E_\e(u,\rho) \leq \E_\e(u, \rho_*)$ which leads to $\E_\e(u,\rho) = \E_\e(u, \rho_*)$. By uniqueness of the minimizer (due to the strict convexity of $\E_\e(u,\cdot)$), we deduce that $ \rho_*=\rho$. Then Lemma~\ref{existsch1} and \eqref{lsccaractslop} with $ \rho_*=\rho$ shows that $\rho_n \to \rho$ strongly in $W^{1,p}(\Om)$.

We now estimate
\begin{multline*}
\E_\e(u,\rho) - \E_\e(v_n,\rho_n) \leq \E_\e(u,\rho_n) - \E_\e(v_n,\rho_n) \\
\leq \int_\Om (\eta_\e+\rho_n^2) \nabla u \cdot (\nabla u - \nabla v_n)\, dx +\beta \int_\Om (u-g)(u-v_n)\, dx\\
= -\int_\Om (u-v_n)\big( {\rm div}((\eta_\e +\rho_n^2)\nabla u) -\beta(u-g)\big) \, dx\,.
\end{multline*}
Note that in the last equality, there is no boundary term since $\frac{\partial u}{\partial \nu}=0$ in $H^{1/2}(\partial \Om)$. Moreover, since $u \in H^2(\Om)$ and
$\rho_n \in W^{1,p}(\Om)$, we have ${\rm div}((\eta_\e +\rho_n^2)\nabla u) \in L^2(\Om)$. Applying Cauchy-Schwarz Inequality we obtain 
$$\frac{\E_\e(u,\rho) - \E_\e(v_n,\rho_n) }{\|v_n - u\|_{L^2(\Om)}} \leq \|{\rm div}((\eta_\e +\rho_n^2)\nabla u)-\beta(u-g) \|_{L^2(\Om)}\,.$$
Since $H^2(\Om)\hookrightarrow W^{1,r}(\Om)$ for every $r\leq 2^*$ and $\rho_n\to \rho$ strongly in $W^{1,p}(\Om)$, we get that 
$${\rm div}((\eta_\e+\rho_n^2)\nabla u) = (\eta_\e+\rho_n^2) \Delta u + 2\rho_n \nabla \rho_n \cdot \nabla u
\mathop{\longrightarrow}\limits_{n\to\infty}  (\eta_\e+\rho^2) \Delta u + 2\rho \nabla \rho \cdot \nabla u ={\rm div}((\eta_\e+\rho^2)\nabla u)
$$
strongly in $L^2(\Om)$. Hence
$$\lim_{n \to \infty}\frac{\E_\e(u,\rho) - \E_\e(v_n,\rho_n) }{\|v_n - u\|_{L^2(\Om)}} \leq \|{\rm div}((\eta_\e +\rho^2)\nabla u)-\beta(u-g) \|_{L^2(\Om)}.$$
Together with \eqref{1700}, this last estimate gives the desired upper bound for the slope $|\partial \E_\e|(u,\rho)$.

\vskip5pt

\noindent{\it Step 3.} Let $(u,\rho)\in D(|\partial \E_\e|)$. By the previous steps, $\tilde f:=-{\rm div}((\eta_\e +\rho^2)\nabla u)+\beta(u-g)\in L^{2}(\Om)$ and  
$u$ solves \eqref{eqellipt} with $f=\tilde f -\beta(u-g)$. Applying Lemma \ref{ellipreg} we find that 
\begin{align*}
\|u\|_{H^2(\Om)} & \leq  C_\e (1+\|\nabla\rho\|_{L^p(\Om;\RR^N)})^{\gamma} \big(\|f\|_{L^2(\Om)}+\|u\|_{H^1(\Om)}\big) \\
& \leq  C_\e (1+\|\nabla\rho\|_{L^p(\Om;\RR^N)})^{\gamma} \big(|\partial \E_\e|(u,\rho)+ \beta\|u-g\|_{L^2(\Om)}+\|u\|_{H^1(\Om)}\big)\,,
\end{align*}
and the proof is complete. 
\end{proof}

The expression of the slope and the characterization of its domain provided by Proposition \ref{slope+domain} enables one to show the lower semicontinuity  of $|\partial\E_\e|$ along sequences with uniformly bounded energy.

\begin{proposition}\label{scislope}
Assume that  $\Om$ has a $\C^{1,1}$-boundary. Let $\{(u_n,\rho_n)\}_{n\in\mathbb{N}}\subset L^2(\Om)\times L^p(\Om)$ be such that 
$\sup_{n\in \NN} \E_\e(u_n,\rho_n)<\infty$ and $(u_n,\rho_n)\to (u,\rho)$ strongly in $L^2(\Om)\times L^p(\Om)$. Then, 
$$|\partial \E_\e|(u,\rho)\leq \liminf_{n\to\infty} |\partial \E_\e|(u_n,\rho_n)\,.$$
\end{proposition}

\begin{proof}
Let us assume without loss of generality that $\liminf_n |\partial \E_\e|(u_n,\rho_n)<\infty$, and extract a subsequence $\{n_k\}$ such that
$$\liminf_{n\to\infty} |\partial \E_\e|(u_n,\rho_n)=\lim_{k\to\infty} |\partial \E_\e|(u_{n_k},\rho_{n_k})\,.$$
Since $\E_\e(u_{n_k},\rho_{n_k})$ is uniformly bounded with respect to $k$, we deduce that the sequence $\{(u_{n_k},\rho_{n_k})\}$ is uniformly bounded in $H^1(\Om)\times W^{1,p}(\Om)$. Moreover $(u_{n_k},\rho_{n_k}) \in D(|\partial \E_\e|)$, and as a consequence of Proposition \ref{slope+domain}, we deduce that $\{u_{n_k}\}$ is uniformly bounded in $H^2(\Om)$, and that $\frac{\partial u_{n_k}}{\partial \nu}=0$ in $H^{1/2}(\partial\Om)$. Whence $\rho_{n_k}Ê\wto \rho$ weakly in $W^{1,p}(\Om)$, $u_{n_k} \wto u$ weakly in $H^2(\Om)$ for a (not relabeled) subsequence, and $\frac{\partial u}{\partial \nu}=0$ in $H^{1/2}(\partial\Om)$. By the Sobolev Imbedding we get that $\rho_{n_k} \to \rho$ in $\mathscr{C}^0(\overline \Om)$, while $u_{n_k}Ê\to u$ strongly in $H^1(\Om)$.  
Thanks to the uniform convergence of $\rho_{n_k}$ to $\rho$, we may  argue as in the proof of  Proposition \ref{minrhoeps} (Step 1) to show that the  sequence of functionals $\mathcal F_k : W^{1,p}(\Om) \to [0,+\infty]$ defined by
\begin{equation}\label{defF_k}
\mathcal F_k(\hat \rho):=
\begin{cases}
\E_\e(u_{n_k},\hat \rho) & \text{ if }\hat \rho \leq \rho_{n_k}\,,\\
+\infty & \text{ otherwise}\,,
\end{cases}
\end{equation}
$\Gamma$-converges (with respect to the sequential weak $W^{1,p}(\Om)$-topology) to the functional $\mathcal F : W^{1,p}(\Om) \to [0,+\infty]$ given by
\begin{equation}\label{calF}
\mathcal F(\hat \rho):=
\begin{cases}
\E_\e(u,\hat \rho) & \text{ if }\hat \rho \leq \rho\,,\\
+\infty & \text{ otherwise}\,.
\end{cases}
\end{equation}

Since 
$$\rho_{n_k} = \mathop{\rm argmin}\limits_{\hat \rho \in W^{1,p}(\Om)}\mathcal F_k(\hat \rho)\,,$$
and $\rho_{n_k} \wto \rho$ weakly in $W^{1,p}(\Om)$, we infer from the $\Gamma$-convergence of $\mathcal F_k$ toward $\mathcal F$ that 
$$\rho = \mathop{\rm argmin}\limits_{\hat \rho \in W^{1,p}(\Om)}\mathcal F(\hat \rho)\,.$$
By the expression of the domain of the slope provided by Proposition \ref{slope+domain}, we infer that $(u,\rho) \in D(|\partial \E_\e|)$.  From the established convergences of $(u_{n_k},\rho_{n_k})$ we deduce that
\begin{multline*}
{\rm div} ((\eta_\e+\rho_{n_k}^2)\nabla u_{n_k}) = (\eta_\e + \rho_{n_k}^2) \Delta u_{n_k} + 2\rho_{n_k}Ê\nabla \rho_{n_k}Ê\cdot \nabla u_{n_k} \\
\wto (\eta_\e + \rho^2) \Delta u + 2\rhoÊ\nabla \rhoÊ\cdot \nabla u ={\rm div} ((\eta_\e+\rho^2)\nabla u)
\end{multline*}
weakly in $L^2(\Om)$. Using now the expression of the slope given by Proposition \ref{slope+domain}, we conclude 
\begin{multline*}
\liminf_{n\to\infty} |\partial \E_\e|(u_n,\rho_n)=\lim_{k\to\infty} |\partial \E_\e|(u_{n_k},\rho_{n_k})\\
=\lim_{k\to\infty} \| {\rm div} ((\eta_\e+\rho_{n_k}^2)\nabla u_{n_k}) - \beta (u_{n_k} -g)\|_{L^2(\Om)}\\
\geq \| {\rm div} ((\eta_\e+\rho^2)\nabla u) - \beta (u -g)\|_{L^2(\Om)}=|\partial \E_\e|(u,\rho)\,,
\end{multline*}
which ends the proof. 
\end{proof}

For completeness (and possible future investigations), we finally prove that the energy is continuous along convergent sequences with uniformly bounded slope.

\begin{proposition}
Assume that  $\Om$ has a $\C^{1,1}$-boundary. Let $\{(u_n,\rho_n)\}_{n\in\mathbb{N}}\subset L^2(\Om)\times L^p(\Om)$ be such that 
$$\sup_{n\in \NN} \{\E_\e(u_n,\rho_n)+ |\partial \E_\e|(u_n,\rho_n)\}<\infty\,,$$
and $(u_n,\rho_n)\to (u,\rho)$ strongly in $L^2(\Om)\times L^p(\Om)$. Then $\E_\e(u_n,\rho_n)\to \E_\e(u,\rho)$ as $n\to\infty$. 
\end{proposition}

\begin{proof}
Arguing as in the proof of Proposition \ref{scislope}, we have $u_n \wto u$ weakly in $H^2(\Om)$ and $\rho_n \wto \rho$ weakly in $W^{1,p}(\Om)$ with $(u,\rho) \in D(|\partial \E_\e|)$. By the Sobolev Imbedding, $\rho_n \to \rho$  in $\mathscr{C}^0(\overline \Om)$ and $u_n \to u$ strongly in $H^1(\Om)$. Hence the functional $\mathcal F_n : W^{1,p}(\Om) \to [0,+\infty]$ defined by \eqref{defF_k} (with $n$ in place of $n_k$) $\Gamma$-converges (with respect to the sequential weak $W^{1,p}(\Om)$-topology) to the functional $\mathcal F : W^{1,p}(\Om) \to [0,+\infty]$ given by \eqref{calF}. By the convergence of the minimum values, we infer that
$$\E_\e(u_n,\rho_n) =\min_{W^{1,p}(\Om)}\mathcal F_n \,\mathop{\longrightarrow}\limits_{n\to\infty} \,\min_{W^{1,p}(\Om)}\mathcal F =\E_\e(u,\rho)\,,$$
and the proposition is proved.
\end{proof}

\subsection{Existence of curve of unilateral maximal slope}

The main result of this section asserts that (under a mild assumption on $\partial\Om$) any unilateral minimizing movement is actually a curve of maximal unilateral slope for the Ambrosio-Tortorelli functional. For simplicity, we shall use again assumption \eqref{condGUAMM} on discrete trajectories. We refer to \cite{BM} for the general case.

\begin{theorem}\label{cms}
Assume that $\Om$ has a $\C^{1,1}$ boundary, and let  $(u_\e,\rho_\e) \in GUMM(u_0,\rho_0^\e)$. If $(u_\e,\rho_\e)$ is a strong $L^2(\Om)\times L^p(\Om)$-limit of some discrete trajectories $\{(u_k,\rho_k)\}_{k\in\NN}$ obtained from a sequence of partitions $\{\bdel_k\}_{k\in\NN}$ of $[0,+\infty)$ satisfying $|\bdel_k|\to 0$ and \eqref{condGUAMM}, then the mapping $(u_\e,\rho_\e):[0,+\infty)\to L^2(\Om)\times L^p(\Om)$ is a curve of maximal unilateral slope for $\E_\e$.
\end{theorem}

\begin{proof}
Let us define for each $k\in\NN$ and $t\geq 0$, $\lambda_k(t):=\E_\e(u_k(t),\rho_k(t))$. By Lemma \ref{inegenergdiscr} the function $\lambda_k:[0,+\infty)\to [0,+\infty)$ is non-increasing and bounded  uniformly with respect to $k$. By Helly's Theorem for monotone functions we can find a (not relabeled) subsequence of $\{k_n\}$ such that 
$$\lambda_{k_n}(t)\mathop{\longrightarrow}\limits_{n\to\infty} \lambda(t)\quad \text{for every $t\geq 0$}\,, $$
for some non-increasing function $\lambda:[0,+\infty)\to [0,+\infty)$. Then we infer from Lemma \ref{convforte} and Proposition~\ref{minrhoeps} that 
$$\lambda(t)=\E_\e(u_\e(t),\rho_\e(t))\quad\text{for every $t\in [0,+\infty)\setminus \mathcal B_\e$}\,. $$
Repeating the proof of Proposition \ref{energyineq} yields for any $0 \leq s \leq t$,
$$\lambda(t) + \int_{s}^{t}\|u'_\e(r)\|^2_{L^2(\Om)}\, dr \leq \lambda(s)\,.$$
According to Propositions  \ref{convdiv}, \ref{minrhoeps} and \ref{slope+domain}, we have $|\partial \E_\e|(u_\e(r),\rho_\e(r))=\|u'_\e(r)\|_{L^2(\Om)}$ for a.e. $r\geq 0$. Consequently,
if $t$ is a point of derivability of $\lambda$,
$$\lambda(t) -\lambda(s) \leq - \frac12 \int_{s}^{t}\|u'_\e(r)\|^2_{L^2(\Om)}\, dr- \frac12  \int_{s}^{t}|\partial \E_\e|^2(u_\e(r),\rho_\e(r))\, dr\, ,$$
and the conclusion follows by dividing the previous inequality by $t-s>0$ and sending $s \to t$.
\end{proof}

\begin{remark}\label{cms3}
It turns out that any curve $(u_\e,\rho_\e) :[0,+\infty) \to L^2(\Om) \times L^p(\Om)$ such that $u_\e \in L^\infty(0,+\infty;H^1(\Om)) \cap AC^2(0,+\infty;L^2(\Om))$,  $\rho_\e \in L^\infty(0,+\infty;W^{1,p}(\Om))$, $t \mapsto \rho_\e(t)$ non-increasing, and  satisfying \eqref{equa1}-\eqref{equa2}-\eqref{equa3} is a curve of unilateral maximal slope.
\end{remark}

\begin{remark}[\bf Non-uniqueness for the system of PDE's]
Let us consider a given curve of unilateral maximal slope $(u_\e,\rho_\e)$. We are going to construct a different solution  $(\tilde u_\e,\tilde \rho_\e)$ from $(u_\e,\rho_\e)$ of system \eqref{equa1}-\eqref{equa2}. 
To this aim, let us assume without loss of generality that $t=1$ is a point of continuity of $t \mapsto \E_\e(u_\e(t),\rho_\e(t))$, and that $\rho_\e(1)\not\equiv 0$. Choose $\hat \rho_1^\e\in W^{1,p}(\Omega)$ such that $0\leq \hat \rho_1^\e\leq \rho_\e(1)$ and different from $\rho_\e(1)$. 
Denote by $\rho_1^\e$ the unique solution of
$$\min\left\{ \E_\e(u_\e(1),\rho) : \rho \in W^{1,p}(\Om), \quad\rho \leq \hat \rho_1^\e \text{ in }Ê\Om\right\},$$
and let $(v_\e,\sigma_\e) \in GUMM(u_\e(1),\rho_1^\e)$. Considering the new curve $(\tilde u_\e,\tilde \rho_\e)$ defined by  
$$(\tilde u_\e(t),\tilde \rho_\e(t))=
\begin{cases}
(u_\e(t),\rho_\e(t)) & \text{ if }  0 \leq t <1\,,\\
(v_\e(t-1),\sigma_\e(t-1)) & \text{ if }Êt \geq 1\,,
\end{cases}
$$
we have
$$\tilde u_\e\in AC^2([0,+\infty);L^2(\Om))\cap L^\infty(0,+\infty;H^1(\Om))\, ,$$
$$\tilde \rho_\e\in L^\infty(0,+\infty;W^{1,p}(\Om))\,,\;\text{  $0\leq \tilde \rho_\e(t)\leq \tilde \rho_\e(s)\leq 1$ for every $t\geq s\geq 0$}\,,$$
and $(\tilde u_\e,\tilde \rho_\e)$ solves the system \eqref{equa1}-\eqref{equa2}. However, one can check that energy inequality \eqref{equa3} fails. In particular it is not a curve of maximal unilateral slope.
\end{remark}

\subsection{Relation with the unilateral slope of the Mumford-Shah functional}\label{connDalM}

  In \cite{DMT}, a notion of unilateral slope of the Mumford-Shah functional has been introduced. In that paper, the Mumford-Shah energy is slightly different from the one we consider here (see~\eqref{defMSfunct}). It is rather given on pairs $(u,K)$ by 
$$\mathcal{E}_*(u,K):=\frac{1}{2}\int_{\Om}  |\nabla u|^2\,dx+\mathscr{H}^{N-1}(K)+\frac{\beta}{2}\int_\Om (u-g)^2\,dx\,, $$
where $u\in SBV^2(\Om)$ and $K$ is a subset of $\Om$ 
satisfying $\mathscr{H}^{N-1}(K)<\infty$ and $J_u\;\widetilde\subset\; K$. The related unilateral slope of $\E_*$ is then given by  
$$|\partial \E_*|(u,K):= \limsup_{v \to u \text{ in }L^2(\Om)}Ê\frac{(\E_* (u,K) - \E_*(v,K \cup J_v))^+}{\|v-u\|_{L^2(\Om)}}\,.$$
In \cite{DMT}, the authors proved that if $|\partial \E_*|(u,K)<\infty$, then ${\rm div} (\nabla u) \in L^2(\Om)$, and that a weak form of
$$
\frac{\partial u}{\partial \nu}=0 \quad \text{on }K
$$
holds, where $\nu$ denotes a unit normal vector field on $K$. They also obtained the inequality $|\partial \E_*|(u,K) \geq \|Ê{\rm div} (\nabla u)-\beta(u-g)\|_{L^2(\Om)}$, and that equality holds if $u$ and $K$ are smooth enough. By means of an explicit counterexample, they have shown that $|\partial \E_*|$ is not lower semicontinuous for any reasonable  notion of convergence. In view of this result, they have introduced a notion of relaxed slope 
corresponding to a lower semicontinuous envelope of $|\partial\E_*|$ with respect to a suitable sequential topology. More precisely, the relaxed slope $|\overline{\partial \E}_*|$ is defined for a pair $(u,K)$ in the domain of $\E_*$ by 
$$|\overline{\partial \E_*}|(u,K):= \inf\big\{Ê\liminf_{n \to \infty}Ê|\partial \E_*|(u_n,K_n)\big\}\,,$$
where the infimum is taken over all sequences $\{(u_n,K_n)\}_{n\in\NN}$ such that $u_n \to u$ strongly in $L^2(\Om)$, $\nabla u_n \wto \nabla u$ weakly in $L^2(\Om;\RR^N)$, and $K_n$ $\sigma^2$-converges to $K$ (see \cite[Definition 4.1]{DMFT} for a precise definition). They established that if $|\overline{\partial \E}_*|(u,K)<\infty$, then there exists $f \in L^2(\Om)$ such that
\begin{equation}\label{conclDalMToad}
\begin{cases}
-{\rm div}(\nabla u)=f & \text{ in } L^2(\Om)\,,\\
|\nabla u|^2  -{\rm div}(u \nabla u) \leq fu& \text{ in } \D'(\Om)\,,\\
\nabla u \cdot \nu =0 & \text{ in } H^{-1/2}(\partial \Om)\,.
\end{cases}
\end{equation}
Again, there is an inequality $|\overline{\partial \E}_*|(u,K)\geq \|{\rm div}(\nabla u)-\beta(u-g)\|_{L^2(\Om)}$, and equality holds in some particular cases. Note that, in the case where $u$ and $K$ are smooth enough, the first line in \eqref{conclDalMToad} implies the continuity of $\frac{\partial u}{\partial \nu}$ across $K$, and the second one is then a weak reformulation of 
$$(u^+-u^-)\frac{\partial u}{\partial \nu} \geq 0 \quad \text{on $K$}\,, $$
where $u^\pm$ are the one-sided traces of $u$ on $K$ according to the orientation $\nu$. 

\vskip3pt

In our context, the analogy between the definitions of the unilateral slopes $|\partial \E_\e|$ and $|\partial \E_*|$ is quite clear, and it was actually one of the motivations to introduce $|\partial \E_\e|$. In view of the relation between the Ambrosio-Tortorelli functional and the Mumford-Shah functional in terms of $\Gamma$-convergence, a very interesting issue would be to find a precise relation between $|\overline{\partial \E}_*|$ and the asymptotic behavior as $\e\downarrow 0$ of  $|\partial \E_\e|$. Even if we do not pursue this issue here, we prove for completeness that a conclusion similar to \cite[Proposition~1.3]{DMT} holds for $|\partial \E_\e|$. For simplicity we only state the result in terms of the asymptotic limit obtained in Theorem~\ref{BabMil}. 

\begin{proposition}
Assume that $\Om$ has a $\C^{1,1}$ boundary. Let $u\in AC^2([0,+\infty);L^2(\Om))$ be the limiting curve obtained in Theorem  \ref{BabMil}. Then, for a.e. $t\geq 0$, we have
\begin{equation}\label{lwdasymptslop}
 \| {\rm div}(\nabla u(t)) -\beta (u(t)-g)\|_{L^2(\Om)}\leq \liminf_{n\to\infty}|\partial \E_{\e_n}|\big(u_{\e_n}(t),\rho_{\e_n}(t)\big)<\infty\,,
 \end{equation}
and
$$|\nabla u(t)|^2-{\rm div}\big(u(t)\nabla u(t)\big) \leq -u(t)\,{\rm div}\big(\nabla u(t)\big)\quad\text{in $\mathscr{D}'(\Om)$}\,.$$
\end{proposition}

\begin{proof}
From Proposition \ref{slope+domain} and \eqref{parabolic} together with \eqref{nrjineq} and Fatou's lemma, we first deduce that 
$$\int_0^{+\infty} \liminf_{n\to\infty} |\partial \E_{\e_n}|^2(u_{\e_n}(t),\rho_{\e_n}(t)\big)\,dt\leq \liminf_{n\to\infty} \int_0^{+\infty}  |\partial \E_{\e_n}|^2(u_{\e_n}(t),\rho_{\e_n}(t)\big)\,dt \leq C\,,  $$
for a constant $C>0$ independent of $n$. Hence there exists an $\mathscr{L}^1$-negligible set $\mathcal{L}\subset (0,+\infty)$ such that 
$$  \liminf_{n\to\infty} |\partial \E_{\e_n}|(u_{\e_n}(t),\rho_{\e_n}(t)\big)<\infty\quad \text{for $t\in (0,+\infty)\setminus\mathcal{L}$}\,.$$
Let us now fix $t\in (0,+\infty)\setminus\mathcal{L}$ and extract a subsequence (depending on $t$) such that
$$\lim_{j\to\infty} |\partial \E_{\e_{n_j}}|(u_{\e_{n_j}}(t),\rho_{\e_{n_j}}(t)\big)=  \liminf_{n\to\infty} |\partial \E_{\e_n}|(u_{\e_n}(t),\rho_{\e_n}(t)\big)\,.$$
By Proposition  \ref{slope+domain}, the sequence $\big\{{\rm div}\big((\eta_{\e_{n_j}}+\rho^2_{\e_{n_j}}(t))\nabla u_{\e_{n_j}}(t)\big)\big\}$ is thus bounded in $L^2(\Om)$, and in view of \eqref{weaknablaue} we deduce that 
\begin{equation}\label{convdivgradu*}
{\rm div}\big((\eta_{\e_{n_j}}+\rho^2_{\e_{n_j}}(t))\nabla u_{\e_{n_j}}(t)\big) \wto {\rm div}\big(\nabla u(t)\big) \quad\text{weakly in~$L^2(\Om)$}\,.
\end{equation}
Then \eqref{lwdasymptslop} follows from the convergences in \eqref{preconvueps}, and the lower semicontinuity of the $L^2(\Om)$-norm.

Using again Proposition  \ref{slope+domain}, we next notice that 
\begin{multline*}
\int_\Om \big(\eta_{\e_n}+\rho^2_{\e_n}(t)\big)|\nabla u_{\e_n}(t)|^2\varphi\,dx 
+ \int_\Om u_{\e_n}(t)\big(\eta_{\e_n}+\rho^2_{\e_n}(t)\big)\nabla u_{\e_n}(t)\cdot\nabla\varphi\,dx\\
=-\int_\Om u_{\e_n}(t)\,{\rm div}\big((\eta_{\e_n}+\rho^2_{\e_n}(t))\nabla u_{\e_n}(t)\big)\,\varphi\,dx
\end{multline*}
for any nonnegative function $\varphi\in \mathscr{D}(\Om)$, and the conclusion follows from \eqref{weaknablaue} and \eqref{convdivgradu*}. 
\end{proof}

\begin{appendices}

\section{Appendix}

The goal of this appendix is to prove the auxiliary elliptic regularity result used in the proof of Lemma \ref{inegenergdiscr} and Proposition \ref{slope+domain}. We would like to stress that the following result strongly relies on the fact that the dissipation energy has $p$-growth with $p>N$.

\begin{lemma}\label{ellipreg}
Assume that  $\Om$ has a $\C^{1,1}$-boundary. For  $f\in L^2(\Om)$ and $\rho\in W^{1,p}(\Om)$, let $u\in H^1(\Om)$ be a solution of
\begin{equation}\label{eqellipt}
\begin{cases}
-{\rm div}\big((\eta_\e+\rho^2)\nabla u\big)=f & \text{in $H^{-1}(\Om)$}\,,\\
(\eta_\e+\rho^2)\nabla u\cdot\nu =0 & \text{in $H^{-1/2}(\partial\Om)$}\,.
\end{cases}
\end{equation}
Then $u\in H^2(\Om)$, $\displaystyle\frac{\partial u}{\partial \nu}=0$ in $H^{1/2}(\partial\Om)$, and 
$$\|u\|_{H^2(\Om)}\leq C_\e(1+\|\nabla\rho\|_{L^p(\Om;\RR^N)})^{\gamma} \big(\|f\|_{L^2(\Om)}+\|u\|_{H^1(\Om)}\big)\,,$$
where $\gamma$ is the smallest integer larger than or equal to $p/(p-N)$, and $C_\e$ only depends on $\eta_\e$, $p$, $N$, and $\Om$.   
\end{lemma}

\begin{proof}
{\it Step 1.} We claim that 
\begin{equation}\label{H-1/2}
\frac{\partial u}{\partial\nu}=0 \quad\text{in $H^{-1/2}(\partial \Om)$}\,.
\end{equation}
To prove this claim, we first rewrite the equation as 
\begin{equation}\label{preestdeltav}
-\Delta u =\frac{2\rho}{\eta_\e+\rho^2}\nabla\rho\cdot \nabla u +\frac{f}{\eta_\e+\rho^2} \quad\text{in $\mathscr{D}'(\Om)$}\,. 
\end{equation}
Hence $\Delta u\in L^q(\Om)$ with $q:=2p/(p+2)$ by H\"older's inequality. Then we observe that $q':=q/(q-1)<2^*$ since $p>N$, so that 
$H^1(\Om)\hookrightarrow L^{q'}(\Om)$ by the Sobolev Imbedding. Hence the linear mapping 
$$\varphi\in H^1(\Om) \mapsto \int_\Om  \big(\nabla u \cdot \nabla  \varphi+(\Delta u) \varphi\big) \,dx $$
is well defined and continuous.  Consequently, $u$ admits a (weak) normal derivative $\frac{\partial u}{\partial\nu}$ on $\partial \Om$ which 
belongs to the dual space $H^{-1/2}(\partial \Om)$, and for any $\varphi\in H^1(\Om)$, 
\begin{align*}
\left\langle \frac{\partial u}{\partial\nu},\varphi_{|\partial\Om}\right\rangle_{\big(H^{-1/2}(\partial \Om),H^{1/2}(\partial \Om)\big)} &= \int_\Om  \big(\nabla u \cdot \nabla \varphi + (\Delta u) \varphi\big) \,dx \\
& = \int_\Om  (\eta_\e+\rho^2)\nabla u \cdot \nabla \left(\frac{\varphi}{\eta_\e+\rho^2}\right)\,dx +\int_\Om\left(\Delta u +\frac{2\rho}{\eta_\e+\rho^2}\nabla\rho\cdot \nabla u\right)\varphi \,dx\,.
\end{align*}
We observe that in the second equality above, we have used the fact that $\frac{\varphi}{\eta_\e+\rho^2}\in H^1(\Om)$ whenever $\varphi\in H^1(\Om)$. Indeed,
$$ \nabla \left(\frac{\varphi}{\eta_\e+\rho^2}\right)= \frac{\nabla \varphi}{\eta_\e+\rho^2} - \frac{2\rho\varphi \nabla \rho}{(\eta_\e+\rho^2)^2}\in L^2(\Om)\,,$$
since $\varphi\in L^{2^*}(\Om)$,  $\rho\in L^\infty(\Om)$, and $\nabla\rho\in L^p(\Om)$ with $p>N$. In view of \eqref{eqellipt}  we have 
$$\int_\Om  (\eta_\e+\rho^2)\nabla u \cdot \nabla \left(\frac{\varphi}{\eta_\e+\rho^2}\right)\,dx= \int_\Om \frac{f\varphi}{\eta_\e+\rho^2}\,dx \,,$$
and by  \eqref{preestdeltav}, 
$$\int_\Om\left(\Delta u +\frac{2\rho}{\eta_\e+\rho^2}\nabla\rho\cdot \nabla u\right)\varphi \,dx=-\int_\Om \frac{f\varphi}{\eta_\e+\rho^2}\,dx \,,$$
from which \eqref{H-1/2} follows. 

\vskip5pt

\noindent{\it Step 2.} We now prove that $u \in H^2(\Om)$. By the previous step, $u \in H^1(\Om)$ satisfies
$$
\left\{
\begin{array}{l}
\Delta u \in L^q(\Om)\,,\\
\ds \frac{\partial u}{\partial \nu}=0 \text{ in } H^{-1/2}(\partial \Om)\,.
\end{array}
\right.
$$
By elliptic regularity (see {\it e.g.} \cite[Proposition 2.5.2.3 \& Theorem 2.3.3.6]{Grisvard}), we deduce that $u \in W^{2,q_0}(\Om)$ with $q_0:=q=\frac{2p}{p+2}$, and 
$$\|u\|_{W^{2,q_0}(\Om)}\leq C\big(\|\Delta u\|_{L^{q_0}(\Om)}+\|u\|_{L^{q_0}(\Om)}\big) \,,$$
for some constant $C>0$ only depending on $N$, $p$, and $\Om$. Observing that the function $t\mapsto t/(\eta_\e+t^2)$ is bounded, we derive from \eqref{preestdeltav} and H\"older's inequality that
\begin{multline*}
\|u\|_{W^{2,q_0}(\Om)}\leq C_\e \big(\|f\|_{L^2(\Om)}+ \|\nabla\rho\|_{L^p(\Om;\RR^N)}\|\nabla u\|_{L^2(\Om;\RR^N)}+\|u\|_{L^2(\Om)}\big) \\
\leq C_\e(1+\|\nabla\rho\|_{L^p(\Om;\RR^N)})\big(\|f\|_{L^2(\Om)}+\|u\|_{H^1(\Om)}\big)\,,
\end{multline*}
where we used the fact that $q_0<2$. By the Sobolev Imbedding, we have $u \in W^{1,q^*_0}(\Om)$, and thus $\nabla u \cdot \nabla \rho \in L^{q_1}( \Om)$ with 
$$
\frac{1}{q_1}=\frac{1}{p} + \frac{1}{q_0^*}\,, \quad \textit{i.e., }\quad q_1:=\frac{2Np}{(N-2)p+4N}\,.
$$
Note that $q_1\geq 2$ if and only if $p \geq 2N$, so we have to distinguish the case $p\geq 2N$ from the case $p<2N$.
\vskip3pt

\noindent{\it Case 1).} Let us first assume that $p\geq 2N$. Then $\nabla u \cdot \nabla \rho \in L^{2}( \Om)$ with 
$$\|\nabla u \cdot \nabla \rho\|_{L^2(\Om)} \leq C\|\nabla u \cdot \nabla \rho\|_{L^{q_1}(\Om)} \leq C \|\nabla \rho\|_{L^p(\Om;\RR^N)}\|\nabla u\|_{L^{q_0^*}(\Om;\RR^N)} \leq C \|\nabla \rho\|_{L^p(\Om;\RR^N)}\|u\|_{W^{2,q_0}(\Om)}\,.$$
Using again \eqref{H-1/2}-\eqref{preestdeltav} and the elliptic regularity, we infer that  $u \in H^2(\Om)$  with the estimate
\begin{align*}
\|u\|_{H^2(\Om)}& \leq  C\big(\|\Delta u\|_{L^{2}(\Om)}+\|u\|_{L^{2}(\Om)}\big)  \\
& \leq  C_\e \big(\|f\|_{L^2(\Om)}+ \|\nabla \rho\|_{L^p(\Om;\RR^N)}\|u\|_{W^{2,q_0}(\Om)}+\|u\|_{L^2(\Om)}\big)\\
& \leq C_\e (1+\|\nabla \rho\|_{L^p(\Om;\RR^N)})^2\big(\|f\|_{L^2(\Om)}+\|u\|_{H^1(\Om)}\big)\,.
\end{align*}
\vskip3pt

\noindent{\it Case 2).} If $p<2N$ then $q_1<2$, and we have $u \in W^{2,q_1}(\Om)$ by \eqref{H-1/2}-\eqref{preestdeltav} and elliptic regularity, with the estimate
\begin{align}
\nonumber\|u\|_{W^{2,q_1}(\Om)}&\leq C_\e \big(\|f\|_{L^2(\Om)}+ \|\nabla\rho\|_{L^p(\Om)}\|u\|_{W^{2,q_0}(\Om)}+\|u\|_{L^2(\Om)}\big)\\
\label{estimq1} & \leq C_\e (1+\|\nabla \rho\|_{L^p(\Om)})^2\big(\|f\|_{L^2(\Om)}+\|u\|_{H^1(\Om)}\big)\,.
\end{align}
In particular, $\nabla u\in L^{q_1^*}(\Om)$ by the Sobolev Imbedding since $q_1<2\leq N$. We then continue the process by setting 
$$
\frac{1}{q_{i}}:=\frac{1}{p} + \frac{1}{q^*_{i-1}}, \quad \textit{i.e., }\quad q_i:=\frac{2Np}{(N-2i)p+2(i+1)N}
$$
as long as $q_{i-1}<2$, that is $i< \gamma$. Since $q_{\gamma-1}\geq 2$, iterating estimates of the form \eqref{estimq1} we obtain 
\begin{align*}
\|u\|_{H^{2}(\Om)}&\leq C_\e \big(\|f\|_{L^2(\Om)}+ \|\nabla\rho\|_{L^p(\Om;\RR^N)}\|u\|_{W^{2,q_{\gamma-2}}(\Om)}+\|u\|_{L^2(\Om)}\big)\\
& \leq C_\e (1+\|\nabla \rho\|_{L^p(\Om;\RR^N)})^{\gamma}\big(\|f\|_{L^2(\Om)}+\|u\|_{H^1(\Om)}\big)\,,
\end{align*}
and the proof is complete. 
\end{proof}

\end{appendices}

%\subsection*{Acknowledgements} 
%\vspace{-0.2cm}

\noindent {\bf {\it Acknowledgements.}} The authors wish to thank Fran\c{c}ois Murat for many helpful discussions  at different stages of the preparation of this work. They  are also indebted to Gilles Francfort for many suggestions which have led to a significant improvement of the paper.Ê Both authors have been supported by the {\sl Agence Nationale de la Recherche} under Grant No.\ ANR 10-JCJC 0106.

%\vspace{-0.4cm}

%%%%%%%%%%%%%%%%%%%%%%%%%%%%%%%%%%%%%%%%%%%%%%%%%%%%%%%%%

\end{document}